\documentclass[10pt,a4paper]{article}

\usepackage{geometry}
\usepackage[utf8]{inputenc}
\usepackage{todonotes}
\usepackage{amsmath}
\usepackage{amsthm}
\usepackage{amssymb}
\usepackage{algpseudocode}
\usepackage{float}
\usepackage[ruled,linesnumbered]{algorithm2e}
\usepackage{tikz}
\usetikzlibrary{automata}
\usetikzlibrary{arrows,shapes}

\usepackage[USenglish]{babel}

\usepackage{graphicx}

\usepackage{amssymb,amsmath,verbatim,tabularx}
\usepackage{colonequals,graphicx,psfrag}
\usepackage{enumerate} 
\usepackage{subfigure}

\newcommand{\N}{\mathbb{N}}

\newcommand{\R}{\mathbb{R}}

\newcommand\1{{\mathbf 1}}
\newcommand{\eps}{\varepsilon}

\newcommand{\conv}[1]{\convexh\left( #1\right)}

\DeclareMathOperator*{\argmax}{arg\,max} 
\DeclareMathOperator*{\convexh}{conv}

\newtheorem{theorem}{Theorem}
\newtheorem{lemma}[theorem]{Lemma}

\newtheorem{corollary}[theorem]{Corollary}

\usepackage{multirow}

\usepackage{authblk}

\begin{document}

\title{Data-Driven Robust Optimization\\using Unsupervised Deep Learning}

\author[1]{Marc Goerigk\footnote{marc.goerigk@uni-siegen.de, corresponding author, supported by the Deutsche Forschungsgemeinschaft (DFG) through grant GO 2069/1-1}}
\author[2]{Jannis Kurtz\footnote{jannis.kurtz@uni-siegen.de}}

\affil[1]{Network and Data Science Management, University of Siegen, Germany }	
\affil[2]{Management Science, University of Siegen, Germany }

\date{}

\maketitle

\begin{abstract}
Robust optimization has been established as a leading methodology to approach decision problems under uncertainty. To derive a robust optimization model, a central ingredient is to identify a suitable model for uncertainty, which is called the uncertainty set. An ongoing challenge in the recent literature is to derive uncertainty sets from given historical data that result in solutions that are robust regarding future scenarios. In this paper we use an unsupervised deep learning method to learn and extract hidden structures from data, leading to non-convex uncertainty sets and better robust solutions. We prove that most of the classical uncertainty classes are special cases of our derived sets and that optimizing over them is strongly NP-hard. Nevertheless, we show that the trained neural networks can be integrated into a robust optimization model by formulating the adversarial problem as a convex quadratic mixed-integer program. This allows us to derive robust solutions through an iterative scenario generation process. In our computational experiments, we compare this approach to a similar approach using kernel-based support vector clustering. We find that uncertainty sets derived by the unsupervised deep learning method find a better description of data and lead to robust solutions that  outperform the comparison method both with respect to objective value and feasibility. 
\end{abstract}

\noindent\textbf{Keywords:} Robust optimization; data-driven optimization; unsupervised deep learning

\section{Introduction}
In many real-world optimization problems, many of the observed parameters are uncertain, which can be due to measurement or rounding errors, or since the true value is first revealed in the future. Examples of uncertain parameters can be future demands, unknown traffic situations, noisy data and many more. Therefore, there is a high demand for optimization models which can handle the occurring uncertainties.
In contrast to stochastic programming, the robust optimization approach is used to tackle optimization problems with uncertain parameters by a worst-case approach and needs no information about the underlying probability distribution of the uncertain parameters. More precisely, in robust optimization problems we typically want to find a solution which is feasible under all future outcomes of the uncertain parameters and which is optimal in the worst-case. To this end, the input of robust optimization models is an uncertainty set containing the so-called scenarios. One of the main questions arising in practical applications is which structure and size the uncertainty set should have.

Initialized by the seminal work of Soyster in 1973 \cite{soyster1973}, the field of robust optimization fully emerged in the 1990s and was studied for several classes of uncertainty sets, including finite and convex uncertainty sets \cite{kouvelis,ben2009robust,aissi_minmax_survey}. For convex uncertainty sets, the most frequently used sub-classes are polyhedral, conic and ellipsoidal uncertainty sets \cite{roconv,bentalRobust,el1998robust,el1997robust}. As a special case of polyhedral uncertainty, the so-called budgeted uncertainty is a popular uncertainty class due to its simplicity and tractability \cite{BertsimasS04,bertsimas2003combinatorial}. Different uncertainty sets and their geometric relationship are studied in \cite{li2011comparative}. Recently mixed uncertainty sets, combining most of the popular uncertainty classes into one set, were studied in \cite{dokka2020mixed}. Next to the classical robust optimization approach, several less conservative approaches have been introduced. A survey about the classical and more recent robust optimization approaches for discrete and convex uncertainty can be found in \cite{buchheimrobust}.

As mentioned above, one of the main questions for a user applying a robust optimization model is how to choose the structure and the size of the uncertainty set. In most real-world situations the user only has a finite set of observations of the uncertain parameters from the past, possibly containing corrupted scenarios or outliers. Constructing a finite uncertainty set containing all observations can lead to overly pessimistic robust solutions due to outliers and since no structural properties of the data are exploited. Furthermore, the robust optimization approach with finite uncertainty sets is known to be hard to solve, even for easy classical combinatorial optimization problems with only two scenarios \cite{kouvelis}. To tackle the latter problems, several data-driven robust optimization approaches have been studied. Uncertainty sets can be constructed by statistical methods using non-parametric estimators to model confidence regions \cite{alexeenko2020nonparametric}, hypothesis tests \cite{bertsimas2018data}, $\phi$-divergence \cite{yanikouglu2013safe} or statistical learning theory \cite{tulabandhula2014robust}. In \cite{hong2017learning} the authors approximate a high probability region by combinations of classical uncertainty sets and then use a data-splitting scheme to determine the size of the region. In \cite{campbell2015bayesian,ning2017data} Dirichlet process mixture models are used to construct uncertainty sets which are given by a union of ellipsoids. Furthermore, in \cite{ning2017data} computationally tractable uncertainty sets with a polyhedral structure are constructed and applied to adaptive robust optimization problems. Polyhedral uncertainty sets incorporating correlations between uncertain parameters are derived in \cite{ning2018data} using principal component analysis and kernel smoothing. In the field of data-driven distributional robustness the idea is to construct a distributional ambiguity set around an empirical distribution, often defined by the Wasserstein metric or $\phi$-divergence, containing all probability distribution to be considered. The task is to find a solution which performs best under the worst-case distribution in the ambiguity set; see e.g. \cite{esfahani2018data,saif2021data,bertsimas2019adaptive,wiesemann2014distributionally,goh2010distributionally}. 
In \cite{ben2009robust,BertsimasS04,bertsimas2019probabilistic} probabilistic performance guarantees of classical and tractable uncertainty sets were studied, i.e. under the assumption that the uncertain parameters follow a probability distribution with certain properties, the probability of a robust solution being infeasible can be controlled. Interestingly, these performance guarantees even hold for artificially constructed probability distributions whose support is not contained in the uncertainty set.

While the goal of most of the latter approaches is the construction of uncertainty or ambiguity sets leading to good solutions with a probabilistic performance guarantee, another line of research is concerned with finding solutions which are feasible for all possible future scenarios, avoiding probabilistic assumptions on the data. On the one hand, deterministic approaches were developed to construct a range of uncertainty sets which were applied to the shortest path problem under real-world traffic scenarios of the City of Chicago \cite{chassein2019algorithms}. In \cite{cheramin2021data} principal component analysis is used to derive polyhedral uncertainty sets which are computationally cheaper than using the convex hull of the observed scenarios. On the other hand, unsupervised machine learning models were deployed to learn and extract information from data, detect anomalies and to construct uncertainty sets which better describe the structure of the data. Clustering and supervised machine-learning approaches were used in \cite{garuba2020comparison}. In \cite{shang2017data,shang2019data,shen2020large} the authors derive a kernel-based support vector clustering model to construct polyhedral uncertainty sets.

In this work we apply and adapt the unsupervised deep classification model developed in \cite{ruff2018deep} to construct uncertainty sets which are given by level sets of a norm-function applied to the output of deep neural networks. Deep learning is widely used to discover and exploit unknown structure in data \cite{lecun2015deep,bengio2012deep}. Due to the high expressivity of neural networks, the derived sets have a more complex structure compared to other methods, which can flexibly be adjusted by varying the architecture of the underlying neural network. Furthermore, the one-class deep learning method in \cite{ruff2018deep} succeeded in detecting anomalies in data. The basic idea of our approach is to train a neural network which can distinguish a true scenario from a corrupted or unrealistic scenario with high accuracy. The constructed uncertainty set then contains all scenarios which are classified by the neural network as true scenarios.  By scaling the radius of our derived set, we can additionally control the conservativeness of the robust solution. 

While this work was under review, our method was applied to a real-world refinery planning problem under uncertainty in \cite{wang2021deep}. The results reported in that paper give further evidence that our method performs well in practice.


\textit{Our contributions are as follows:}
\begin{itemize}
\item We apply the unsupervised deep classification model developed in \cite{ruff2018deep} to create uncertainty sets for robust optimization problems.

\item We show that the constructed sets are given by a finite union of convex subsets, each of them having a polyhedral structure intersected with a transformed norm inequality.

\item We prove that most of the classical uncertainty classes considered in the literature are special cases of our derived sets and that optimizing over our set is strongly NP-hard.

\item We show that it is possible to optimize over our sets using a mixed-integer programming formulation, which means that the robust optimization problem can be solved by an iterative scenario generation procedure.

\item We test our method on randomly generated data and on realistic traffic data and compare it to the kernel-based support vector clustering sets from \cite{shang2017data} and to the convex hull of the historical data points. Our experiments show that solutions calculated by our method often outperform both other methods.
\end{itemize}

\section{Preliminaries}\label{sec:preliminaries}
\subsection{Notation}
We define $[k]:=\left\{ 1,\ldots ,k\right\}$ for each $k\in \N$ and $\R_+^N:=\left\{ x\in\R^N: x\ge 0\right\}$. The $\ell_p$-norm of a vector $x\in\R^N$ is defined by $\|x\|_p:=\left(\sum_{i\in[N]}x_i^p\right)^{\frac{1}{p}}$. For a matrix $A\in\R^{m\times n}$ we denote by $a_i$ the $i$-th row and by $A_j$ the $j$-th column of $A$. For a vector $v\in\R^N$ we denote by $\text{diag}(v)$ the $N\times N$ matrix which has diagonal entries $v$ and zero-entries otherwise. For given matrices $W^1,\ldots ,W^L$ of appropriate dimensions we define the product
\[
\prod_{i=1}^{L} W^i := W^LW^{L-1}\cdots W^1 .
\]
The identity matrix in dimension $n$ is denoted by $E_n$.

\subsection{Robust Optimization}
Consider the deterministic linear optimization problem
\begin{equation}\label{eq:deterministicProblem}
\begin{aligned}
\min \ & d^\top x \\ 
s.t. \quad & c^\top x \le b \\
& x \in X 
\end{aligned}
\end{equation}
where $d\in\R^N$ is a given cost vector, $X = \{ x \in \R^N : Tx \le e \}$ a polyhedron and $c\in \R^N$, $b \in \R$.
Assume that the coefficient parameters of the constraint $c^\top x \le b$ are uncertain, i.e. the vector $c$ is not known precisely. In robust optimization we assume that an uncertainty set $U\subset\R^N$ is given, which contains all possible realizations of the vector $c$. The aim is then to calculate an optimal solution which is feasible for each realization in $U$, i.e. we want to solve the problem
\begin{equation}\label{eq:robustProblem}
\begin{aligned}
\min \ & d^\top x \\ 
s.t. \quad & c^\top x\le b \quad \forall c\in U\\
& x\in X
\end{aligned}
\end{equation}
which is equivalent to the problem
\begin{equation}\label{eq:robustProblem2}
\begin{aligned}
\min \ & d^\top x \\ 
s.t. \quad & \max_{c\in U} c^\top x\le b\\
& x\in X.
\end{aligned}
\end{equation}
If more than one constraint is uncertain, a similar reformulation can be applied for each such constraint.
For certain classes of convex uncertainty sets $U$, replacing $\max_{c\in U} c^\top x$ by its dual formulation, Problem \eqref{eq:robustProblem2} can be transformed to a deterministic problem of a certain class \cite{ben2009robust}. If $U$ is a polyhedral uncertainty set, \eqref{eq:robustProblem2} is equivalent to a linear program, while for ellipsoidal uncertainty sets it becomes a second-order cone problem.

Alternatively, Problem \eqref{eq:robustProblem} can be solved by iteratively generating new worst-case scenarios in $U$ and adding them to the problem. More precisely, we alternately calculate an optimal solution $x^*\in \{ x\in X : \max_{c \in U'} c^\top x \le b \}$ of Problem~\eqref{eq:robustProblem} for a finite subset $U'\subset U$ and afterwards a worst-case scenario 
\begin{equation}\label{eq:adversarialProblem}
c^* \in \argmax_{c\in U} c^\top x^*
\end{equation}
which is then added to $U'$ if $(c^*)^\top x^* > b$ holds. Otherwise, we stop with an optimal solution $x^*$.

Sometimes the special case of the robust optimization problem is studied where the uncertain parameters only appear in the objective function, which can be modeled by
\begin{equation}\label{eq:robustProblemObjectiveUncertainty}
\min_{x\in X} \max_{c\in U} c^\top x . 
\end{equation}
Nevertheless, all results from above also hold for Problem \eqref{eq:robustProblemObjectiveUncertainty}.
Finally, note that we can equivalently replace $U$ by its convex hull $\conv{U}$ in Problems~\eqref{eq:robustProblem} and \eqref{eq:robustProblemObjectiveUncertainty}, since we are optimizing a linear function over $U$. 


\subsection{Unsupervised Learning Methods}
\label{sec:unsuplearn}

In this work we consider a subclass of unsupervised learning models, sometimes called anomaly detection or one-class classification. Here the task is to decide whether a given data point is a normal data point or if it is anomalous. To this end the aim is to train an appropriate model on a set of unlabeled training data, which is assumed to contain the normal data points, and extract structural information from this training set to distinguish normal data from anomalous data in the future. Often the idea behind one-class classification models is to find a minimal norm-ball in a certain feature space, such that all anomalous data points are lying outside of the ball. In the following we summarize two approaches, the first based on support vector clustering and the second using deep neural networks.

Given a set of data points $c^1,\ldots ,c^m\in \R^N$ and a mapping $\phi: \R^N\to \R^K$ to a possibly high-dimensional feature space, the idea of soft margin support vector clustering (SVC) is to find the smallest sphere that encloses most of the training data, which can be done by solving the problem
\begin{equation}\label{eq:SVC}
\begin{aligned}
\min \ & R^2 + \frac{1}{N\nu}\sum_{i=1}^{m}\xi_i \\
s.t. \quad & \| \phi(c^i)-\bar c\|_2^2\le R^2 + \xi_i \quad i=1,\ldots ,m \\
& \bar c\in\R^K, \ R\ge 0, \ \xi\in \R_+^m .
\end{aligned}
\end{equation}
Note that the mapping $\phi(c^i)$ of a data point can lie outside of the sphere around $\bar c$ with radius $R$ in which case $\xi_i>0$ in an optimal solution. The distance $\xi_i$ of each point outside of the sphere is penalized in the objective function and the parameter $\nu\in [0,1]$ can be used to adjust the fraction of data points which will lie outside of the optimal sphere. Problem \eqref{eq:SVC} is a convex problem and by applying the KKT conditions we obtain its dual problem
\begin{equation}\label{eq:SVCdual}
\begin{aligned}
\min_{\alpha} \ & \sum_{i,j=1}^{m}\alpha_i\alpha_jK(c^i,c^j)-\sum_{i=1}^{m}\alpha_iK(c^i,c^i) \\
s.t. \quad & 0\le \alpha_i\le \frac{1}{N\nu} \quad i\in [m] \\
& \sum_{i\in[m]}\alpha_i = 1 ,
\end{aligned}
\end{equation}
where $K(\cdot,\cdot )$ is a kernel function and $K(c^i,c^j)=\phi(c^i)^\top\phi(c^j)$. If $K$ is a positive definite kernel, Problem \eqref{eq:SVCdual} is a convex quadratic problem and can be solved by classical QP methods \cite{boyd2004convex} or due to its specific structure by the sequential minimal optimization procedure \cite{zeng2008fast}.

The main idea in \cite{shang2017data} is to construct uncertainty sets 
\[
U = \left\{ c\in \R^N: \|\phi(c)-\bar c\|_2^2\le R^2\right\} ,
\]
where $\bar c$ and $R$ are given by an optimal solution of \eqref{eq:SVC}. The authors apply the kernel
\[
K(u,v):=\sum_{i\in[m]}l_i - \|Q(u-v)\|_1
\]
where $Q$ is a weighting matrix containing covariance information from data and the $l_i$ are specified values which were chosen such that the kernel $K$ is positive definite. They derive the uncertainty set
\begin{equation}\label{eq:uncertaintysetSVC}
\begin{aligned}
 U_\nu=\{ c\in\R^N: & \ \exists v_i\in\R^N \quad \forall i\in \text{SV} \\ & \sum_{i\in\text{SV}}\alpha_i v_i^\top \1 \le \min_{i'\in \text{BSV}}\sum_{i\in\text{SV}}\alpha_i \|Q(c^{i'}-c^i) \|_1 \\ 
 & -v_i \le Q(c-c^i)\le v_i \quad \forall i\in \text{SV}\}   
 \end{aligned}
\end{equation}
where $\alpha$ is the optimal solution of Problem \eqref{eq:SVCdual}, $\text{SV}:=\{ i: \alpha_i>0\}$ is the set of support vectors and $\text{BSV}:=\{ i: \frac{1}{N\nu}>\alpha_i>0\}$ is the set of boundary support vectors. Geometrically, SV contains all indices of the data points which lie outside or on the boundary of the sphere, while BSV contains only the data points on the boundary. Note that the number of variables in the description of $U_\nu$ depends on the number of support vectors and therefore grows with increasing sample size $m$ and decreasing $\nu$.

In \cite{ruff2018deep} the authors study Problem \eqref{eq:SVC} where $\phi(c)$ is replaced by the output of a neural network for data point $c$. A neural network is a function $f_{W^1,\ldots ,W^L}:\R^N\to \R^{d_L}$ which maps a data point $c\in\R^N$ to an output vector \[\mathop{f_{W^1,\ldots ,W^L}(c)=\sigma^L\left( W^L\sigma^{L-1}\left( W^{L-1}\ldots \sigma^1\left(W^1c\right)\ldots \right)\right)}\] where $W^l\in \R^{d_{l}\times d_{l-1}}$ are the weight matrices, and $\sigma^l:\R\to \R$ for each $l\in[L]$ are activation functions which are applied component-wise. The dimension $d_{l}$ is called the \textit{width} of the \textit{$l$-th layer} with $d_0 := N$. We define $W:=(W^1 ,\ldots ,W^L)$. The authors present a model called \textit{One-Class Deep Support Vector Data Description} (Deep SVDD), which is given by the problem 
\begin{equation}\label{eq:DeepSVDD}
\min_{W^1,\ldots ,W^L} \ \frac{1}{m}\sum_{i\in [m]} \|f_{W^1,\ldots ,W^L}(c^i)-\bar c\|_2^2 + \frac{\lambda}{2}\sum_{l\in[L]}\| W^l\|_F^2 
\end{equation}
where $\bar c$ is a given center point, $\lambda\ge 0$ a given control parameter and $\|\cdot\|_F$ denotes the Frobenius norm. The term $\frac{\lambda}{2}\sum_{l\in[L]}\| W^l\|_F^2$ is a regularizer which leads to smaller weights after training and which can be controlled by the parameter $\lambda$. After optimizing Problem \eqref{eq:DeepSVDD} the radius $R^2>0$ to control the size of the sphere. A new data point $c\in\R^N$ is then classified as a normal scenario if $\|f_{W^1,\ldots ,W^L}(c)-\bar c\|_2\le R$. We will use this model in the next section to create uncertainty sets which can be used for robust optimization problems. 

Note that the center point $\bar c$ is fixed and not part of the decision variables. This is because in the latter, case a trivial optimal solution would be to set all network weights and the center to $0$, provided that the activation functions do not have a bias term, and then all points in $\R^N$ would be mapped to the center. In \cite{ruff2018deep} the authors mention that a good strategy is to set the center point $\bar c$  to the average of the network representations that result from performing an initial forwardpass on some training data sample. Another drawback is that bias terms in the network weights or in the activation functions can lead to useless solutions, since e.g. setting all weights in the first layer matrix $W^1$ to $0$ and the bias vector of the first layer to $\bar c$ yields a network which maps all points in $\R^N$ to the center and hence does not extract any information from the data. To tackle this problem in \cite{chong2020simple} two regularizers are presented, one based on injecting random noise via the standard cross-entropy loss, and one which penalizes the minibatch variance when it becomes too small. By adding one of these regularizers to the loss function in \eqref{eq:DeepSVDD} the mode collapse problem is avoided and bias terms can be used.

\section{Creating Uncertainty Sets via Unsupervised Deep Learning}\label{sec:theoreticalSection}

\subsection{Definition and Properties of Uncertainty Sets}

In this section we derive non-convex uncertainty sets using level sets of trained neural networks, which are given as described in Section~\ref{sec:unsuplearn}. We assume that all activation functions are continuous piecewise affine functions, i.e. they are of the form
\begin{equation}\label{eq:activationFunction}
    \sigma^l(w):= \alpha_i^l w + \gamma_i^l \ \text{ if } \ \underline{\beta}^{i,l}\le w\le\overline{\beta}^{i,l} \ \forall i\in [k_l]
\end{equation}
for all $l\in [L]$, where $\alpha_i^l\in\R$ are the given slopes, $\gamma_i^l$ are the given $y$-intercepts,  $k_l\in\N$ is the number of intervals and $\underline{\beta}^{i,l}< \overline{\beta}^{i,l}$ are the bounds of the intervals where $\overline{\beta}^{i,l}=\underline{\beta}^{i+1,l}$, $\underline{\beta}^{1,l}=-\infty$ and $\overline{\beta}^{k_l,l}=\infty$ for all $l\in[L]$. 
If all possible data points are contained in a bounded set and if the neural network is already trained, then we can replace $\infty$ by a large enough value $M>0$. Note that the ReLU activation function $\sigma(w)=\max\{ w,0\}$ can be modeled by \eqref{eq:activationFunction} setting $k_l=2$, $\alpha_1^l = 0$, $\alpha_2^l=1$, $\gamma_1^l=\gamma_2^l = 0$, $\underline{\beta}^{1,l}=-\infty$, $\overline{\beta}^{1,l}=0$, $\underline{\beta}^{2,l}=0$ and $\overline{\beta}^{2,l}=\infty$. Other piecewise affine activation functions as the Hardtanh or the hard sigmoid function can be modeled by \eqref{eq:activationFunction} as well. In general, any continuous function could be approximated by these piecewise affine functions.

For a neural network, trained on the training sample $c^1,\ldots ,c^m\in\R^N$, and given by its weight matrices $W=(W^1,\ldots , W^L)$, a center point $\bar c\in\R^{d_L}$ and a radius $R>0$, we define the uncertainty set
\begin{equation}\label{eq:NNUncertaintySet}
    U_f(W,\bar c,R):=\{c\in\R^N \ : \ \|f_{W^1,\ldots ,W^L}(c)-\bar c\|\le R\}
\end{equation}
for an arbitrary norm $\|\cdot\|$. Using the Deep SVDD method defined in \eqref{eq:DeepSVDD} to train the neural network, the natural choice of $U_f(W,\bar c,R)$ would be given by using the Euclidean norm and the center point $\bar c$ returned by the model.
All of the following results hold if we allow bias terms in the network architecture. In this case each layer is given of the form $W^iy + b^i$ where $b^i$ is the trained bias term of layer $i$.

In the following we define
\begin{equation}
    \mathcal A :=\left\{ u:=(u^{i,l})_{l\in[L], i \in [k_l]} : u^{i,l}\in\{0,1\}^{d_l}, \sum_{i\in [k_l]} u^{i,l}=\1\right\} .
\end{equation}
The idea is that the vectors $u$ encode the activation decisions of all neurons, i.e. $u_j^{i,l}=1$ if the outcome of the $j$-th component of layer $l$ lies in the interval $[\underline{\beta}^{i,l},\overline{\beta}^{i,l})$ and $0$ otherwise. The constraint $\sum_{i\in [k_l]} u^{i,l}=\1$ ensures that for each neuron exactly one interval is chosen. We call a vector $u\in \mathcal A$ an \textit{activation pattern}.

We can now prove the following theorem which shows that the uncertainty set $U_f(W,\bar c,R)$ is given by a finite union of convex sets, where each convex set is an intersection of a polyhedron with a norm constraint. To this end, we define $\tilde W^1=W^1$, $W^{L+1}=E_{d_L}$, where $E_{d_L}$ is the identity matrix in dimension $d_L$, and
\[\tilde W^l=\left(\prod_{s=1}^{l-1}W^{s+1}\text{diag}(\sum_{i\in [k_{s}]} u^{i,s}\alpha_i^{s})\right)W^1\] for each $l=2,\ldots L+1$. Furthermore, we set $\tilde \gamma^1=0$ and 
\[\tilde \gamma^l=\sum_{s=2}^{l}W^l\left(\prod_{t=s}^{l-1}\text{diag}(\sum_{i\in [k_{t}]} u^{i,t}\alpha_i^{t})W^t\right)\left(\sum_{i\in [k_{s-1}]} u^{i,s-1}\gamma_i^{s-1}\right)\]
for all $l=2,\ldots L+1$.
\begin{theorem}\label{thm:unionOfConvexSets}
Given the weight matrices $W$ of a trained neural network, a center point $\bar c\in\R^{d_L}$ and a radius $R>0$, then it holds that
\begin{equation}
    U_f(W,\bar c,R) = \bigcup_{u\in \mathcal A} P(W,u,\bar c, R)
\end{equation}
with
\begin{align*}
        P(W,u,\bar c, R):=\{ c\in\R^N \ : \ & \tilde W^l c + \tilde \gamma^l < \sum_{i\in [k_l]} u^{i,l}\overline{\beta}^{i,l} \ \forall l\in [L], \\
        & \tilde W^l c + \tilde \gamma^l\ge \sum_{i\in [k_l]} u^{i,l}\underline{\beta}^{i,l} \ \forall l\in [L], \\
        & \|\tilde W^{L+1}c+\tilde \gamma^{L+1}-\bar c\|\le R \}.
\end{align*}
\end{theorem}
\begin{proof}
It is a well-known fact that neural networks with classical piecewise affine activation functions cluster the data space into polyhedra and apply a different affine function on each of them \cite{rister2017piecewise,montufar2014number,wang2018max,raghu2017expressive}. Nevertheless, we will  derive our result explicitly for the more general piecewise affine activation functions in this proof.

Note that for each data point $c\in\R^N$ applied to the neural network, there exists exactly one activation pattern $u\in \mathcal A$. Consider a fixed activation pattern $u\in \mathcal A$, and for a data point $c\in U_f(W,\bar c,R)$ let $\tilde w^l(c)\in\R^{d_l}$ be the output of layer $l\in [L]$ before applying the activation function. The data point $c$ has activation pattern $u$ if the conditions
\begin{equation}\label{eq:constr_neuron}
\sum_{i\in [k_l]} u^{i,l}\underline{\beta}^{i,l}\le \tilde w^l(c) < \sum_{i\in [k_l]} u^{i,l}\overline{\beta}^{i,l}
\end{equation}
are true for each $l\in [L]$. The output of layer $1$ is given by $W^1c$. Applying the activation function componentwise and afterwards applying $W^2$, the output of layer $2$ is given by
\begin{align*}
\tilde w^2 & = W^2 \left(\text{diag}(\sum_{i\in [k_{1}]} u^{i,1}\alpha_i^{1})W^1c + (\sum_{i\in [k_{1}]} u^{i,1}\gamma_i^{1})\right)\\
&= W^2 \text{diag}(\sum_{i\in [k_{1}]} u^{i,1}\alpha_i^{1})W^1c + W^2(\sum_{i\in [k_{1}]} u^{i,1}\gamma_i^{1}) \\
& = \tilde W^2 c + \tilde \gamma ^2.
\end{align*}
Inductively we conclude that the output of layer $l\in [L]$ is $\tilde W^lc + \tilde \gamma^l$ and together with \eqref{eq:constr_neuron} we obtain the linear inequalities in $P(W,u,\bar c,R)$. Since the activation function is applied to the output of the last layer, applying the identity matrix $W^{L+1}$ afterwards we obtain the output of the neural network by
\[
f_{W^1,\ldots ,W^L}(c)=\tilde W^{L+1}c + \tilde \gamma^{L+1} .
\]
Since for any data point in $U_f(W,\bar c,R)$ the condition $\|f_{W^1,\ldots ,W^L}(c)-\bar c\|\le R$ must hold, we obtain the last constraint in the set $P(W,u,\bar c,R)$ which proves the result.
\end{proof}
Note that the set $P(W,u,\bar c,R)$ can be unbounded, since if there exists a point $\hat c\in P(W,u,\bar c,R)$, the polyhedron corresponding to the linear constraints has an infinite ray $r\in\R^N$, and if $r$ is in the kernel of the matrix $\tilde W^{L+1}$, then each point $\hat c + \lambda r$ for $\lambda\in[0,\infty)$ is contained in $P(W,u,\bar c,R)$. 

Estimating the size of $\mathcal{A}$, we find the following result.
\begin{corollary}\label{cor:numberOfSets}
The uncertainty set $ U_f(W,\bar c,R)$ is the union of at most $(d(k-1))^{NL}$ convex sets $P(W,u,\bar c, R)$, where $k=\max_{l\in [L]}k_l$ and $d:=\max_{l\in [L]}d_l$.
\end{corollary}
\begin{proof}
The idea of the proof is to bound the number of possible regions which arise by the linear inequalities in the definition of $P(W,u,\bar c,R)$. It was proved in \cite{raghu2017expressive} that given $t$ hyperplanes in $\R^N$, the number of regions (i.e. the number of connected open sets bounded on some sides by hyperplanes) is bounded from above by $t^N$. Considering the linear inequalities in $P(W,u,\bar c,R)$, in the first layer we have at most $d(k-1)$ hyperplanes describing the feasible region, since we have at most $d$ normal vectors and for each at most $k-1$ intervals given by the right hand sides. Therefore, the number of regions is bounded by $(d(k-1))^N$. Considering one of the regions, it has fixed activation pattern for the first layer and is then again divided into at most $(d(k-1))^N$ regions by the hyperplanes of the second layer. Hence, after the second layer we have at most $(d(k-1))^{2N}$ possible regions. Inductively we conclude that the number of possible regions given by the constraints in $P(W,u,\bar c,R)$ is bounded by $(d(k-1))^{NL}$.
\end{proof}
Note that if every layer has ReLU activation (which is continuous) and we choose the Euclidean norm, then
\begin{align*}
        P(W,u,\bar c, R):=\{ c\in\R^N \ | \ & \tilde w_j^l c \le  0 \text{ if } u_j^{l,1} =1 \text{ and } u_j^{l,2}=0 \ \forall j\in [d_l], l\in [L], \\
        & \tilde w_j^l c \ge  0 \text{ if } u_j^{l,1}=0 \text{ and } u_j^{l,2} = 1\ \forall j\in [d_l], l\in [L], \\
        & c^\top (\tilde W^{L+1})^\top  \tilde W^{L+1} c - 2 c^\top \tilde W^{L+1}\bar c + \|\bar c\|^2\le R^2 \} .
\end{align*}
Therefore, applying Theorem~\ref{thm:unionOfConvexSets} and Corollary~\ref{cor:numberOfSets} to the ReLU case, $U_f(W,\bar c,R)$ is the union of $\mathcal O(d^{NL})$ polyhedral cones intersected with the level set of one convex quadratic function. 

A direct consequence of Theorem \ref{thm:unionOfConvexSets} is that in general, the uncertainty set $U_f(W,\bar c,R)$ can be non-convex and even non-connected. Note that this property is irrelevant when we consider linear problems, since in this case, it is equivalent to replace the uncertainty set by its convex hull. This does not hold for multi-stage problems, where our method to construct uncertainty sets can be applied as well.

Another useful feature is that the size of the uncertainty set $U_f(W,\bar c,R)$ and hence the conservativeness of the corresponding robust optimization model can be controlled by the radius $R$. 
The radius $R$ can be chosen as the $(1-\eps)$-quantile of the radii $r_i=\|f_{W^1,\ldots ,W^L}(c^i)-\bar c\|_2$ of the training data, which results in a corresponding probability guarantee over the empirical distribution. However, as the resulting solution may be significantly more robust than the a priori guarantee suggests, a more suitable approach is to carry out experiments over a validation set to determine $R$, i.e., to choose a value that leads to the best tradeoff between feasibility and objective value.

 Moreover due to the non-convexity of $U_f(W,\bar c,R)$ its robust counterpart \eqref{eq:robustProblem} cannot be reformulated using classical duality results. Nevertheless, it is possible to optimize over $U_f(W,\bar c,R)$ in direction $x\in\R^N$ by solving for each $u\in \mathcal A$ the problem
\[
\max_{c\in P(W,u,\bar c,R)} x^\top c
\]
or deciding that it is infeasible. Then the best solution over all sets $P(W,u,\bar c,R)$ is the optimal solution. Alternatively it is possible to optimize over $U_f(W,\bar c,R)$ by solving a convex quadratic mixed-integer program. A similar idea was already used to model trained neural networks with ReLU activation in \cite{fischetti2018deep} and to train binarized neural networks in \cite{icarte2019training,bahBDNN2020}.
\begin{theorem}\label{thm:IPFormulation}
For a given solution $x\in X$ and a continuous activation function, the problem \[ \max_{c\in U_f(W,\bar c,R)} c^\top x\] is equivalent to the problem
\begin{align}
    \max \ & x^\top c^1 \label{adv-s} \\
    s.t. \quad & W^lc^l< \sum_{i\in [k_l]} u^{i,l}\overline{\beta}^{i,l} \quad \forall l\in [L] \label{constr:lower}\\
    & W^lc^l\ge \sum_{i\in [k_l]} u^{i,l}\underline{\beta}^{i,l} \quad \forall l\in [L] \label{constr:upper}\\
    & c^{l+1} = \text{diag}(\sum_{i\in [k_l]} u^{i,l}\alpha_i^l)W^lc^l + \sum_{i\in [k_l]} u^{i,l}\gamma_i^l \quad \forall l\in [L] \label{constr:outputLayer}\\
    & \sum_{i\in[k_l]}u^{i,l}=\1 \quad \forall l\in{L} \\
    & \|c^{L+1}-\bar c\|\le R\\
   & u^{i,l}\in \{0,1\}^{d_l} \quad \forall i\in [k_l], \ l\in [L]\\
   & c^{l}\in \R^{d_{l-1}} \quad \forall l\in[L+1].  \label{adv-e}
\end{align}
\end{theorem}
\begin{proof}
The set of feasible variable assignments of the variables $u^{i,l}$ are exactly all possible activation patterns in $\mathcal A$. The constraints \eqref{constr:lower}-\eqref{constr:upper} ensure that the activation pattern given by the $u$-variables is true for the data point $c^1\in\R^N$. Furthermore, the variables $c^{l+1}$ model the output of layer $l$ after applying the activation function componentwise. Following the proof of Theorem~\ref{thm:unionOfConvexSets}, the set of feasible solutions $c^1\in\R^N$ of the formulation is equal to $U_f(W,\bar c,R)$. 
\end{proof}
Note that the quadratic terms in Constraint~\eqref{constr:outputLayer} can be linearized by standard linearization techniques. Therefore, the problem in Theorem~\ref{thm:IPFormulation} is equivalent to a mixed-integer program with a convex quadratic constraint
if we use the Euclidean norm, or a mixed-integer linear program if we use the $\ell_1$-norm. Hence, the adversarial Problem \eqref{eq:adversarialProblem} can be solved by this formulation using classical integer programming solvers as CPLEX or Gurobi. Using the iterative constraint generation procedure (see Section~\ref{sec:preliminaries}), we can then solve the robust problem~\eqref{eq:robustProblem}.

We conclude this section by describing an alternative solution method to simplify the adversarial problem~(\ref{adv-s}-\ref{adv-e}) by avoiding the use of binary variables $u^{i,l}$. To this end, we apply all scenarios $c^1,\ldots ,c^m$ from the training data to the neural network to calculate the activation pattern for each of them, i.e. calculating the values of the $u$-variables for each data point $c^j$ which we denote by $u^{i,l}(c^j)$ for all $i\in [k_l],l\in[L],j\in [m]$. These values correspond to those sets $P(W,u,\bar{c},R)$ that contain actual training data. Having determined $S := \{ u(c^1),\ldots ,u(c^m)\}\subset \mathcal A$, we then solve~(\ref{adv-s}-\ref{adv-e}) repeatedly for each choice of $u\in S$. This means that we solve possibly up to $m$ problems (in practice, $|S|\ll m$), where each problem is given as a quadratic convex program which can be solved efficiently. Furthermore, solving these problems can be easily parallelized. This approach to generating adversarial scenarios is used in the following experiments. However, note that in general, the radius $R$ can be chosen in a way that $U_f(W,\bar c,R)$ does not contain any of the training data points at all.

\subsection{Complexity of Robust Optimization}\label{sec:complexity}

As shown in the previous section, the uncertainty sets constructed via deep one-class learning have a more complex structure than the classical uncertainty classes considered in the robust optimization literature, e.g. polyhedral, ellipsoidal or discrete uncertainty. Therefore it is reasonable to expect that the robust problem is at least as hard to solve as for the classical uncertainty sets. In this section we prove that indeed all of the three mentioned uncertainty classes are special cases of the uncertainty sets constructed in the previous section, and therefore known NP-hardness results for the classical sets are valid for deep learning uncertainty as well.
\begin{lemma}\label{lem:complexityRO}
Respectively for each of the following uncertainty sets, there exists a neural network with at most two layers, given by its weight matrices $W^1,W^2$ and bias terms $b^1,b^2$, a center point $\bar c$ and a radius $R\ge 0$, such that $U_f(W,\bar c, R)=U_i$, where
\begin{enumerate}[(i)]
\item $U_1=\left\{ c\in \R^N \ | \ (c-a)^\top \Sigma (c-a)\le 1\right\}$, $\Sigma\in\R^{N\times N}$ is a symmetric and positive definite matrix and $a\in\R^N$,
\item $U_2=\left\{ c\in\R^N \ | \ Ac\le b\right\}$ where $A\in\R^{p\times N}$, $b\in\R^p$,
\item $U_3=\left\{ c\in\{ 0,1\}^N \ | \ Ac = b\right\}$ where $A\in\R^{p\times N}$, $b\in\R^p$.
\end{enumerate}
\end{lemma}
\begin{proof}
\textit{Case (i):} Calculate the Cholesky decomposition of $\Sigma$, i.e. a matrix $V\in\R^{N\times N}$ such that $V^\top V = \Sigma$ which can be done in polynomial time. Define a neural network with $L=2$ layers and bias term $b^1$ in the first layer as follows: set $W^1=E_n$, $b^1=-a$, $W^2=V$, $\sigma^1(x)=x$ and $\sigma^2(x)=x$. Furthermore define $\bar c = 0$, $R=1$ and use the Euclidean norm. Then for a given data point $c\in\R^N$ the output of the first layer is $c-a$. The output of the second layer is then $V(c-a)$. Substituting this into the norm constraint we obtain that $c\in U_f(W,\bar c, R)$ if and only if
\[
\| V(c-a)\|^2 = (c-a)^\top V^\top V (c-a) = (c-a)^\top \Sigma (c-a)\le 1
\] 
and hence $U_f(W,\bar c, R)=U_1$.

\textit{Case (ii):} Define a neural network with $L=1$ layer and bias term $b^1$ as follows: set $W^1=A$, $b^1=-b$ and $\sigma^1(x)=\max\{0,x\}$.  Furthermore define $\bar c = 0$, $R=0$ and use the Euclidean norm. Then for data point $c\in\R^N$ the output $y\in\R^p$ of the first layer is the all-zero vector if and only if $c\in U_2$. If $c\notin U_2$ at least one of the components of $y$ is strictly positive. Since by definition $c\in U_f(W,\bar c, R)$ if and only if $\|y\|^2\le 0$ and therefore $y=0$, we have $U_f(W,\bar c, R)=U_2$.

\textit{Case (iii):} Define a neural network with $L=1$ layer and bias term $b^1$ as follows: set 
\[
W^1 = \begin{pmatrix}E_n \\ A \\ A \end{pmatrix}, \quad b^1 = \begin{pmatrix}0 \\ -b \\ -b + \1 \end{pmatrix}
\]
and define the continuous piecewise affine activation function 
\[
\sigma^1(x) = \begin{cases} 
-x & \text{ if } x\le 0 \\ 
x & \text{ if } x\in [0,\frac{1}{2}] \\ 
-x +1 & \text{ if } x\in [\frac{1}{2},1] \\ 
x - 1 & \text{ if } x\ge 1 .\\ 
\end{cases}
\]
Note that $\sigma^1(0)=\sigma^1(1)=0$ and $\sigma^1(x)>0$ for all $x\in\R\setminus\{0,1\}$. Furthermore define $\bar c = 0$, $R=0$ and use the Euclidean norm. We have to show $U_f(W,\bar c, R)=U_3$. On the one hand for a given $c\in U_3$ we have $Ac-b = 0$ and therefore $Ac-b+\1 = \1$. Furthermore each component $c_i$ is either $0$ or $1$ and therefore by definition of $\sigma^1$ the output $y$ of the first layer is the all-zero vector and hence $\|y\|^2 \le 0$. On the other hand for a given $c\in U_f(W,\bar c, R)$ the output of the first layer must be the all-zero vector. Therefore by definition of $\sigma^1$ each component of $W^1c+b^1$ must be either $0$ or $1$. Hence $E_nc=c\in\{0,1\}^N$, $Ac-b\in \{0,1\}^p$ and $Ac-b+\1\in \{0,1\}^p$. The latter two conditions can only be true at the same time if $Ac=b$ and therefore $c\in U_3$.
\end{proof}
It is well known that the robust problem with objective uncertainty \eqref{eq:robustProblemObjectiveUncertainty} is already strongly NP-hard if $U$ is a polyhedron or an ellipsoid and if $X=\{0,1\}^N$; see \cite{buchheimrobust}. Therefore by Lemma~\ref{lem:complexityRO} Problem \eqref{eq:robustProblemObjectiveUncertainty} with uncertainty set $U_f(W,\bar c, R)$ is strongly NP-hard even if $X=\{0,1\}^N$. In the following we prove that even calculating the worst-case scenario is already strongly NP-hard.
\begin{theorem}
The adversarial problem 
\[
\max_{c\in U_f(W,\bar c, R)} c^\top x
\]
is strongly NP-hard.
\end{theorem} 
\begin{proof}
We can prove the result by reducing the $3$-partition problem, see \cite{gareyjohnson}. Given a set of weights $a_1,\ldots ,a_{3m}$ and a bound $B\in\N$ such that $\sum_{i=1}^{3m} a_i = mB$ and $\frac{B}{4}< a_i < \frac{B}{2}$, the $3$-partition problem asks if there is a partition of the weights into $m$ disjoint sets $A_1,\ldots , A_m\subset [3m]$ such that $\sum_{i\in A_j}a_i=B$ for all $j\in [m]$. Note that due to the weight restrictions any set $A_j$ with $\sum_{i\in A_j}a_i=B$ contains exactly $3$ elements. Hence the problem can be modeled by the constraints
\begin{align*}
& a^\top x^j = B \quad j\in [m]\\
& \sum_{j\in[m]}x^j = \1 \\
& x^j\in \{0,1\}^{3m} \ j\in [m] .
\end{align*} 
We can extend this to the optimization problem
\begin{align*}
& \max \ -y \\
s.t. \quad & a^\top x^j = B - By \quad j\in [m]\\
& \sum_{j\in[m]}x^j = \1 - \1 y \\
& x^j\in \{0,1\}^{3m} \ j\in [m]\\
& y\in \{0,1\} .
\end{align*} 
Note that the latter problem always has a feasible solution since we can set $y=1$ and all $x^j$ to the all-zero vector. On the other hand if $y=0$ then any feasible solution is a solution of the $3$-partition problem. Since we maximize $-y$ we can conclude that the $3$-partition instance is a yes-instance if and only if the optimal value of the latter optimization problem is $0$. The latter optimization problem is of the form
\[
\max_{c\in U} c^\top x 
\]  
with $U=\left\{ c\in \{ 0,1\}^{3m^2+1} \ | \ Ac=b\right\}$ and $x_{3m^2+1}=-1$, $x_i=0$ for all $i\in [3m^2]$. By Lemma \ref{lem:complexityRO} this is an instance of the problem $\max_{c\in U_f(W,\bar c, R)} c^\top x$.
\end{proof}

\section{Experiments}\label{sec:experiments}

In the following we conduct two experiments to assess the quality of the robust solutions found by our approach based on constructing uncertainty sets via deep neural networks (denoted as NN in the following). To this end, we compare against the related method from \cite{shang2017data} presented in Section~\ref{sec:preliminaries} which is to construct polyhedra based on support vector clustering with suitable kernel (denoted as Kernel). We will sometimes write NN or Kernel when referring to the solutions that were generated by the respective methods.

In the first experiment we conduct tests on randomly generated data to gain statistical insights. In the second experiment we conduct an experiment with real traffic-data from the City of Chicago.

All experiments were carried out on a virtual machine running Ubuntu 18.04.~and using ten Intel Xeon CPU E7-2850 processors. Code was implemented in Python, where we solved mathematical programs with Gurobi version 9.0.3 (except for solving the Kernel training problem~\eqref{eq:SVCdual}, where CPLEX version 12.8 was used). To train neural networks, we used the PyTorch implementation from \cite{ruff2018deep}, which is available online\footnote{https://github.com/lukasruff/Deep-SVDD-PyTorch}. Our code and data is made available online\footnote{https://github.com/goerigk/RO-DNN} as well.

\subsection{Experiment 1: Random Data Experiments}

\subsubsection{Setup}

In this section we solve two types of robust optimization problems. In the first type, the uncertain parameters only appear in the objective function. We use a simple budget constraint in addition, which results in optimization problems of the form
\begin{align*}
\min\ &\max_{c\in U} c^\top x \tag{Obj}\\
\text{s.t. } & \sum_{i=1}^{N}{x_i} = \text{RHS}^{obj} \\
& x \in [-1,1]^N .
\end{align*}
In the second type of problems, we assume that the uncertain parameters appear in the constraints, and solve
\begin{align*}
\max\ & \sum_{i=1}^{N}{x_i} \tag{Feas}\\
\text{s.t. } &c^\top x \le \text{RHS}^{feas} \quad \forall \ c\in U\\
& x \in [-1,1]^N,
\end{align*}
where in both cases, $U$ can be replaced by the uncertainty set generated by Kernel or by NN. In our experiments, we use problem dimensions $N=10,20,40$ and set $\text{RHS}^{obj} = N/2$ and $\text{RHS}^{feas}=50N$.

To generate data sets, we use \cite{shang2017data} as a starting point and consider three types of data. The first type is multivariate normal distributed (denoted as \textit{Gaussian}), where we create random positive definite covariance matrices using the sklearn Python package. The second type consists of two such distributions that are created independently; for each data point that is sampled, we decide with equal probability whether the first or the second distribution is used (denoted as \textit{Mixed Gaussian}). For Mixed Gaussian data, we additionally ensure that the mean-points of the two distributions do not lie in the same quadrant. Finally, the third set is sampled uniformly from a polyhedron that is constructed in the manner of budgeted uncertainty \cite{BertsimasS04}. That is, scenarios are sampled from sets
\[ \mathcal{U} = \{ c : c_i = \underline{c}_i + \overline{c}_i \delta_i,\ \sum_{i=1}^N \delta_i \le \Gamma \} \]
where the lower and upper bounds $\underline{c}_i$ and $\overline{c}_i$ are chosen randomly and $\Gamma=\frac{N}{2}$ (denoted as \textit{Polyhedral}). For a detailed description of the data generation, we refer to our publicly available code. Figure~\ref{exp1:fig1} gives examples for these types of data in two dimensions. White crosses indicate the training data points. The first two figures show the Gaussian set, the middle two figures show the Mixed Gaussian set, and the last two figures show the Polyhedral set.

For each set, we sample 10,000 test data points. For training, we use different sample sizes $m=250,500,1000$, where $0.95m$ data points are sampled from the correct distribution, and $0.05m$ data points are sampled uniformly in $[0,300]^N$ to simulate outliers which are also included in Figure~\ref{exp1:fig1}. For each configuration of $m$, $N$, and type (Gaussian, Mixed Gaussian, Polyhedral), 10 data sets are generated (a total of 270). 

For NN, we use neural networks with three layers of dimensions $N\times 50$, $50\times 50$ and $50\times 50$, respectively, with a ReLU operator after the first and second layer. To train the networks, we use a loss function that aims at minimizing the radius of data points preceding a given $(1-\eps)\%$ radius quantile, and maximizing the radius of subsequent data points. More specifically, to calculate our loss function, we sort the radii of the $m$ training data points; let $r_{\pi(1)} \le  r_{\pi(2)} \le \ldots \le r_{\pi(m)}$ be such a sorting. The loss is then calculated as
\[ \sum_{i=1}^k a_i \cdot r_{\pi(\lfloor(1-\epsilon)m\rfloor-i)} - \sum_{i=1}^k b_i \cdot r_{\pi(\lfloor(1-\epsilon)m\rfloor+i)}.\]
In this experiment, we use $k=5$, $a_i = 5i$, $b_i = i$, and $\epsilon = 0.1$.
This way, we enforce the network to include those data points where the classification confidence is high, and to exclude the others. 


\begin{figure}[htbp]
\begin{center}
\subfigure[Gaussian data, NN set.]{\includegraphics[width=0.49\textwidth]{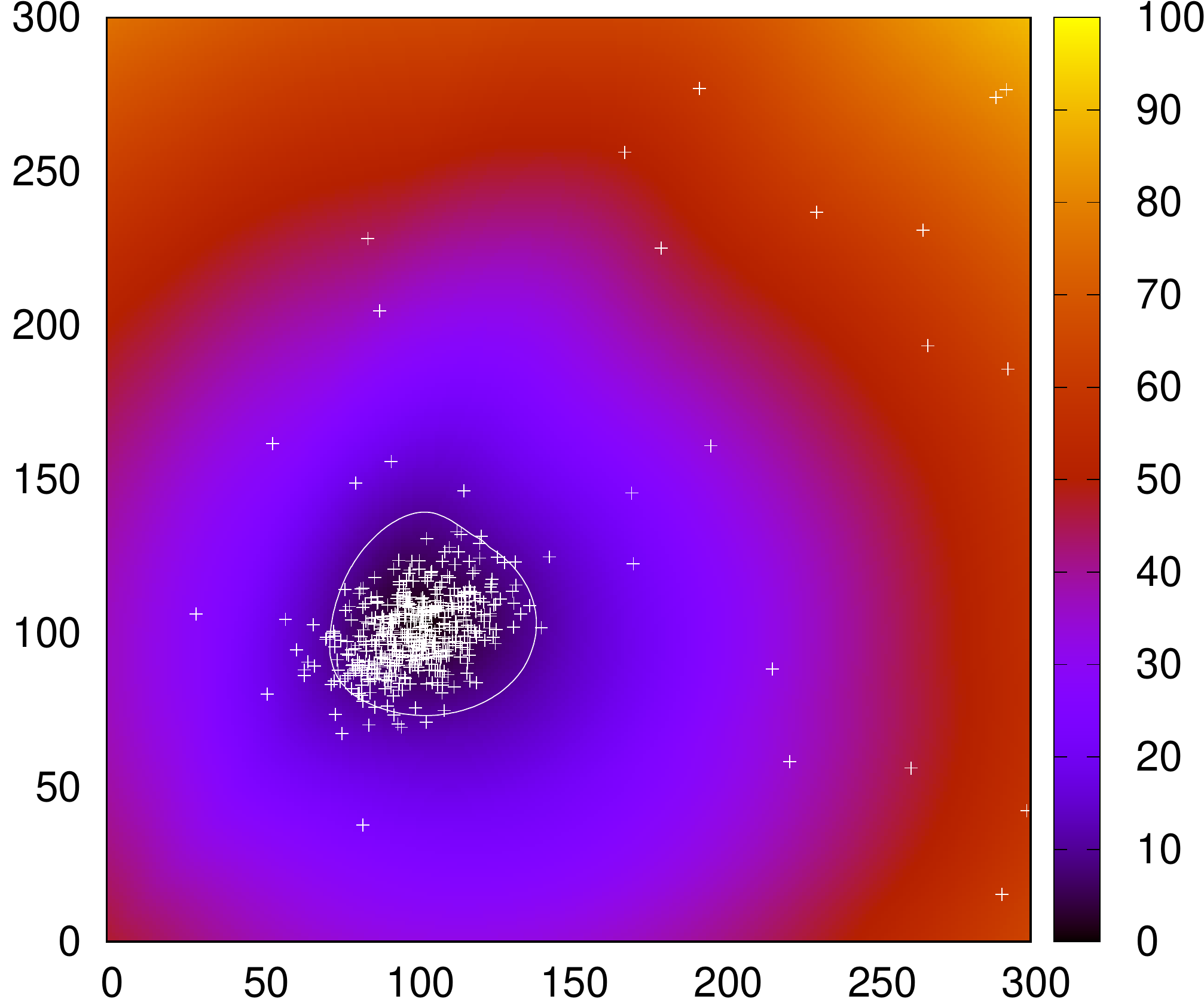}}%
\hfill%
\subfigure[Gaussian data, Kernel set.]{\includegraphics[width=0.49\textwidth]{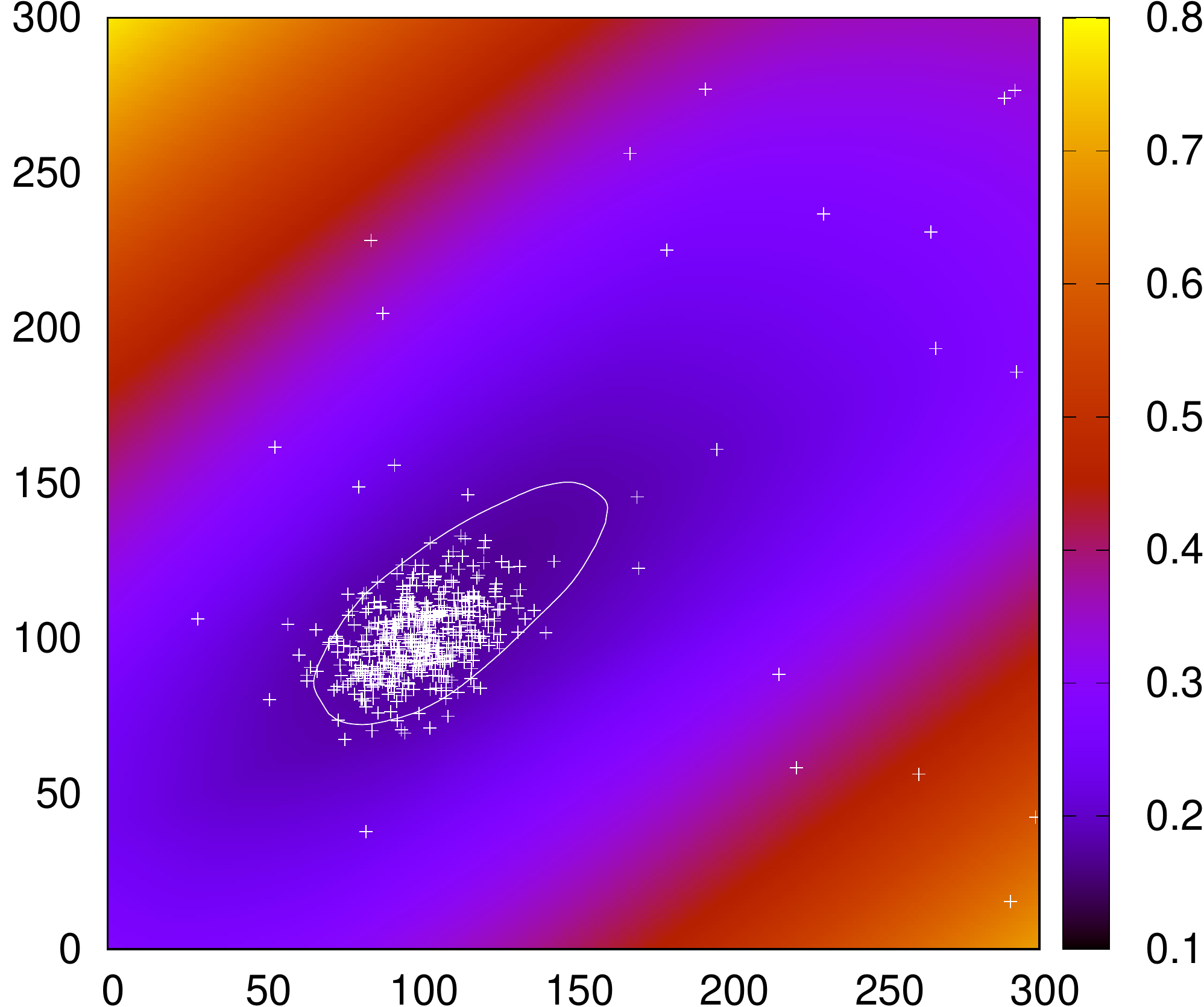}}
\subfigure[Mixed Gaussian data, NN set.]{\includegraphics[width=0.49\textwidth]{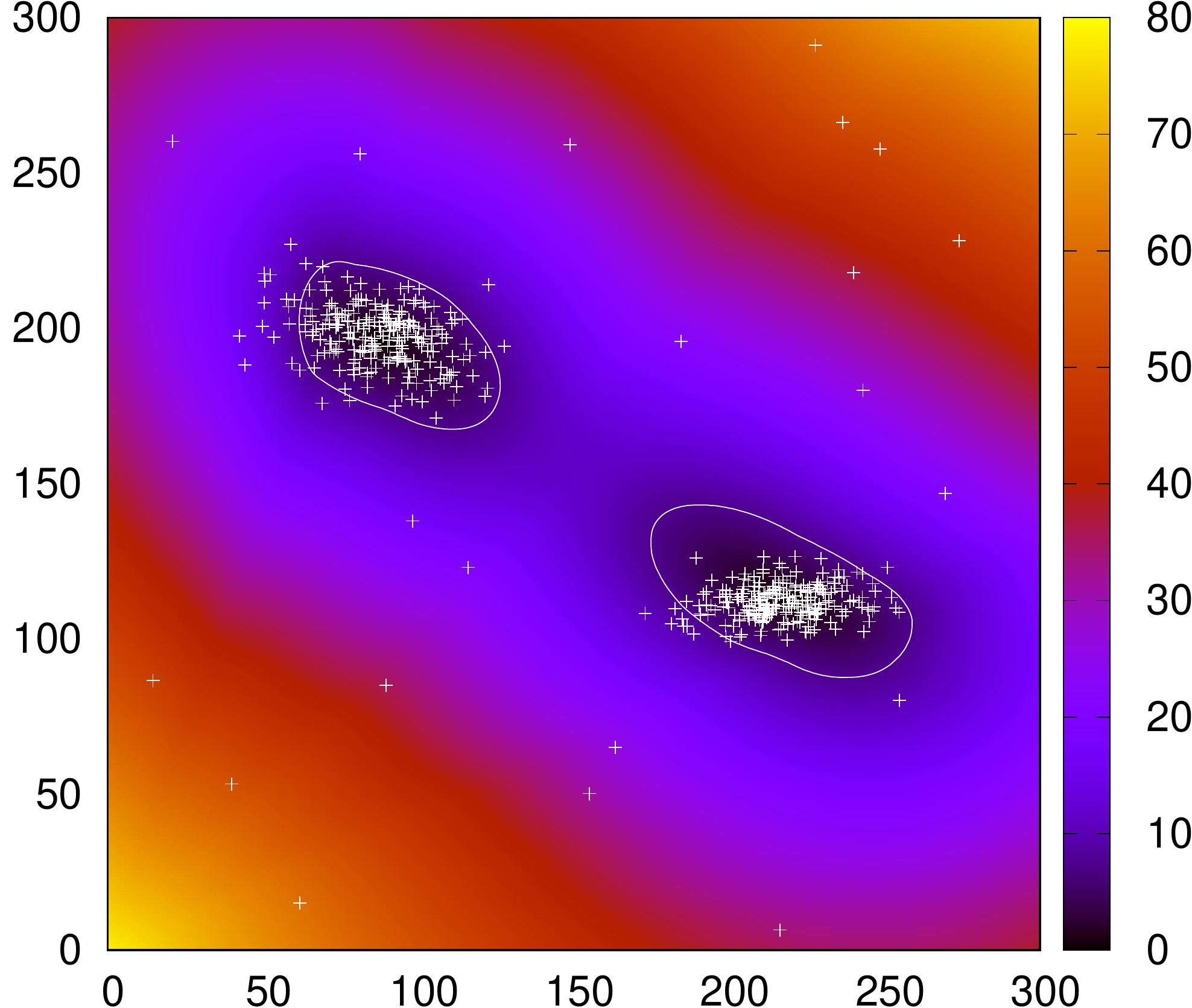}}%
\hfill%
\subfigure[Mixed Gaussian data, Kernel set.]{\includegraphics[width=0.49\textwidth]{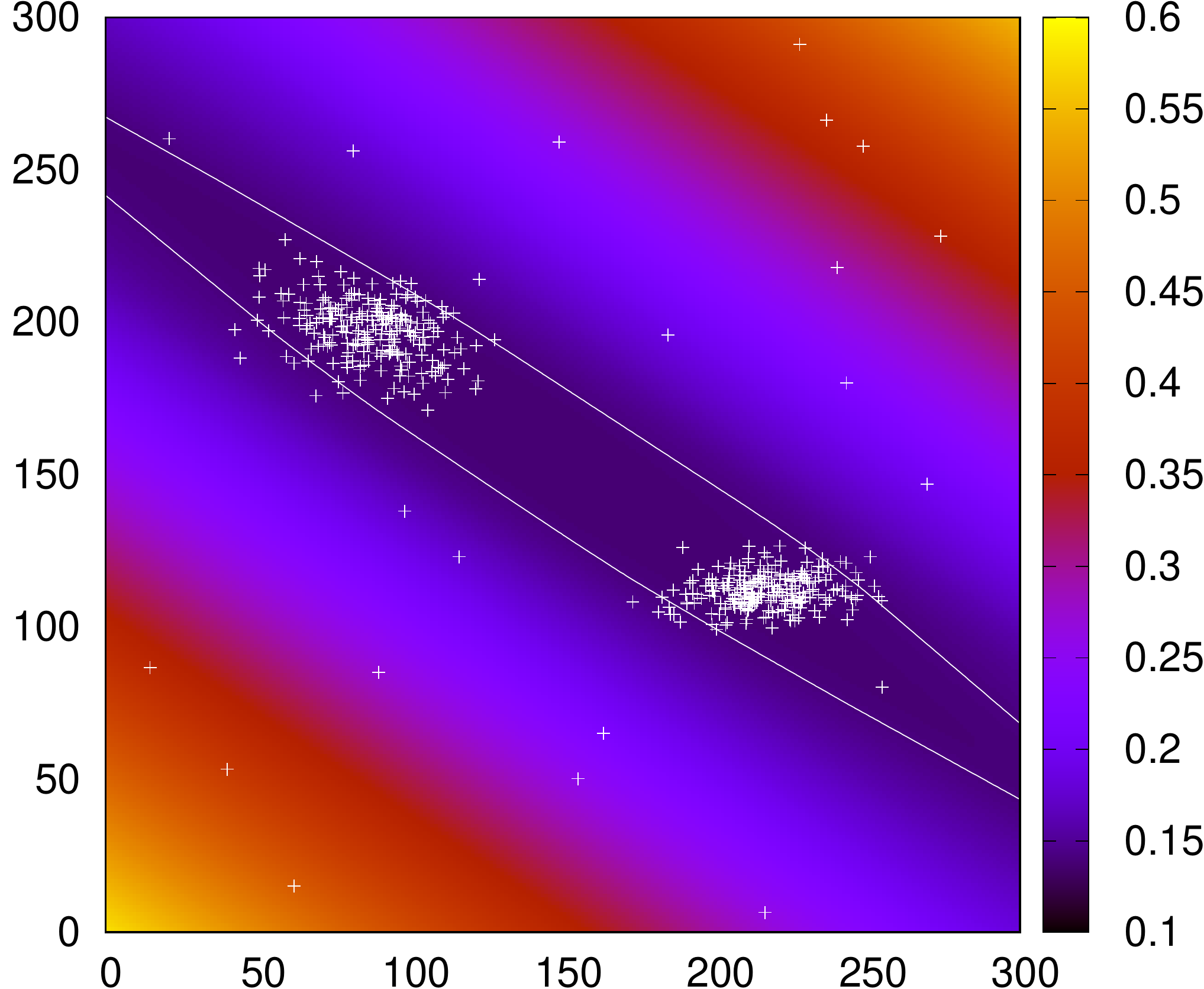}}
\subfigure[Polyhedral data, NN set.]{\includegraphics[width=0.49\textwidth]{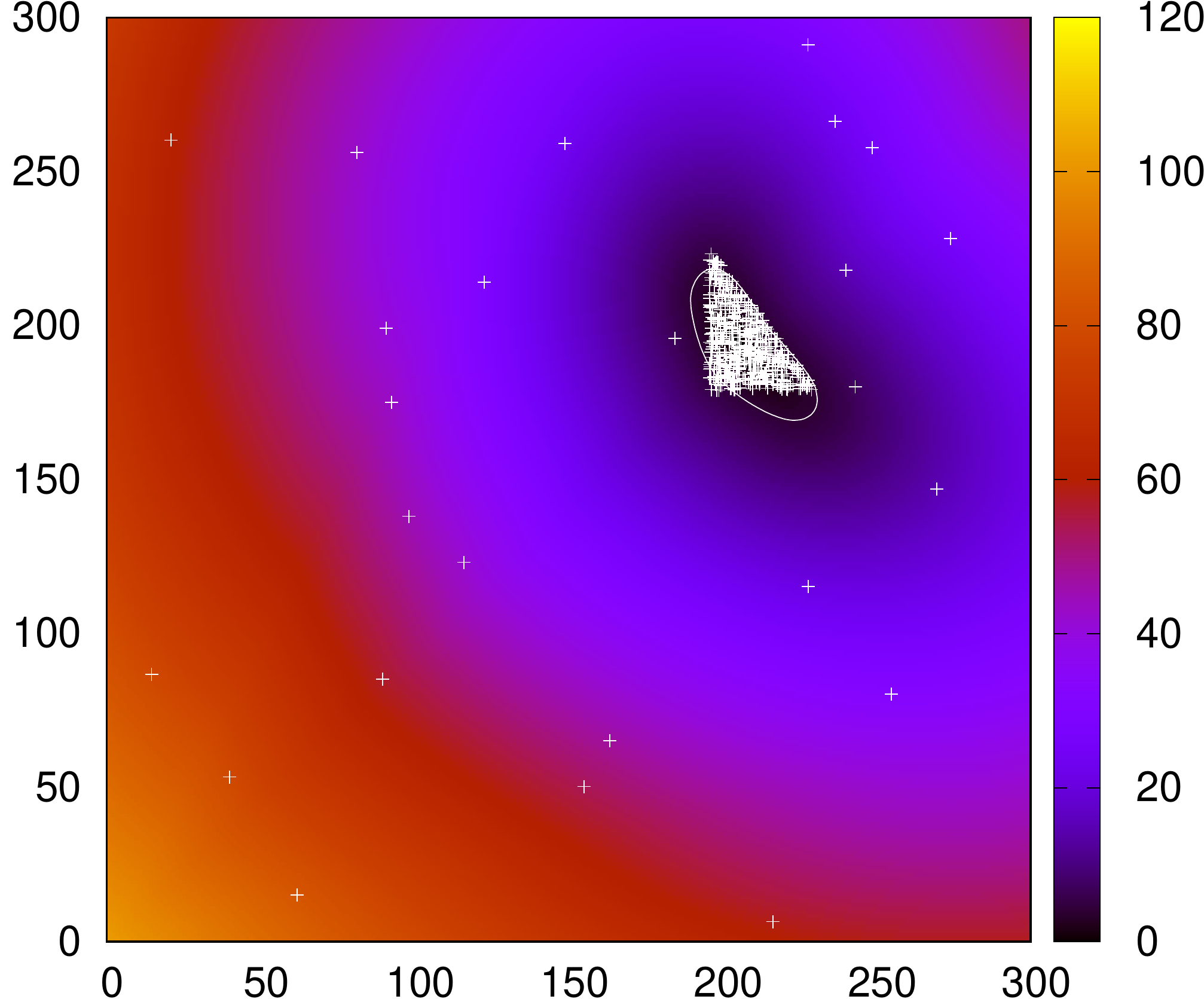}}%
\hfill%
\subfigure[Polyhedral data, Kernel set.]{\includegraphics[width=0.49\textwidth]{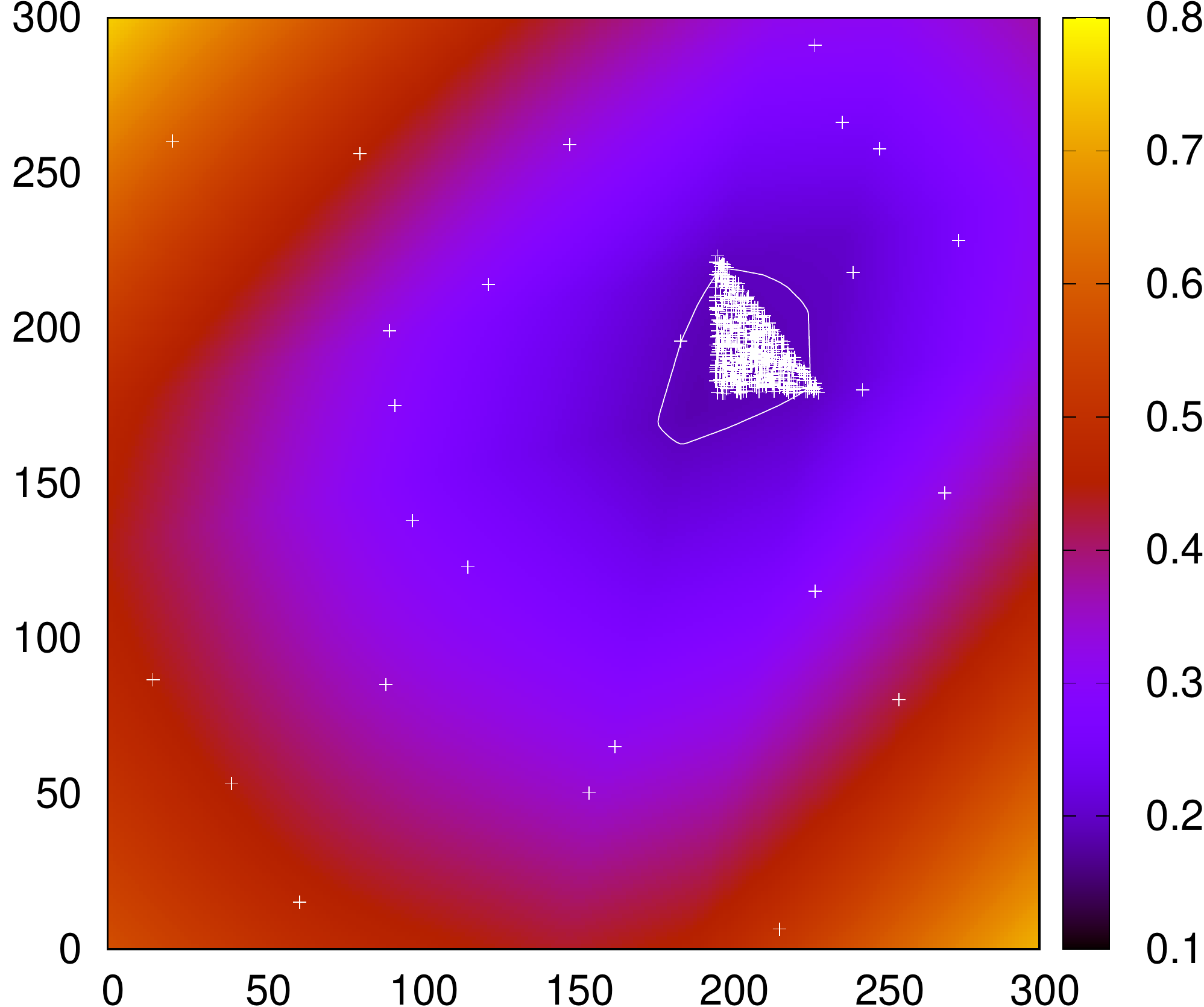}}
\caption{Visual comparison of data sets and uncertainty sets for examples with $N=2$.}\label{exp1:fig1}
\end{center}
\end{figure}

Plots on the left of Figure~\ref{exp1:fig1} show the radius the neural network associates with each point, where a white line shows the level set of the 90\% radius quantile; i.e., all points within the white lines are considered as possible scenarios. For the figures on the right, we show the value $\sum_{i\in \text{SV}} \alpha_i v_i^\top \1$ with $v_{ij} = |(Q(c-c^i))_j|$. The white line indicates those points where this value is less or equal to $\min_{i'\in \text{BSV}}\sum_{i\in\text{SV}}\alpha_i \|Q(c^{i'}-c^i) \|_1$, i.e., points that are considered as possible scenarios by Kernel.


\subsubsection{Results for Objective Uncertainty}

We first concentrate on results for problem (Obj), where uncertainty only appears in the objective. We train Kernel using $1-\nu=0.1$, corresponding to a $90\%$ quantile. For NN, we use a radius quantile of $90\%$ as well, which makes results comparable since both sets contain the same amount of training data. Note the radius $R$ could be determined by performing a validation step testing the corresponding robust solution for a certain set of radii. Note also that training NN does not necessarily give an optimal network, and repetitions can lead to different results. We train each neural network three times using 1000 epochs and use the network that gives the best loss value. When measuring training time, we add all three run times together.

We first present average training times and solution times in Table~\ref{exp2:tab2}. All values are in seconds. For training times, both methods require comparably little effort, with average times for NN ranging from 7 to 28 seconds, while times for Kernel are in the lower range of 1 to 8 seconds. Training times of both methods primarily depend on the sample size $m$, with problem dimension $N$ and the data type having less influence. Note that NN gives a neural network that can be used in combination with any radius $R$, while Kernel needs to be reoptimized for different values of $\nu$. While we keep these values fixed in this experiment, this affects training times if many quantiles need to be tested.

\begin{table}[htbp]
\begin{center}
\begin{tabular}{rrr|rr|rr}
 &  & & \multicolumn{2}{c|}{Train} & \multicolumn{2}{c}{Solve} \\
Type & $N$ & $m$ &  NN & Kernel & NN &  Kernel \\
\hline
\parbox[t]{2mm}{\multirow{9}{*}{\rotatebox[origin=c]{90}{Gaussian}}} & 10 & 250 & 8.2 & 0.9 & 26.7 & 0.1 \\
 & 10 & 500 & 16.4 & 2.2 & 22.9 & 0.2 \\
 & 10 & 1000 & 23.7 & 7.9 & 34.4 & 0.3 \\
 & 20 & 250 & 8.5 & 0.9 & 64.3 & 0.3 \\
 & 20 & 500 & 11.6 & 2.2 & 102.9 & 0.5 \\
 & 20 & 1000 & 19.6 & 8.1 & 120.9 & 1.1 \\
 & 40 & 250 & 8.9 & 1.0 & 151.0 & 1.2 \\
 & 40 & 500 & 19.1 & 2.3 & 280.8 & 2.0 \\
 & 40 & 1000 & 23.7 & 8.0 & 596.7 & 3.8 \\
\hline
\parbox[t]{2mm}{\multirow{9}{*}{\rotatebox[origin=c]{90}{Mixed Gaussian}}} & 10 & 250 & 8.1 & 0.9 & 48.6 & 0.1 \\
 & 10 & 500 & 15.0 & 2.2 & 91.4 & 0.2 \\
 & 10 & 1000 & 24.9 & 8.0 & 88.2 & 0.3 \\
 & 20 & 250 & 8.0 & 0.9 & 171.1 & 0.3 \\
 & 20 & 500 & 12.1 & 2.2 & 244.8 & 0.5 \\
 & 20 & 1000 & 21.0 & 8.0 & 295.0 & 1.1 \\
 & 40 & 250 & 10.2 & 0.9 & 470.3 & 1.2 \\
 & 40 & 500 & 20.2 & 2.2 & 820.9 & 2.0 \\
 & 40 & 1000 & 28.6 & 8.0 & 1117.2 & 3.9 \\
\hline
\parbox[t]{2mm}{\multirow{9}{*}{\rotatebox[origin=c]{90}{Polyhedral}}} & 10 & 250 & 8.3 & 0.9 & 31.9 & 0.1 \\
 & 10 & 500 & 14.6 & 2.2 & 46.1 & 0.2 \\
 & 10 & 1000 & 23.7 & 8.0 & 47.1 & 0.3 \\
 & 20 & 250 & 10.1 & 0.9 & 58.4 & 0.3 \\
 & 20 & 500 & 13.0 & 2.2 & 103.0 & 0.5 \\
 & 20 & 1000 & 19.1 & 7.8 & 329.9 & 1.1 \\
 & 40 & 250 & 10.6 & 0.9 & 236.9 & 1.3 \\
 & 40 & 500 & 14.6 & 2.3 & 523.3 & 2.1 \\
 & 40 & 1000 & 25.7 & 8.0 & 730.8 & 3.9 \\
\end{tabular}
\caption{Average solution and training times in seconds for Problem (Obj).}\label{exp2:tab2}
\end{center}
\end{table}

For solution times, differences between NN and Kernel become more pronounced. While Kernel only needs to solve a single linear program, our approach needs to generate scenarios iteratively. Hence, solving NN is orders of magnitude slower than solving Kernel.

\begin{table}[htbp]
\begin{center}
\begin{tabular}{rrr|rrr|rrr}
 &  & & \multicolumn{3}{c|}{Avg} & \multicolumn{3}{c}{90\% Quantile} \\
Type & $N$ & $m$ & NN & Kernel & Gap & NN & Kernel  & Gap  \\
\hline
\parbox[t]{2mm}{\multirow{9}{*}{\rotatebox[origin=c]{90}{Gaussian}}} & 10 & 250 & 554.1 & 702.7 & 26.8 & 586.4 & 725.4 & 23.7 \\
 & 10 & 500 & 458.1 & 614.7 & 34.2 & 487.4 & 639.7 & 31.2 \\
 & 10 & 1000 & 410.4 & 536.3 & 30.7 & 443.2 & 565.3 & 27.6 \\
 & 20 & 250 & 907.1 & 1402.6 & 54.6 & 956.9 & 1434.7 & 49.9 \\
 & 20 & 500 & 1011.7 & 1409.8 & 39.3 & 1051.6 & 1440.1 & 36.9 \\
 & 20 & 1000 & 1029.6 & 1396.9 & 35.7 & 1080.4 & 1428.2 & 32.2 \\
 & 40 & 250 & 1814.1 & 2770.7 & 52.7 & 1865.2 & 2812.5 & 50.8 \\
 & 40 & 500 & 1921.7 & 2897.8 & 50.8 & 1980.9 & 2938.9 & 48.4 \\
 & 40 & 1000 & 1914.6 & 2875.3 & 50.2 & 1994.6 & 2914.2 & 46.1 \\
\hline
\parbox[t]{2mm}{\multirow{9}{*}{\rotatebox[origin=c]{90}{Mixes Gaussian}}} & 10 & 250 & 620.4 & 715.3 & 15.3 & 689.5 & 742.6 & 7.7 \\
 & 10 & 500 & 622.5 & 675.6 & 8.5 & 689.5 & 703.3 & 2.0 \\
 & 10 & 1000 & 642.5 & 666.9 & 3.8 & 696.9 & 700.3 & 0.5 \\
 & 20 & 250 & 1206.5 & 1437.8 & 19.2 & 1275.4 & 1471.3 & 15.4 \\
 & 20 & 500 & 1143.9 & 1413.4 & 23.6 & 1208.9 & 1445.7 & 19.6 \\
 & 20 & 1000 & 1134.4 & 1363.7 & 20.2 & 1206.1 & 1396.9 & 15.8 \\
 & 40 & 250 & 2407.6 & 2808.2 & 16.6 & 2497.8 & 2850.6 & 14.1 \\
 & 40 & 500 & 2408.0 & 2905.9 & 20.7 & 2502.7 & 2949.5 & 17.9 \\
 & 40 & 1000 & 2379.7 & 2874.5 & 20.8 & 2478.2 & 2915.6 & 17.6 \\
\hline
\parbox[t]{2mm}{\multirow{9}{*}{\rotatebox[origin=c]{90}{Polyhedral}}} & 10 & 250 & 546.1 & 686.4 & 25.7 & 592.7 & 721.2 & 21.7 \\
 & 10 & 500 & 578.7 & 687.5 & 18.8 & 624.9 & 725.0 & 16.0 \\
 & 10 & 1000 & 517.5 & 692.9 & 33.9 & 562.3 & 722.0 & 28.4 \\
 & 20 & 250 & 1103.6 & 1424.9 & 29.1 & 1171.0 & 1473.2 & 25.8 \\
 & 20 & 500 & 1087.9 & 1462.0 & 34.4 & 1150.6 & 1505.0 & 30.8 \\
 & 20 & 1000 & 1195.1 & 1469.3 & 23.0 & 1260.2 & 1513.9 & 20.1 \\
 & 40 & 250 & 2046.2 & 2879.7 & 40.7 & 2143.7 & 2948.6 & 37.6 \\
 & 40 & 500 & 2341.9 & 2969.7 & 26.8 & 2432.7 & 3045.3 & 25.2 \\
 & 40 & 1000 & 2078.6 & 2924.6 & 40.7 & 2178.7 & 2989.5 & 37.2
\end{tabular}
\caption{Average objective values of Problem (Obj). Smaller values indicate better performance.}\label{exp2:tab1}
\end{center}
\end{table}

With this disadvantage in mind, we now consider the quality of solutions as presented in Table~\ref{exp2:tab1}. We show the average of the average out-of-sample objective value and the average of the $90\%$-quantile out-of-sample objective value for both robust solutions derived by NN and Kernel. Recall that the aim is to minimize the objective value. In columns denoted ''Gap'', we show the relative increase of the average results for Kernel over NN.

We see that NN strongly outperforms Kernel on all combinations of instance type, problem dimension, and sample size, with respect to both average and quantile performance. In most cases, even the 90\% objective quantile of NN is smaller than the average objective value of Kernel. Gaps on average objective values range from 26\% to 54\% for Gaussian data (i.e., the Kernel solution can be 54\% more costly on average when evaluated out-of-sample), from 3\% to 23\% for Mixed Gaussian data, and from 18\% to 40\% for Polyhedral data. There is no clear trend in performance gaps when considering sample size $m$, but gaps tend to become larger with increasing problem dimension $N$, which indicates better scaling capabilities of our method.

Summarizing these findings, we see that NN solutions require more computational effort to find, but have the benefit of a considerably stronger performance.

\subsubsection{Results for Constraint Uncertainty}

We now consider the performance when solving problems of type (Feas) with $N=20$ and $m=500$. In this experiment, we generate solutions for different degrees of robustness (controlled through parameters $\nu$ and $R$). We consider the cases $(1-\nu)\in \{0.02, 0.04, 0.08, 0.16, 0.32, 0.64\}$, and use the same quantile values for $R$.

Each solution is evaluated with respect to its performance in the objective function $\sum_{i=1}^N x_i$ and with respect to its feasibility. To have a detailed understanding of the latter, we evaluate $c^\top x$ out-of-sample and consider the 90\% quantile of these values, i.e., we calculate how small $\text{RHS}^{feas}$ can be such that the solution remains feasible in 90\% of scenarios. When optimizing, we use $\text{RHS}^{feas}=50N=1000$. We solve each of the 10 instances of each data type and then average these performance values for each parameter value $\nu$ and $R$, respectively.

Figure~\ref{exp2:fig2} shows the resulting average trade-off curves for each data type. Each curve consists of six points, corresponding to the different degrees of robustness when calculating solutions. Points towards the right of the plot have better quality with respect to the objective function. Points towards the bottom of the plot are more robust, as they remain $90\%$ feasible for smaller right-hand sides, i.e., plan in more buffer in the uncertain constraint. An ideal solution would be in the bottom right corner, and a reasonable trade-off curve would go from top right to bottom left.

\begin{figure}[htb]
\begin{center}
\subfigure[Gaussian.]{\includegraphics[width=0.3\textwidth]{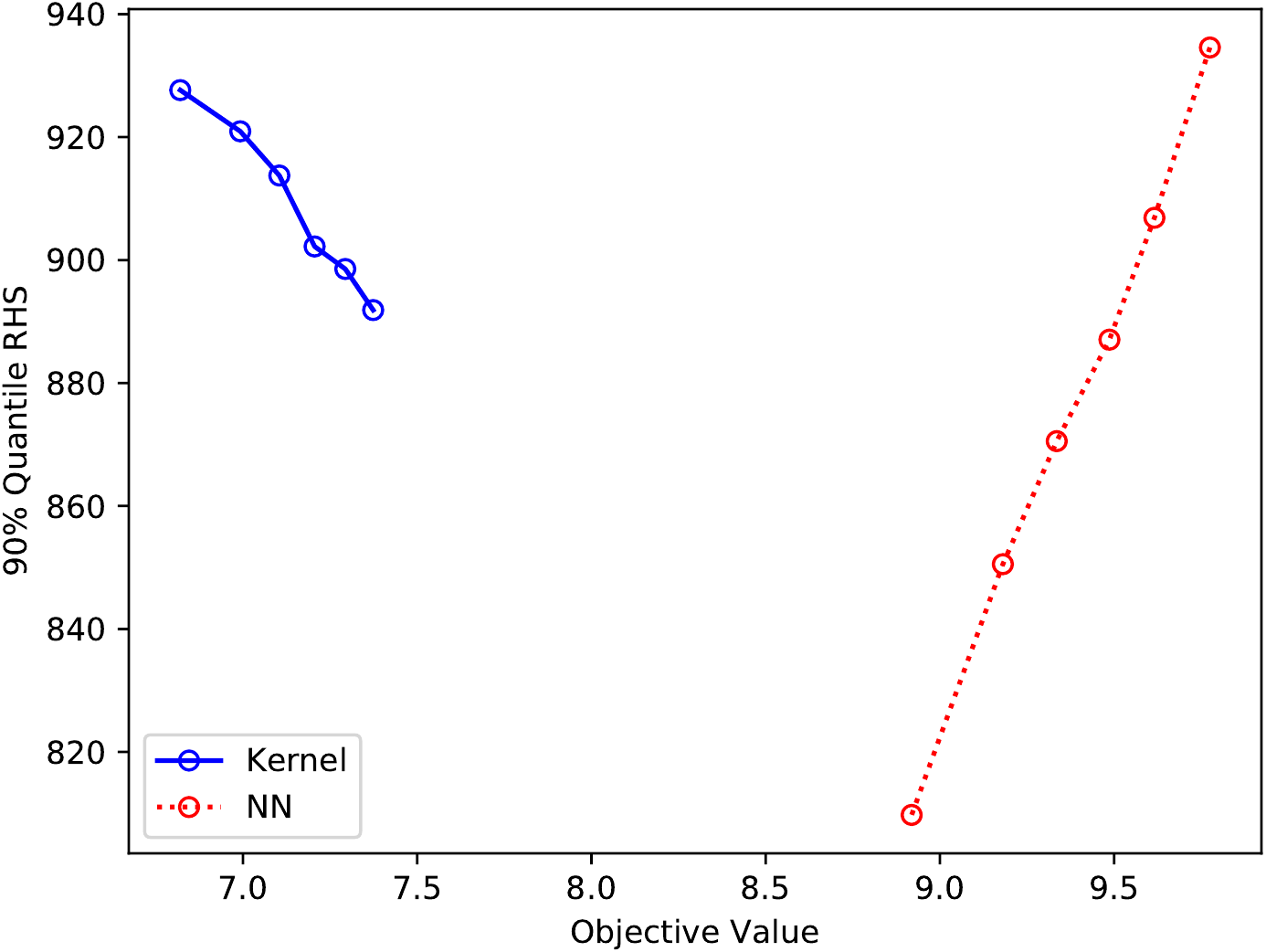}}%
\hfill%
\subfigure[Mixed Gaussian.]{\includegraphics[width=0.3\textwidth]{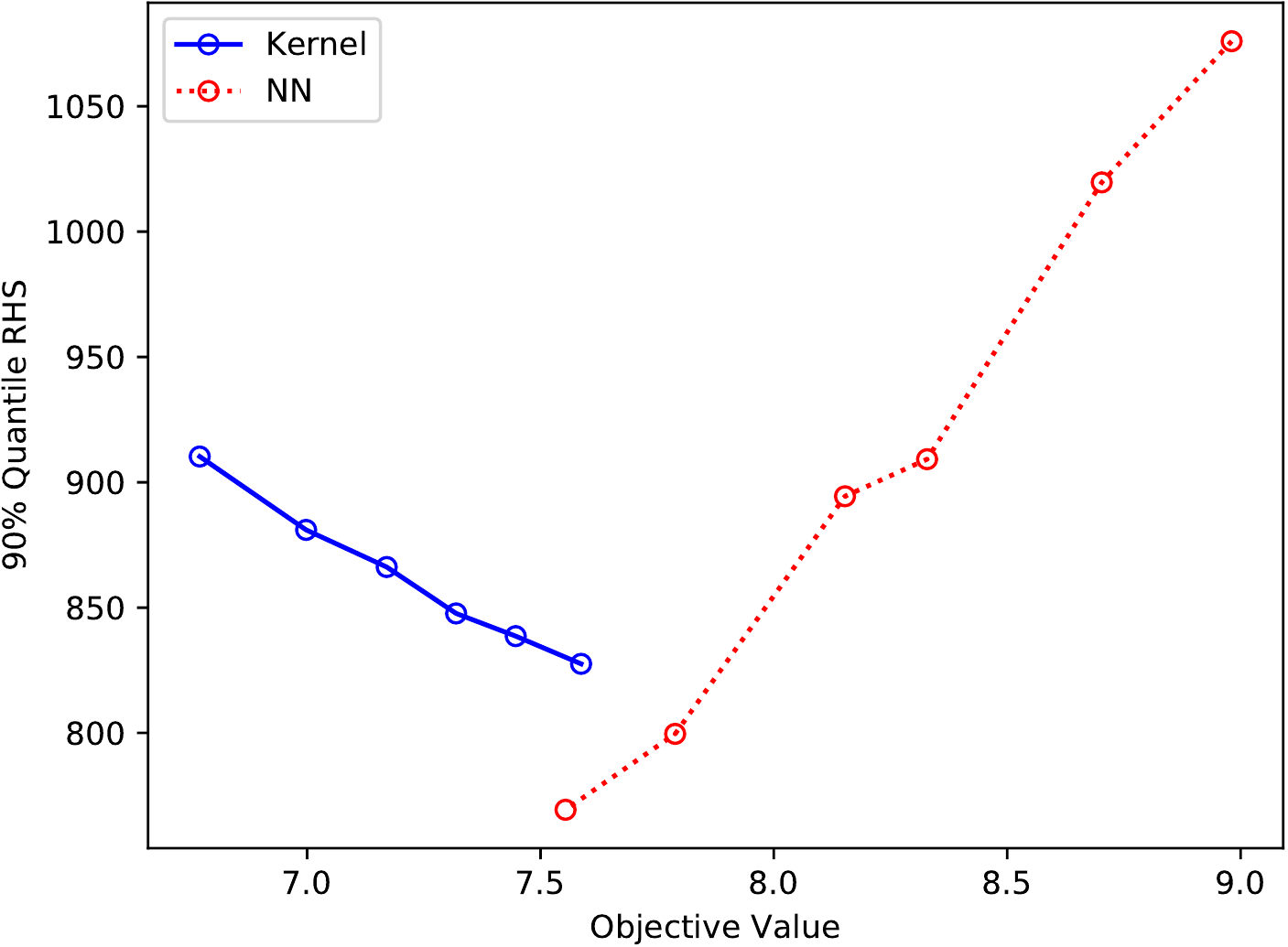}}%
\hfill%
\subfigure[Polyhedral.]{\includegraphics[width=0.3\textwidth]{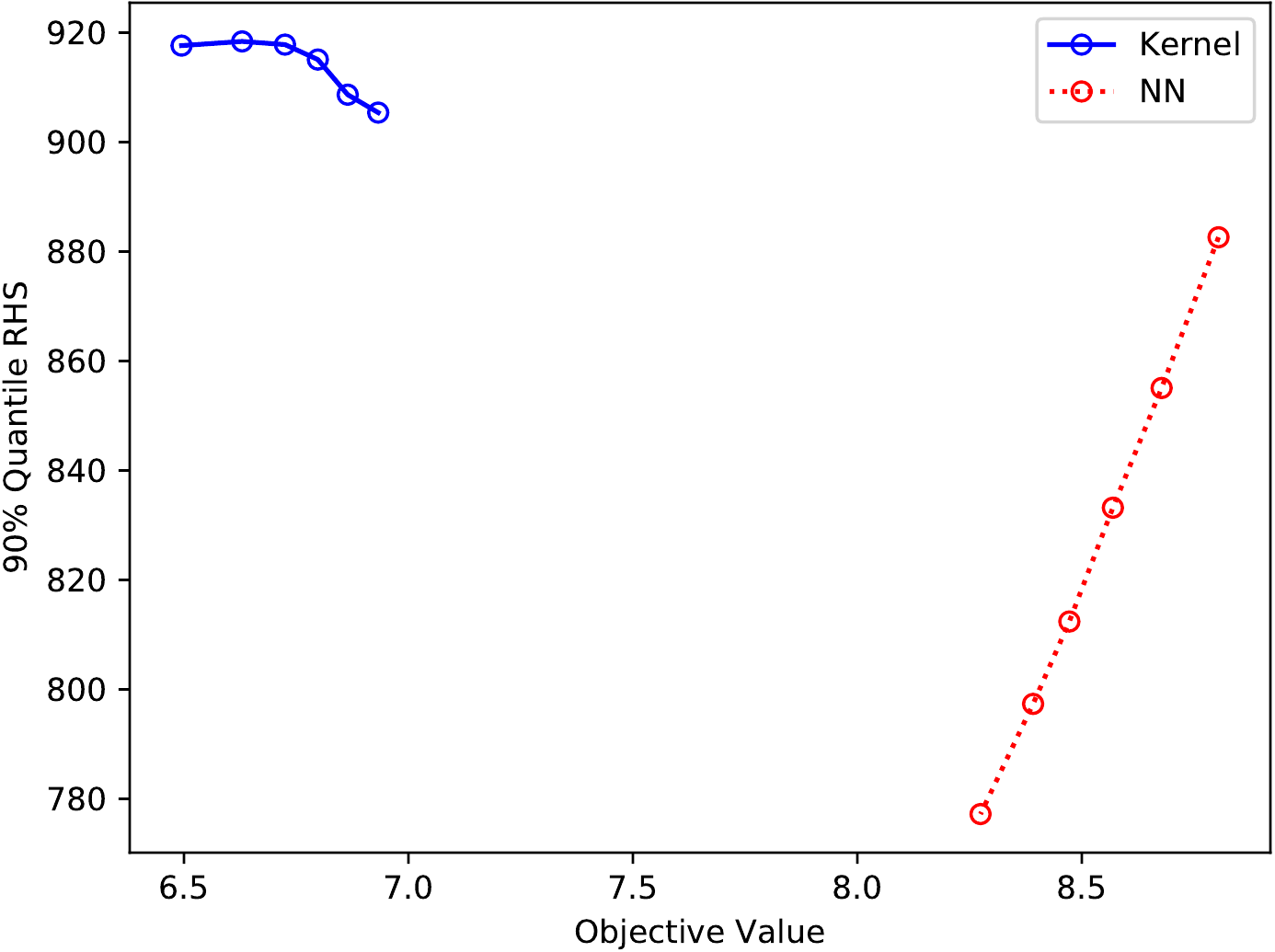}}
\caption{Robustness vs objective value in Problem (Feas) for $N=20$, $m=500$. Values to the bottom right indicate better performance.}\label{exp2:fig2}
\end{center}
\end{figure}

Comparing Kernel and NN solutions, we find that NN shows a considerably better performance being able to produce solutions which at the same time reach larger buffers in the uncertain constraint while simultaneously reaching consistently larger objective values. Furthermore, NN solutions give a useful trade-off between robustness and objective value (curves are diagonal from bottom left to top right), while this is not the case for Kernel solutions. Here, for all three data sets, solutions calculated with different values of $\nu$ are dominated by the smallest set $U_\nu$ with $1-\nu=0.02$, which gives both better objective values and better robustness than other values of $\nu$. This may mean that the corresponding uncertainty sets do not offer protection towards the relevant scenarios.

\subsection{Experiment 2: Real-World Data Experiments}

While the results of the first experiment demonstrate strong dominance of NN solutions over Kernel solutions, instances were randomly generated and are thus not necessarily representative for real-world data. For this reason as a second experiment we solve the shortest path problem with travel time uncertainty using the travel time data of the City of Chicago that was first collected and used in \cite{chassein2019algorithms}. Travel times for each arc in the Chicago network were collected from March to May 2017 through a live traffic data interface. The whole network has 1308 arcs and 538 nodes, but for this experiment, we only use the inner city region containing 189 arcs and 165 nodes. Furthermore, we use scenarios representing weekend travel times. There is a total of 1141 such observations (each observation giving the travel time for every arc in the network), of which we use 570 for training and 571 for testing.

Due to the way the graph is generated, there exist edges that are completely correlated. This happens as original travel speed observations were collected on segments that may contain multiple edges. As the covariance matrix becomes degenerate in this case, small random noise was added to the travel times to calculate the Kernel sets. As noted below, this effect seems to have considerable impact on both training and solution times for Kernel. Note that the covariance matrix is not used in the training process of NN. The neural network we used has three layers of size $189\times 50$, $50 \times 50$ and $50\times 50$, respectively. 


Since for the larger data set the kernel method failed to terminate in reasonable time we restricted the experiment to the city center graph and weekend scenarios. Furthermore, we used a $95\%$ quantile instead of $90\%$ as in the previous experiment with objective uncertainty, as this further reduces the size of the optimization model.

Additionally, we include a robust optimization approach where training data is used without any further modification as a discrete uncertainty set which we refer to as \textit{Discrete}. Note that since optimize a linear function over the uncertainty set this approach is equivalent to selecting the convex hull of the training data as uncertainty set. This method was omitted in the previous experiment since it does not filter out outliers and thus cannot be expected to perform competitively.

The robust optimization problem we consider here is to find a shortest path over the uncertain data where the uncertainty is only present in the objective function. We generate ten random source-sink pairs, where we ensure that at least eight arcs need to be traversed from source to sink to exclude pairs that are nearby. For each source-sink pair, a robust path is calculated using Discrete, NN and Kernel. On average after training, Kernel paths take 1497.7 seconds to compute, while NN paths take 794.8 seconds. The discrete scenario approach takes only $4.4$ seconds to be solved on average.

We present results on this experiment in Figures~\ref{exp3:figSP1}-\ref{exp3:figSP3}. In each figure, we show the path computed by Kernel on the left and the path computed by NN on the right for an exemplary instance. In the middle, we show the out-of-sample histogram of travel times for all paths, including Discrete.

\begin{figure}[htbp]
\begin{center}
\subfigure[Kernel path 01.]{\includegraphics[width=0.27\textwidth]{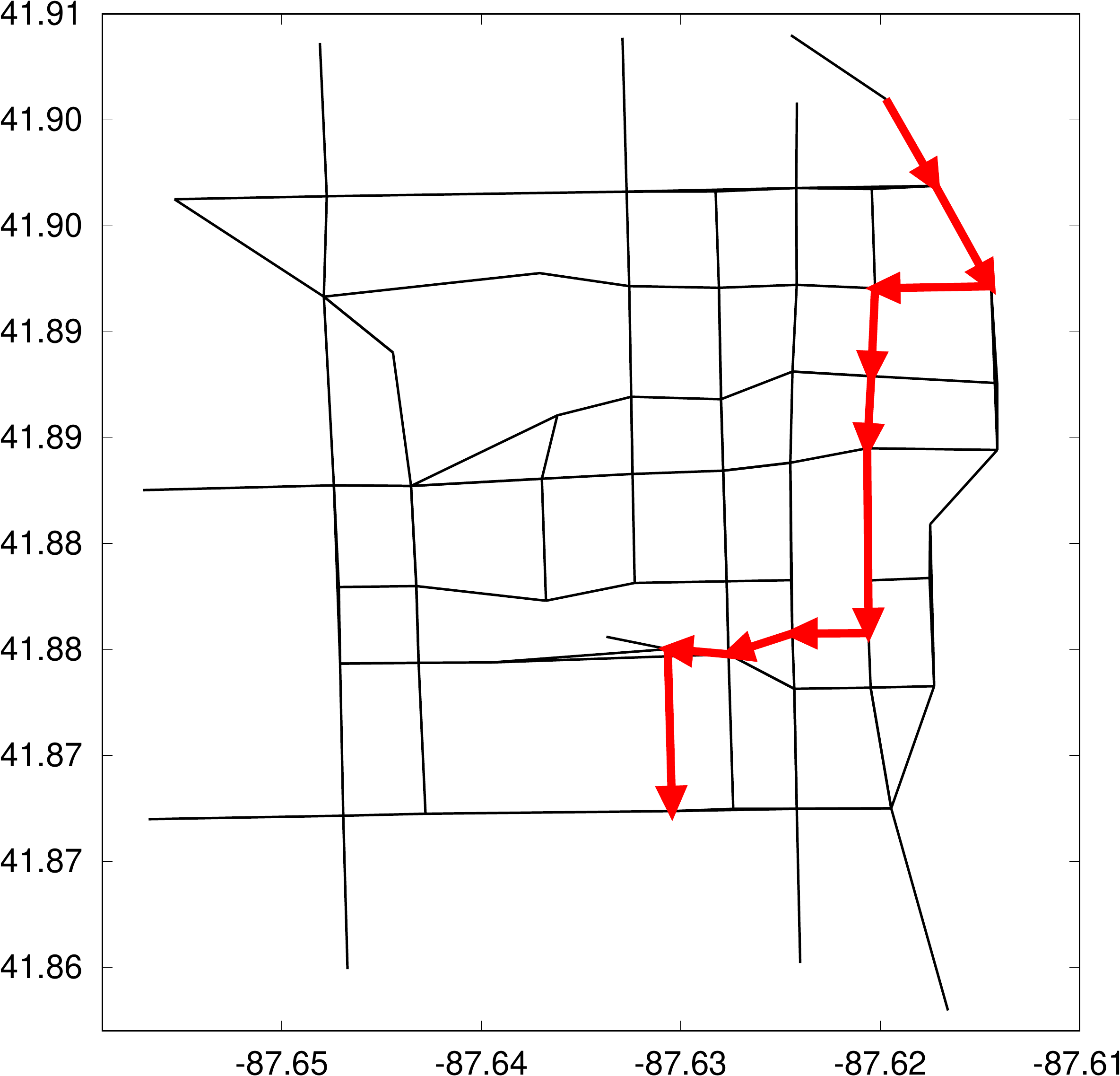}}%
\hspace{0.04\textwidth}%
\subfigure[Out-of-sample path lengths 01.\label{histo:01}]{\includegraphics[width=0.36\textwidth]{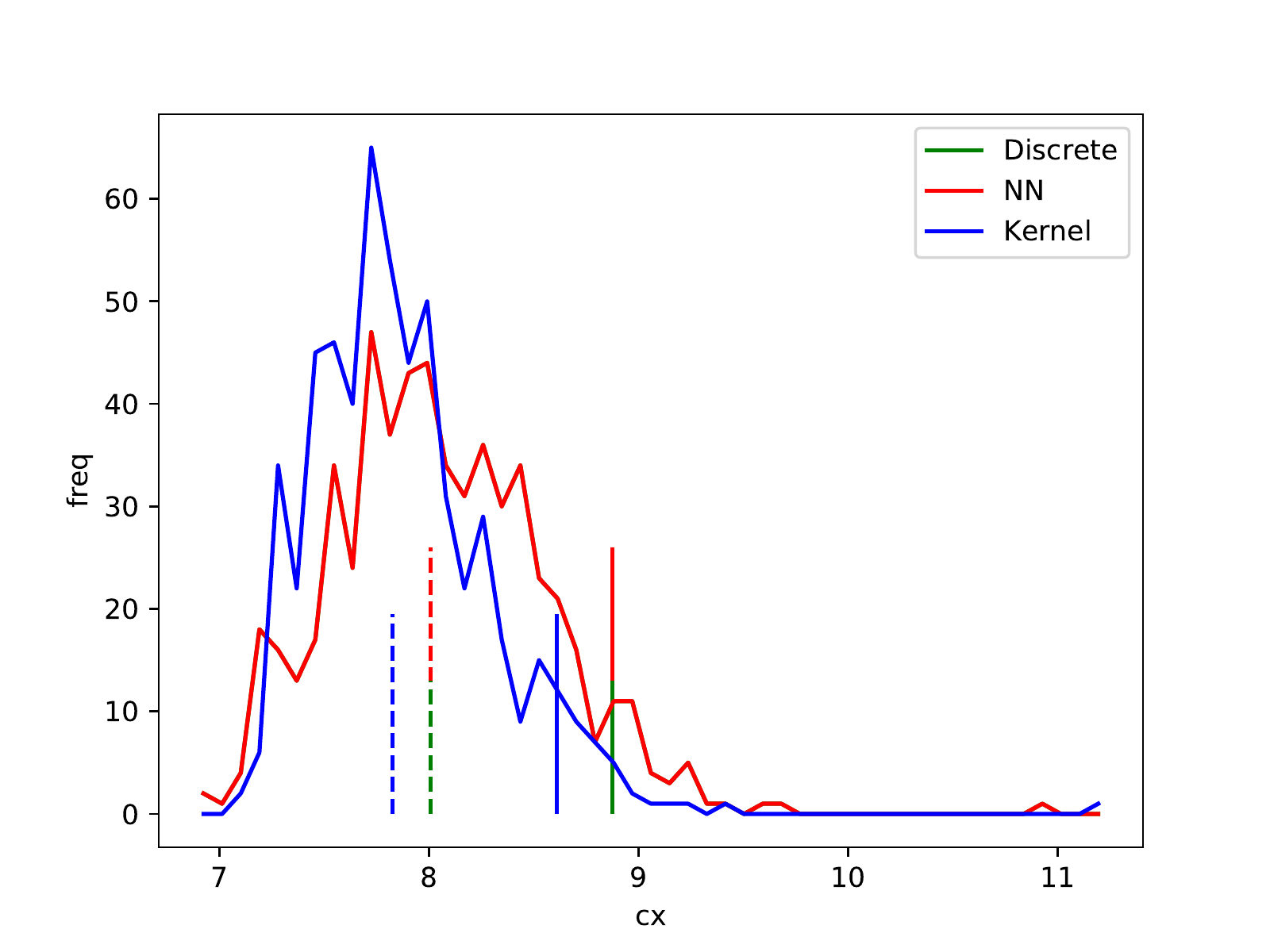}}%
\hspace{0.04\textwidth}%
\subfigure[NN path 01.]{\includegraphics[width=0.27\textwidth]{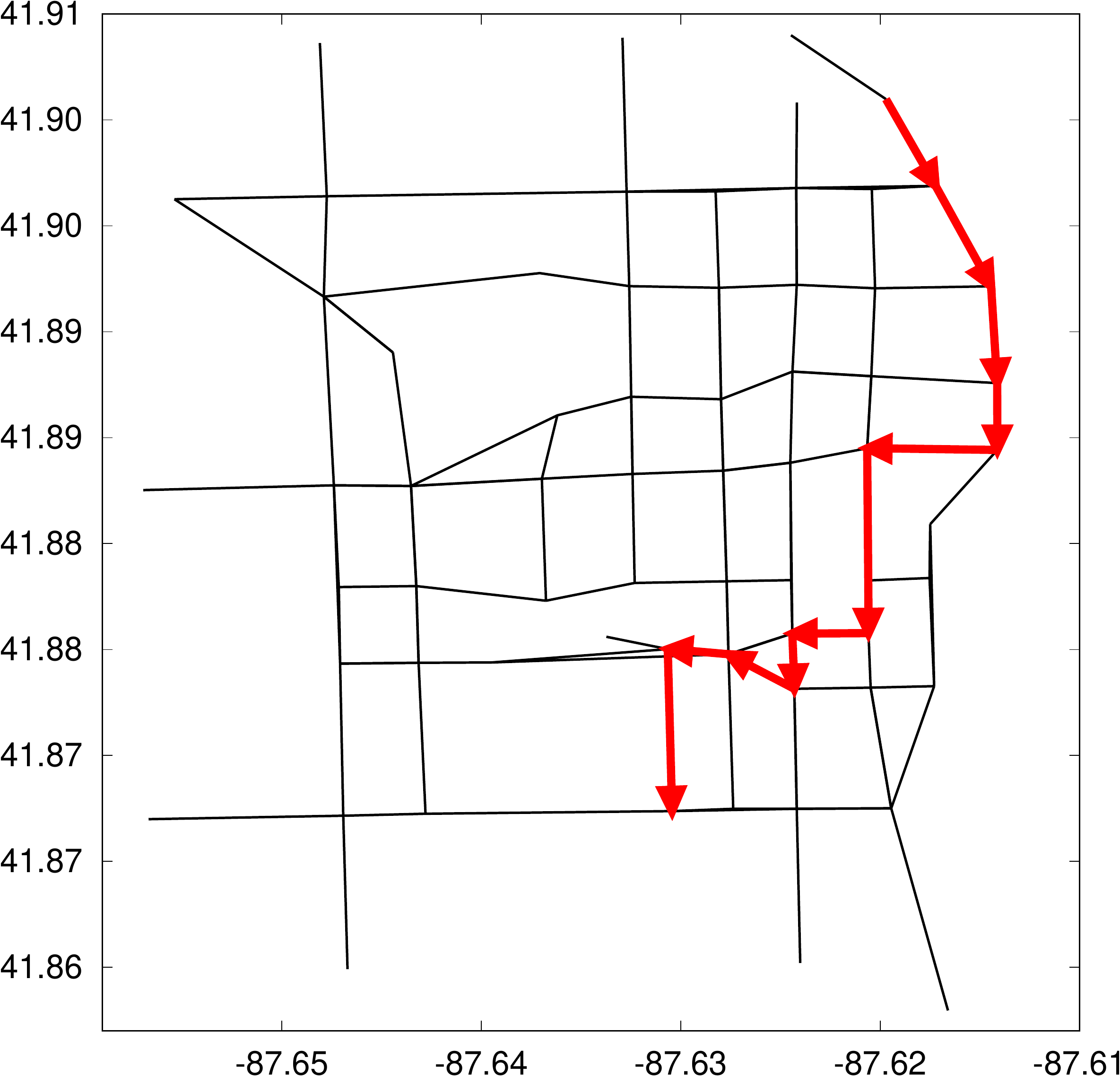}}
\subfigure[Kernel path 02.]{\includegraphics[width=0.27\textwidth]{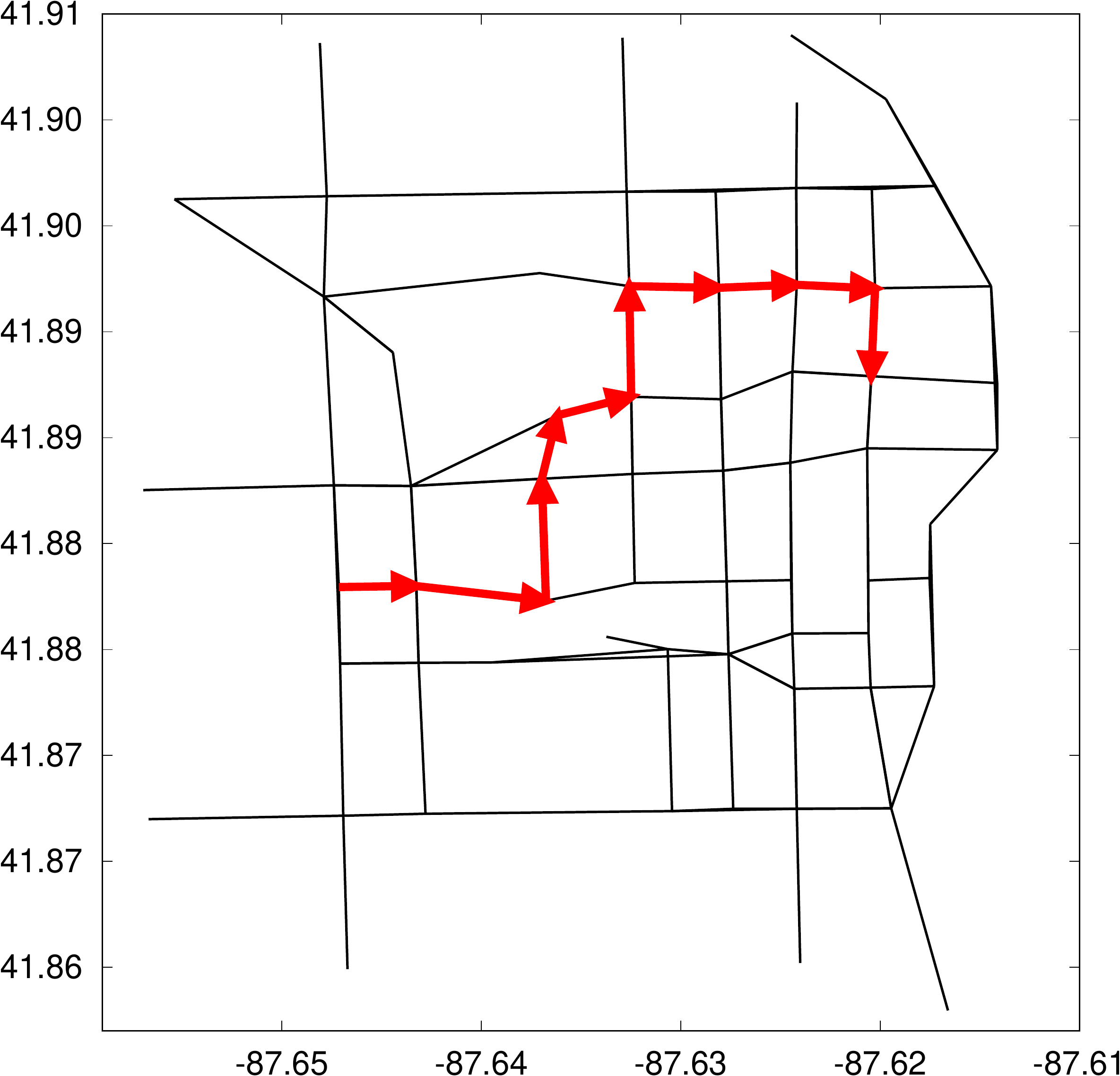}}%
\hspace{0.04\textwidth}%
\subfigure[Out-of-sample path lengths 02.]{\includegraphics[width=0.36\textwidth]{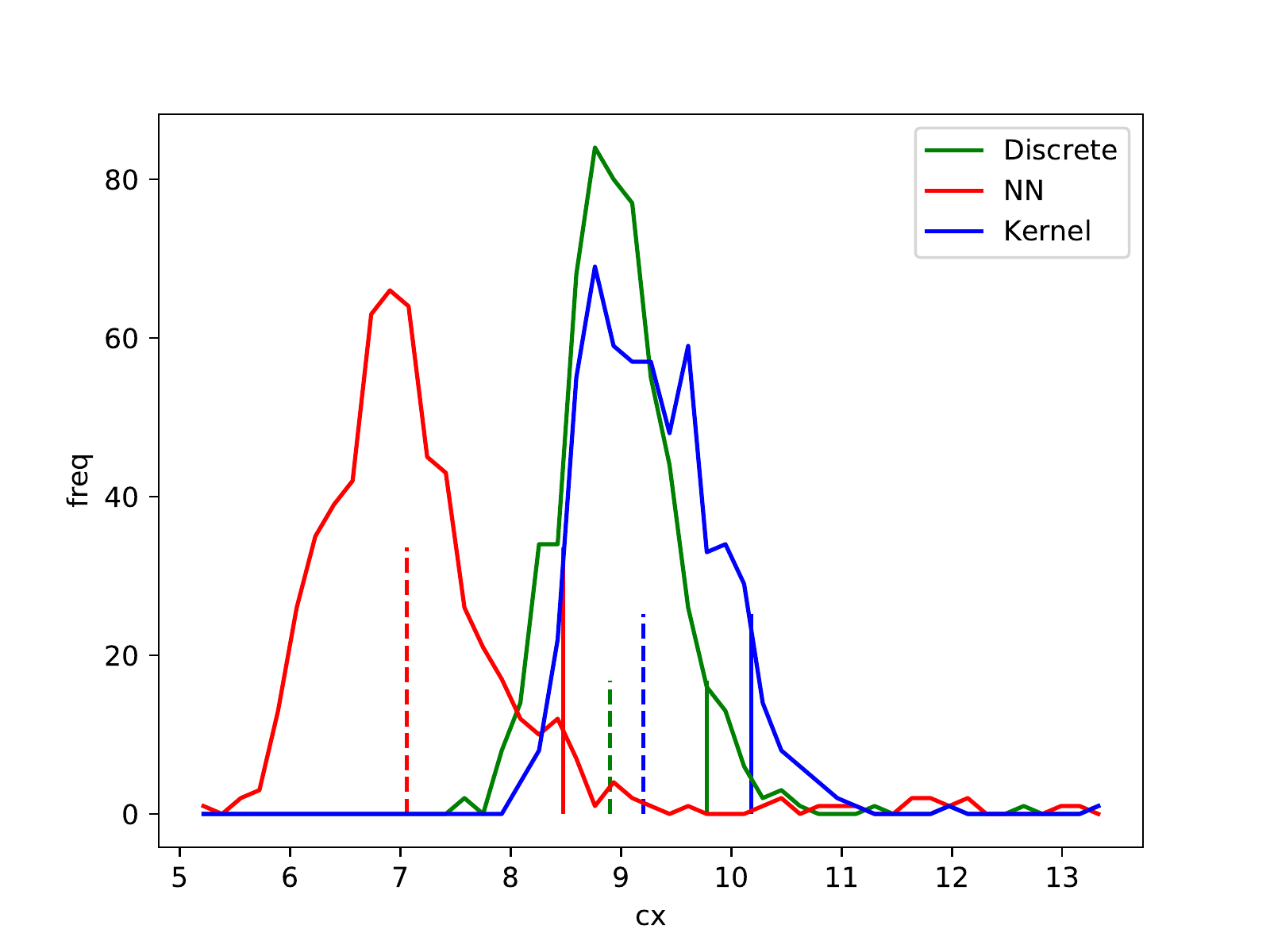}}%
\hspace{0.04\textwidth}%
\subfigure[NN path 02.]{\includegraphics[width=0.27\textwidth]{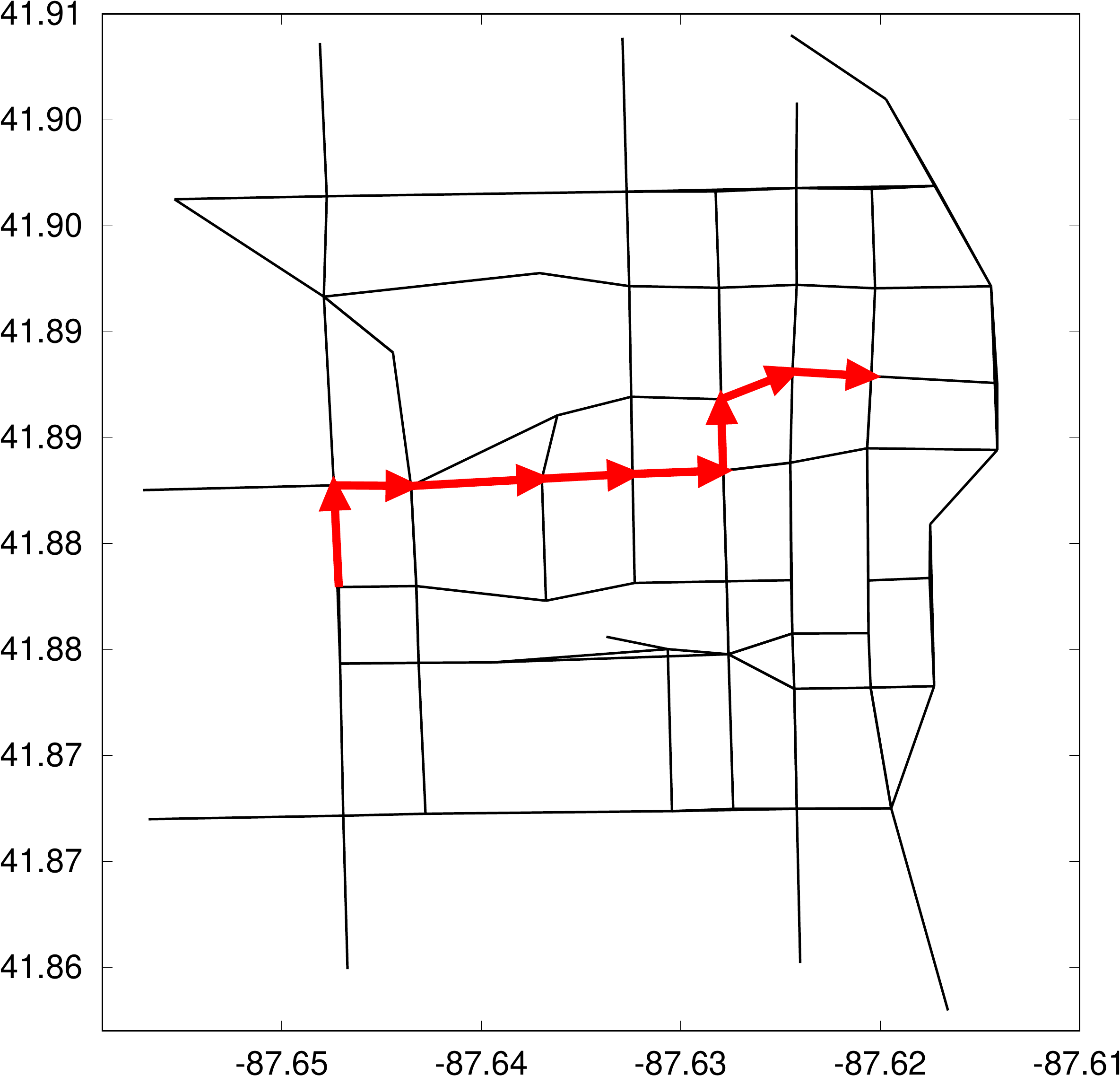}}
\subfigure[Kernel path 03.]{\includegraphics[width=0.27\textwidth]{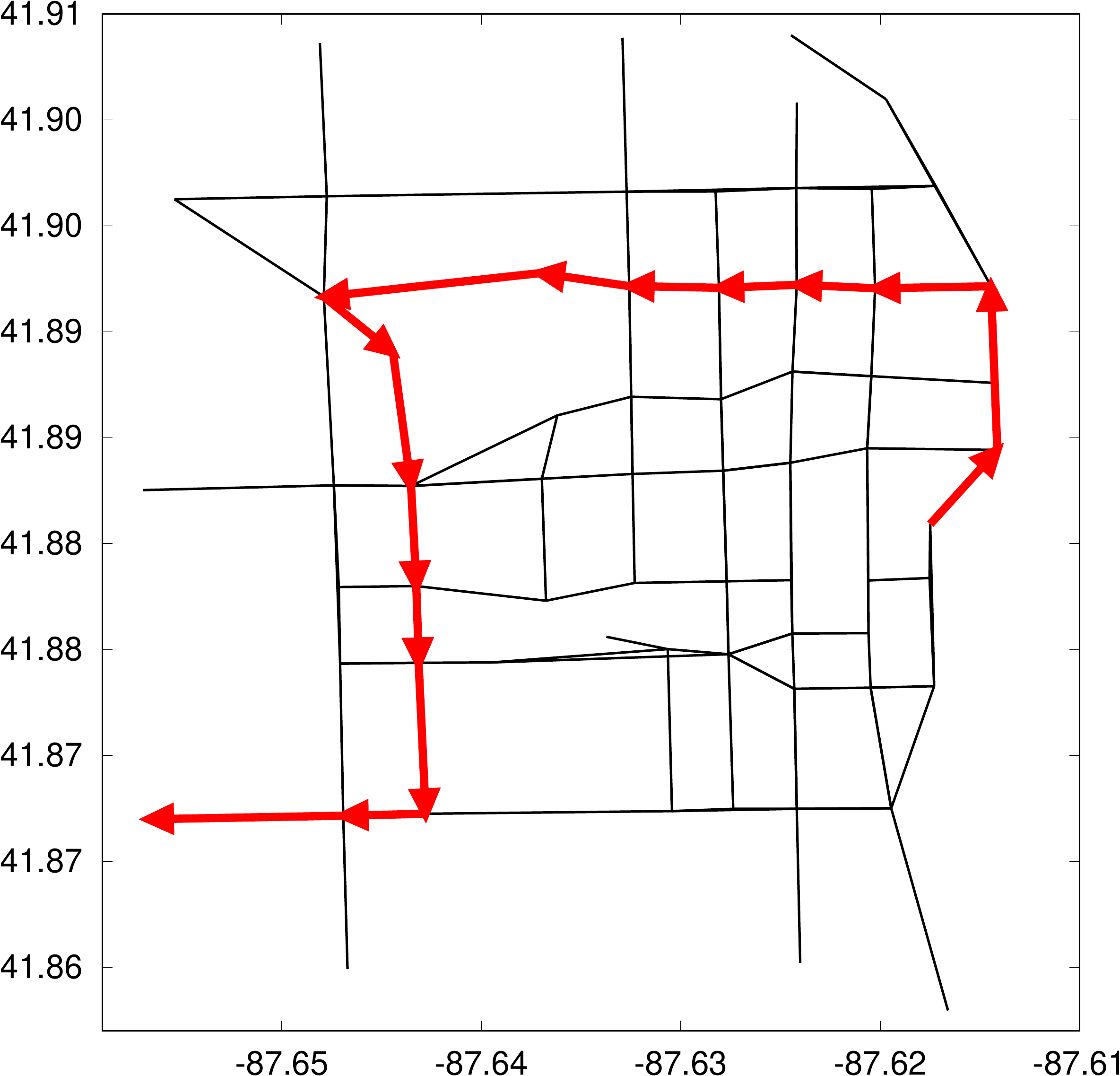}}%
\hspace{0.04\textwidth}%
\subfigure[Out-of-sample path lengths 03.]{\includegraphics[width=0.36\textwidth]{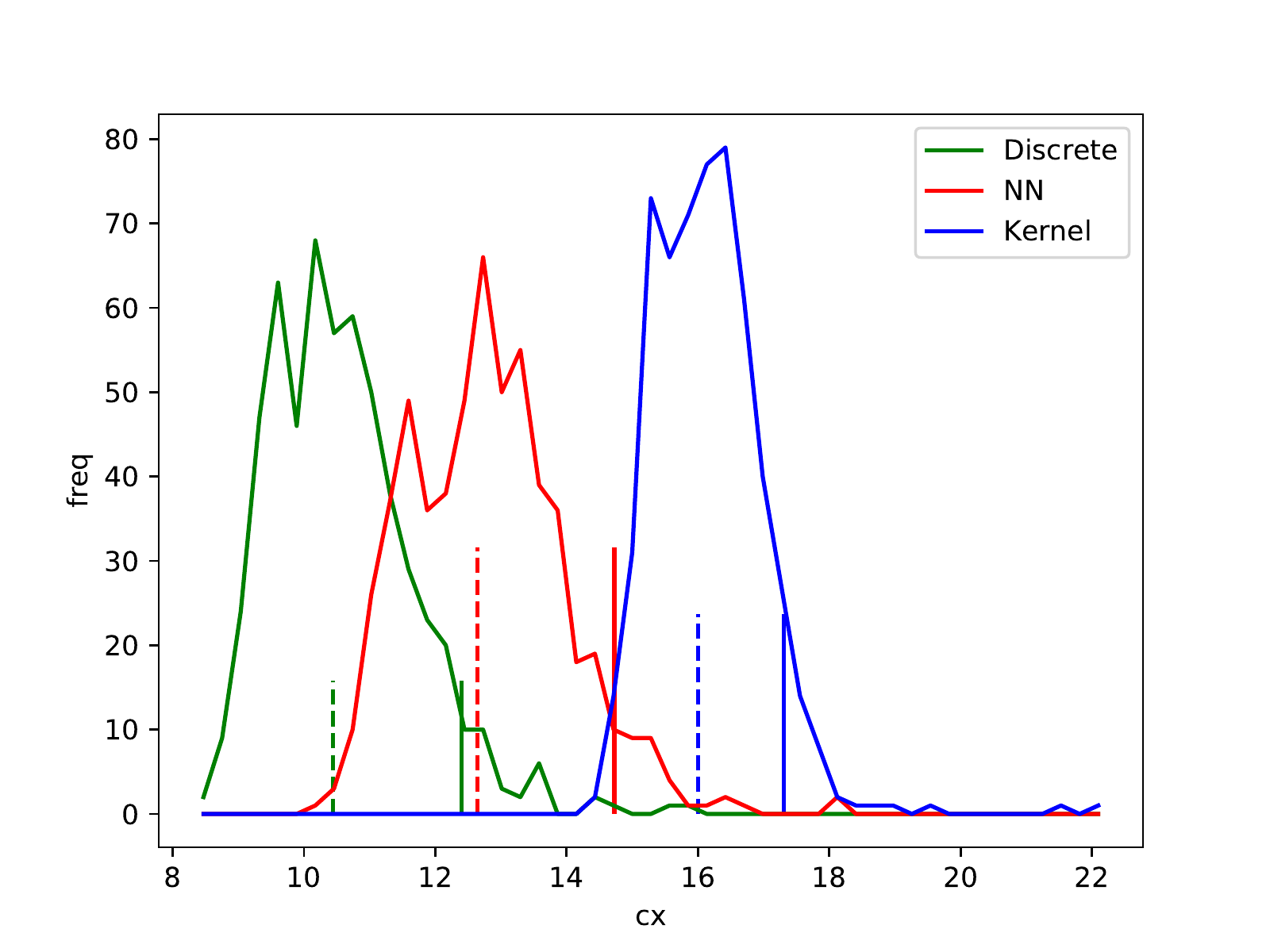}}%
\hspace{0.04\textwidth}%
\subfigure[NN path 03.]{\includegraphics[width=0.27\textwidth]{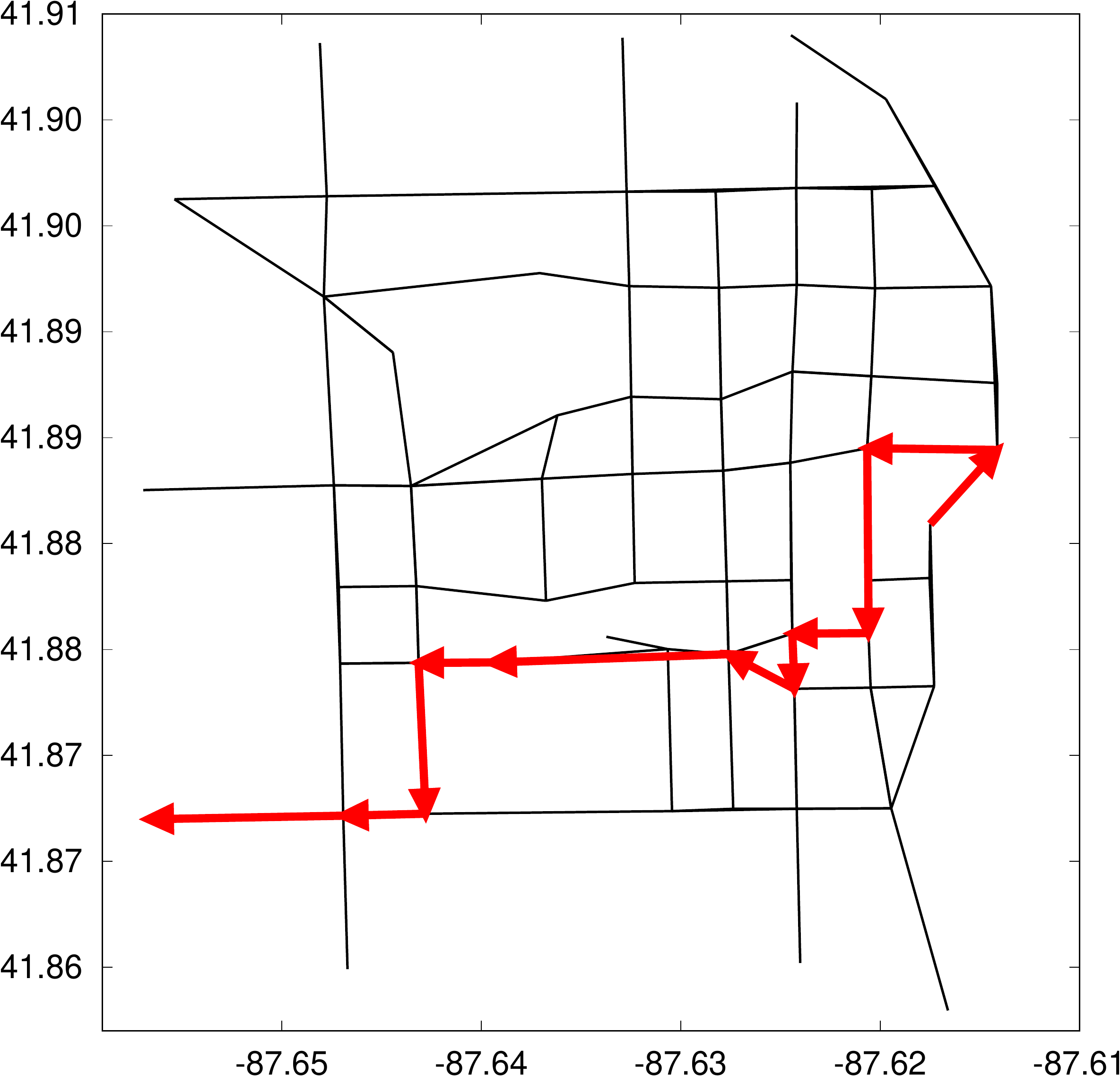}}
\subfigure[Kernel path 04.]{\includegraphics[width=0.27\textwidth]{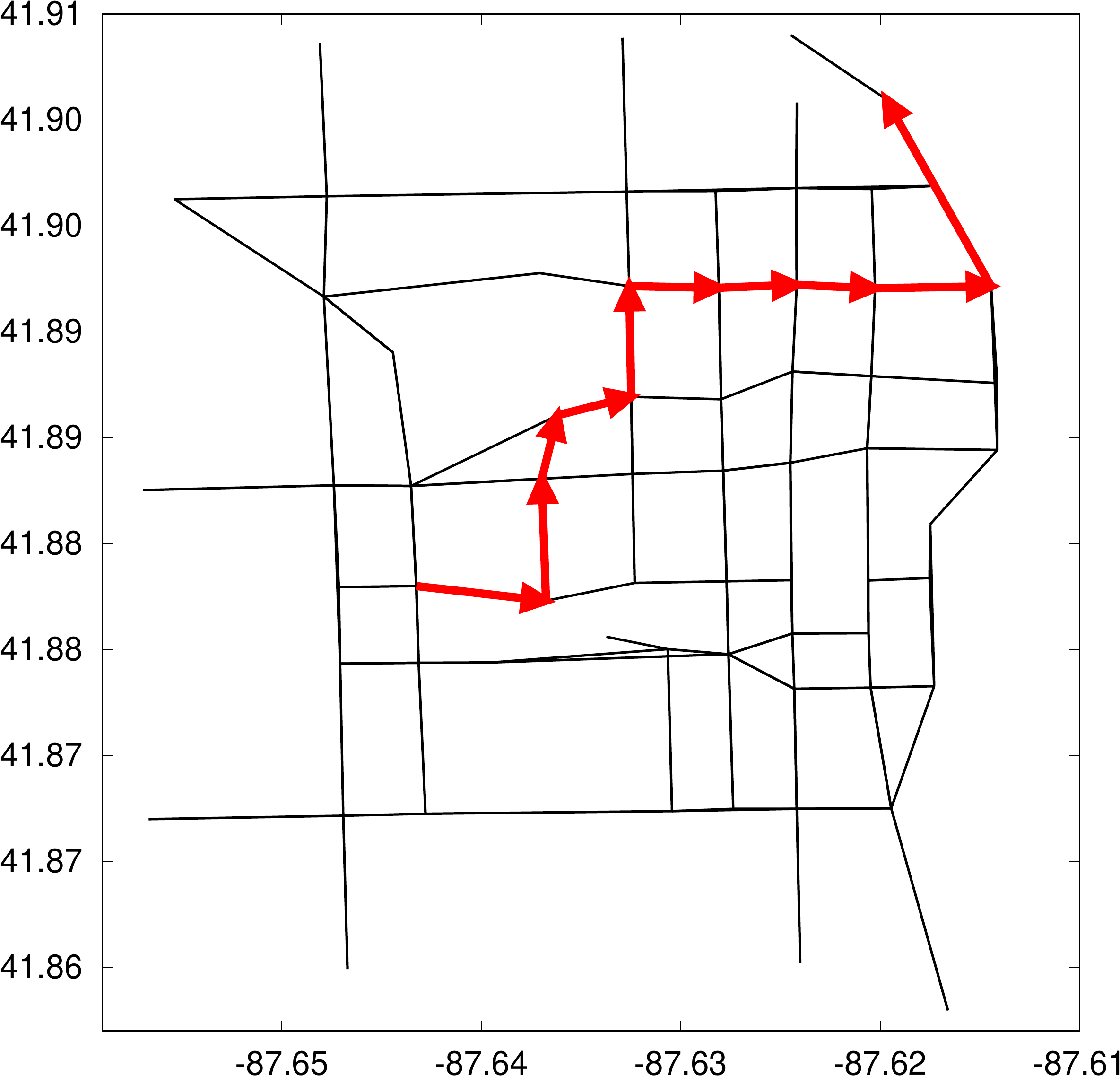}}%
\hspace{0.04\textwidth}%
\subfigure[Out-of-sample path lengths 04.]{\includegraphics[width=0.36\textwidth]{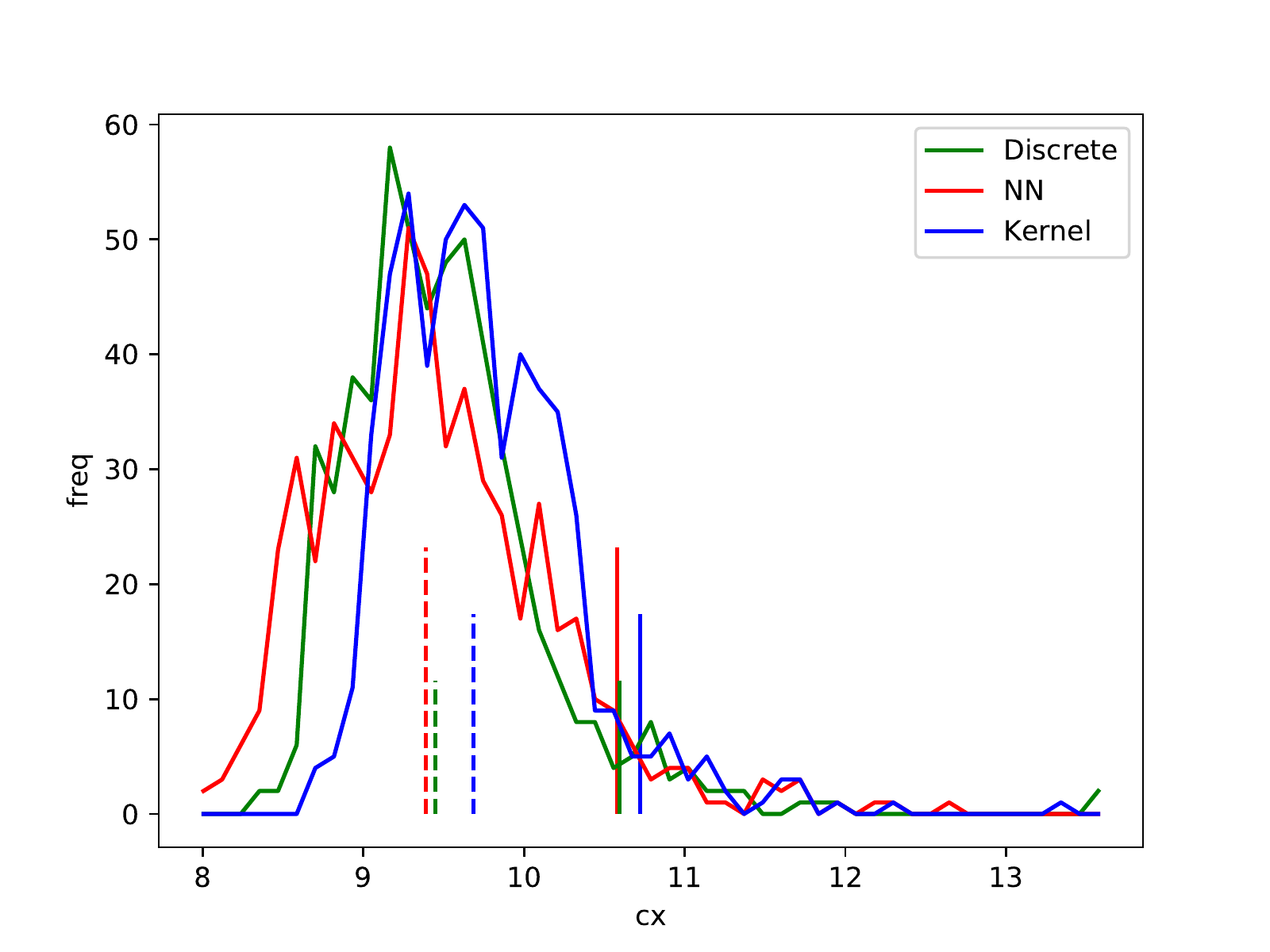}}%
\hspace{0.04\textwidth}%
\subfigure[NN path 04.]{\includegraphics[width=0.27\textwidth]{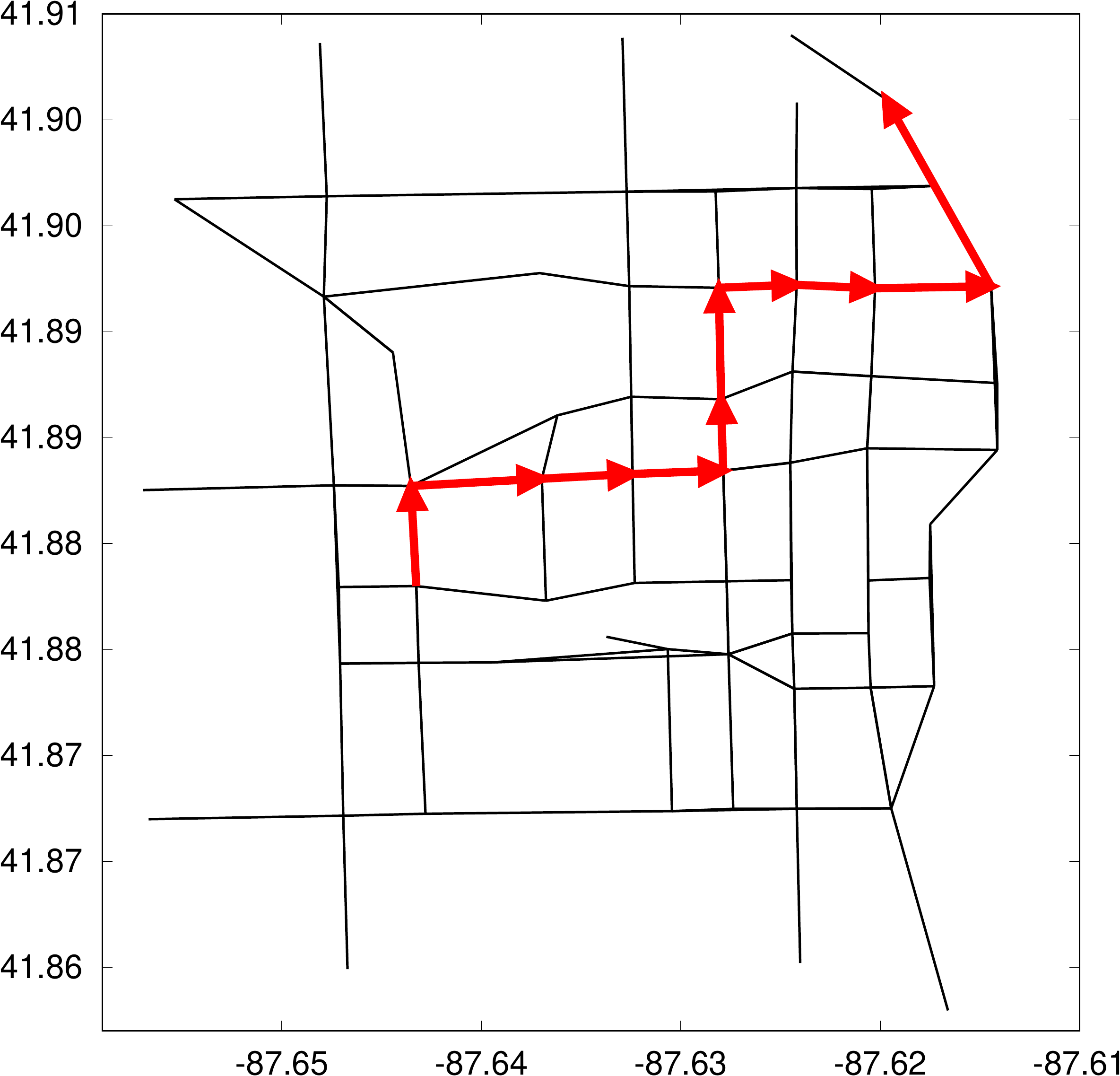}}
\caption{Comparison of paths 1-4 in experiment 3. In Figure~\ref{histo:01}, red line is on top of green line.}\label{exp3:figSP1}
\end{center}
\end{figure}

\begin{figure}[htbp]
\begin{center}
\subfigure[Kernel path 05.]{\includegraphics[width=0.27\textwidth]{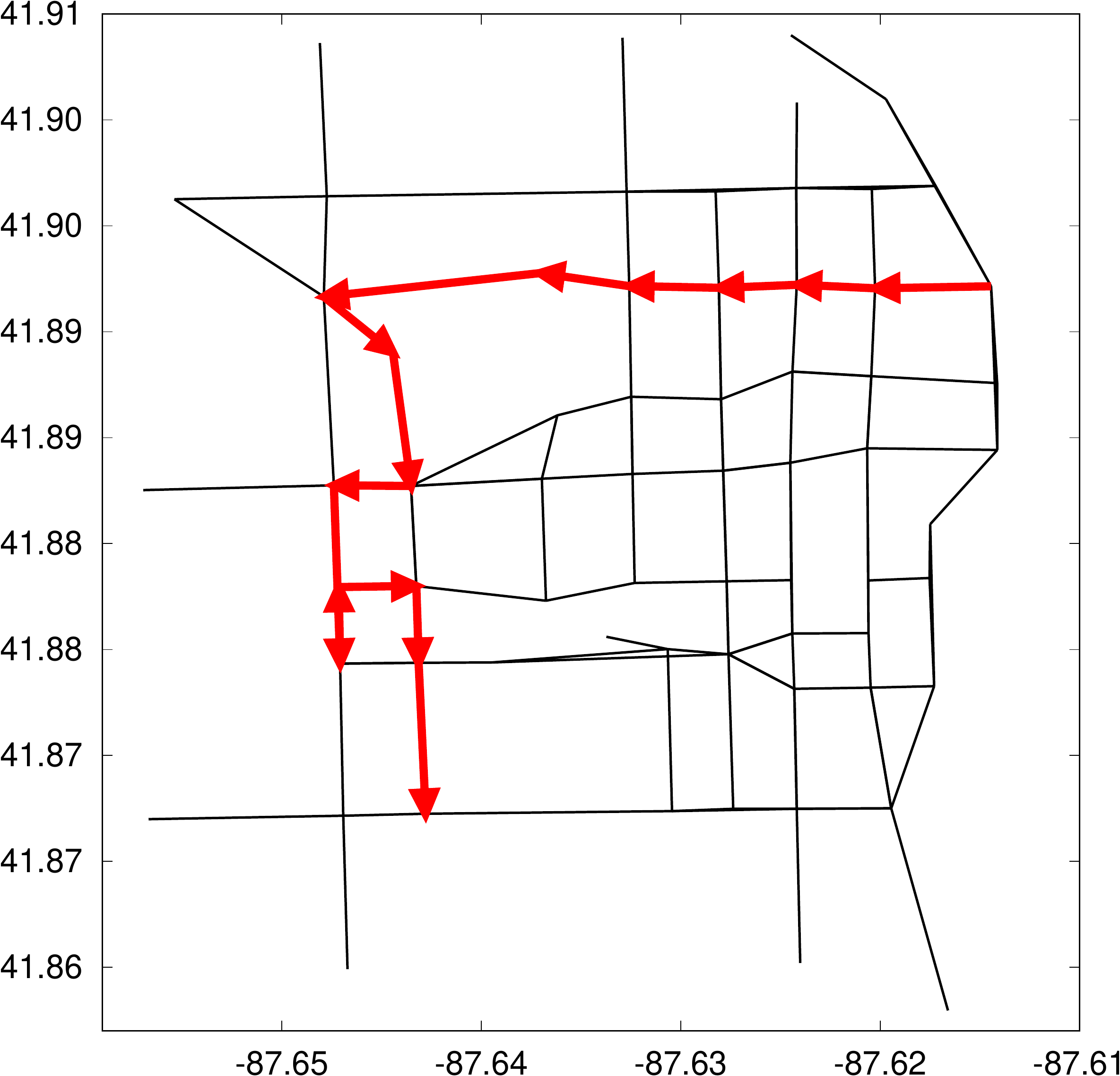}}%
\hspace{0.04\textwidth}%
\subfigure[Out-of-sample path lengths 05.]{\includegraphics[width=0.36\textwidth]{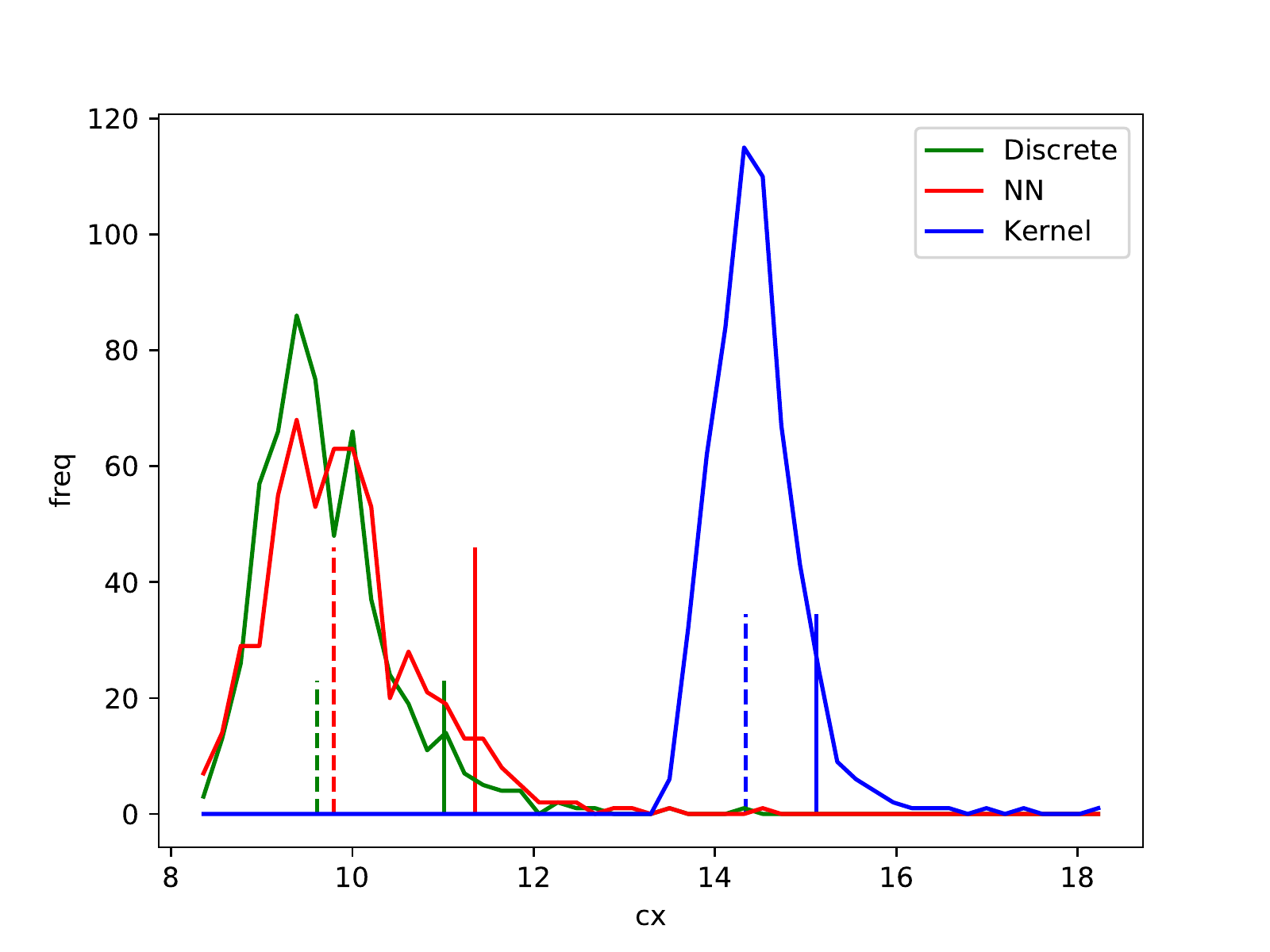}}%
\hspace{0.04\textwidth}%
\subfigure[NN path 05.]{\includegraphics[width=0.27\textwidth]{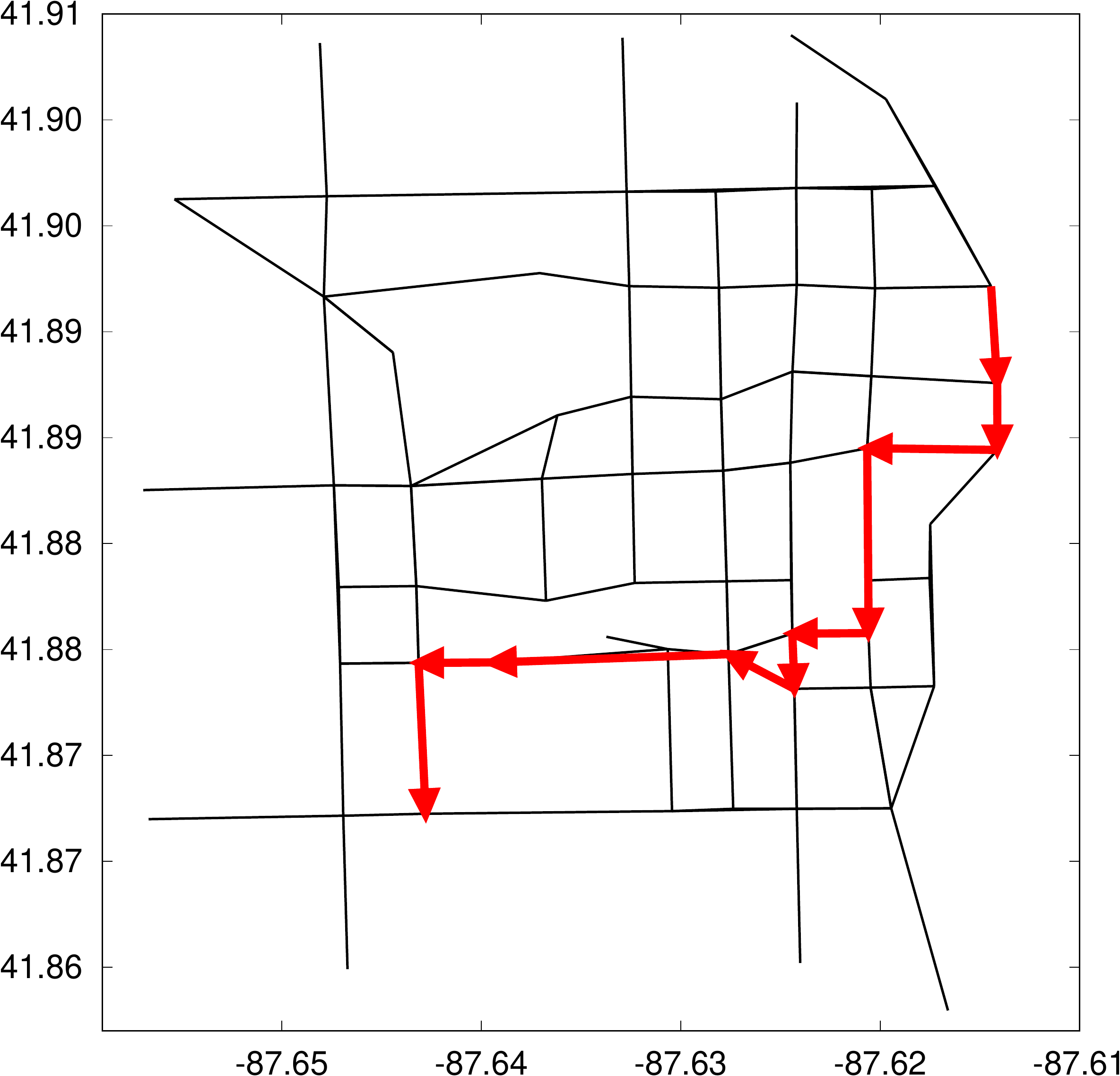}}
\subfigure[Kernel path 06.]{\includegraphics[width=0.27\textwidth]{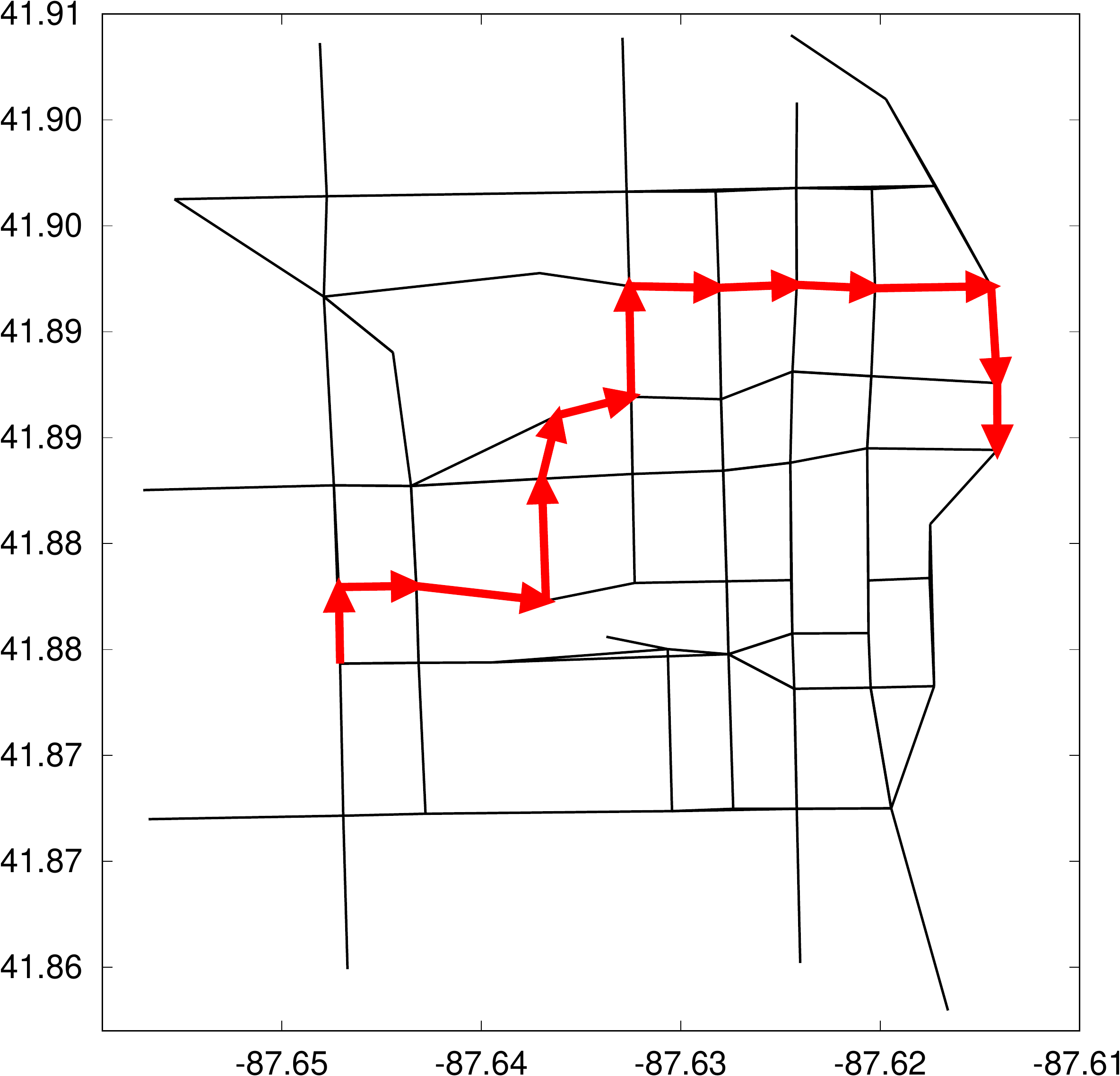}}%
\hspace{0.04\textwidth}
\subfigure[Out-of-sample path lengths 06.]{\includegraphics[width=0.36\textwidth]{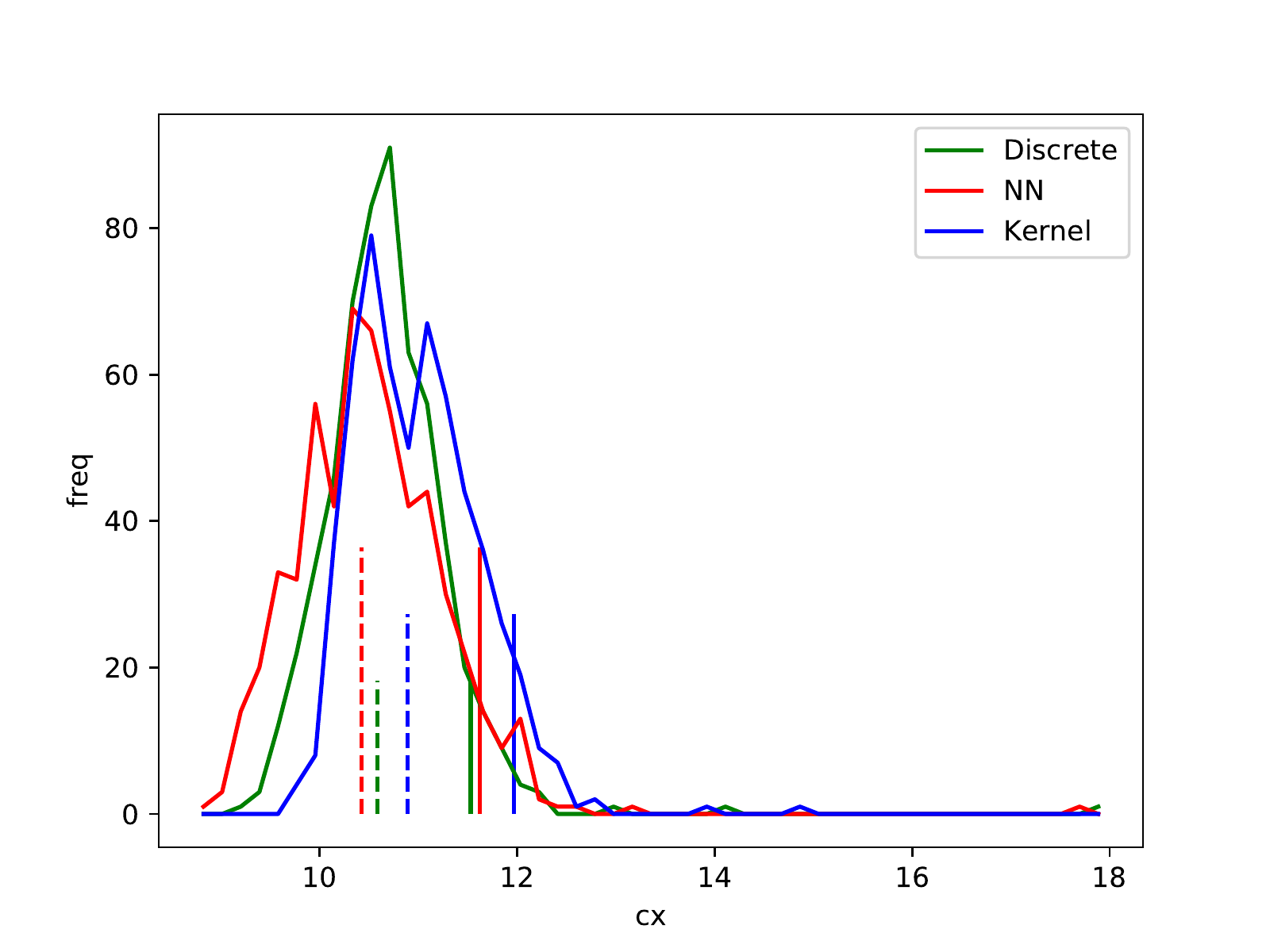}}%
\hspace{0.04\textwidth}%
\subfigure[NN path 06.]{\includegraphics[width=0.27\textwidth]{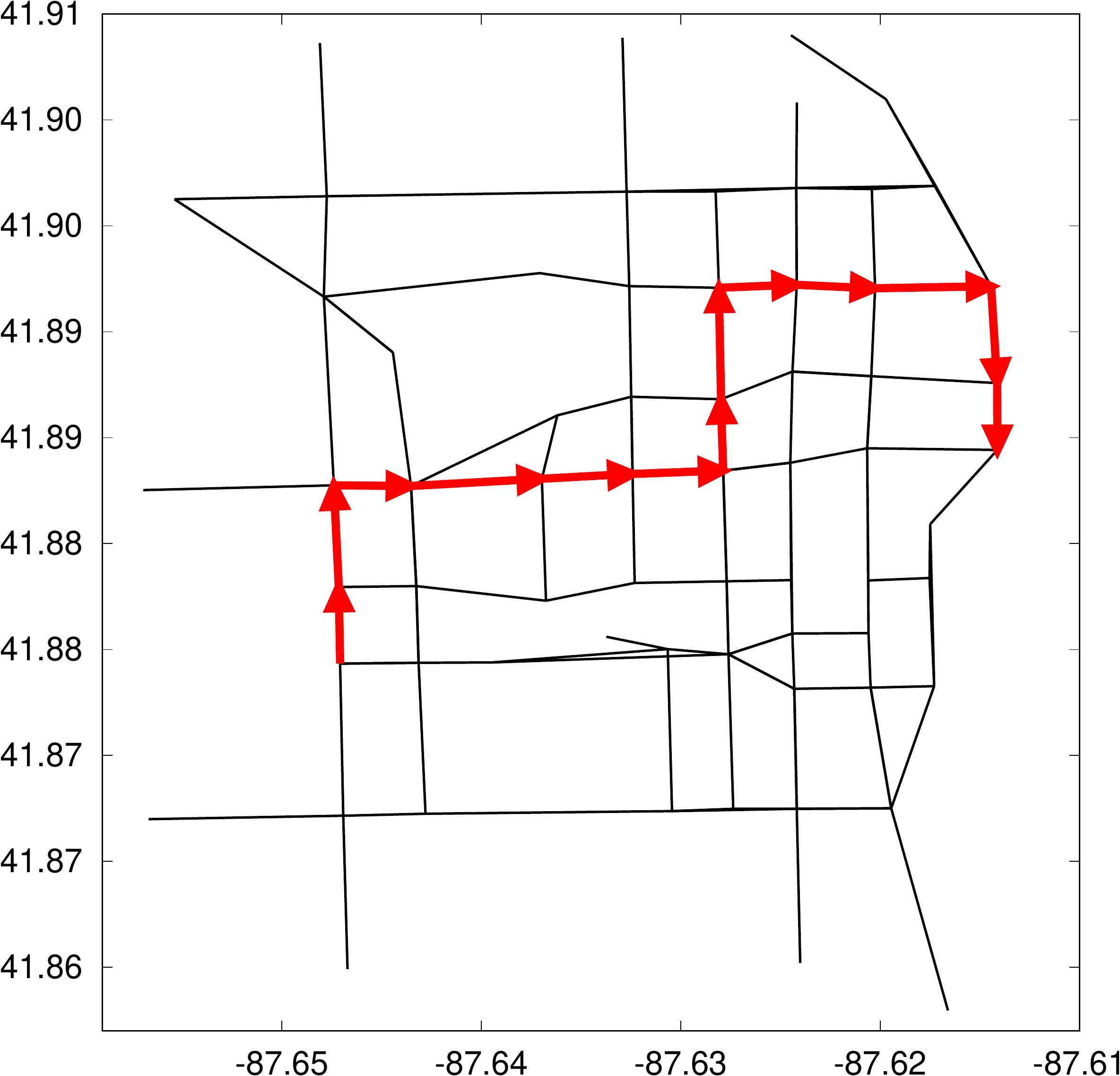}}
\subfigure[Kernel path 07.]{\includegraphics[width=0.27\textwidth]{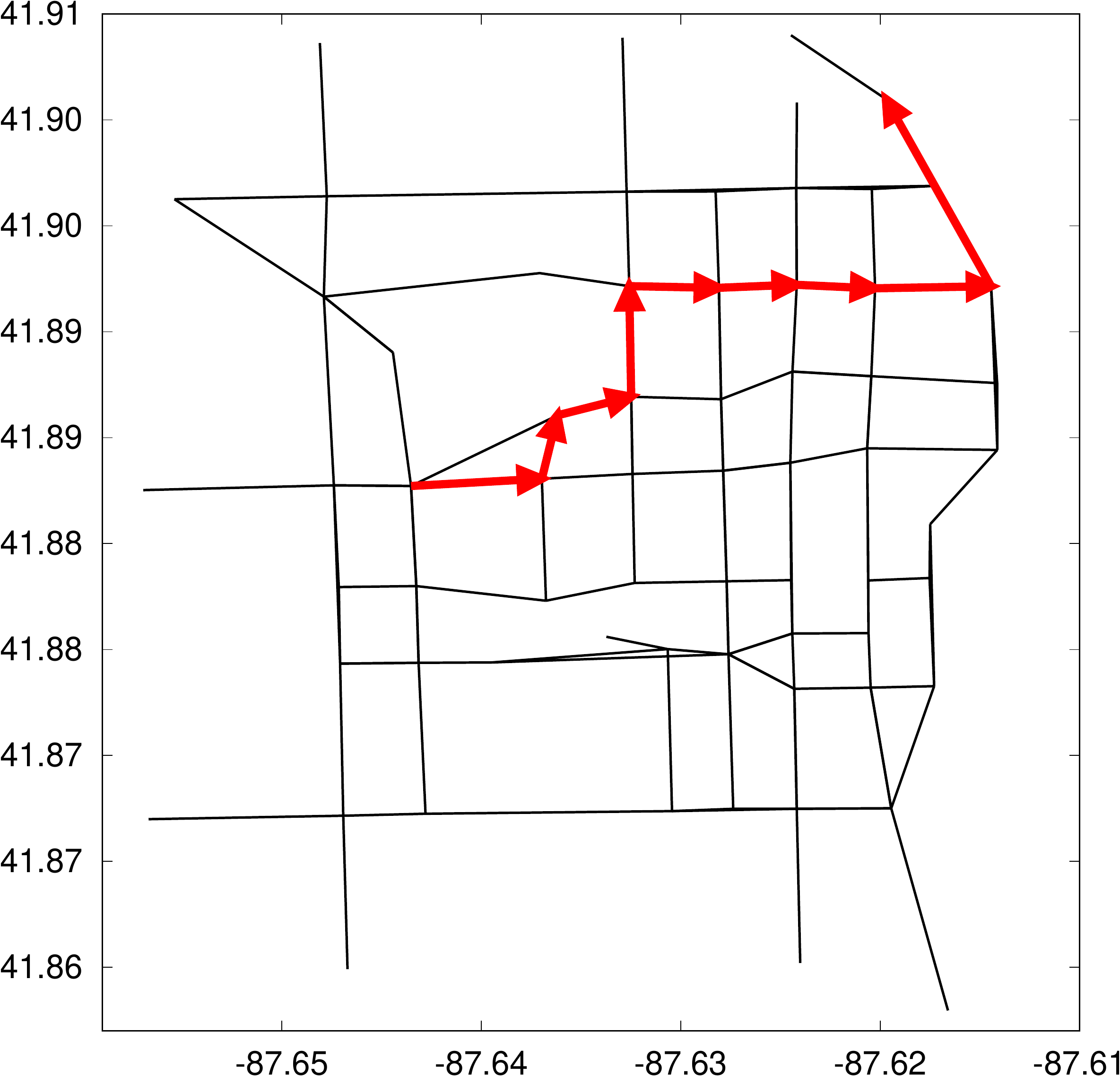}}%
\hspace{0.04\textwidth}%
\subfigure[Out-of-sample path lengths 07.]{\includegraphics[width=0.36\textwidth]{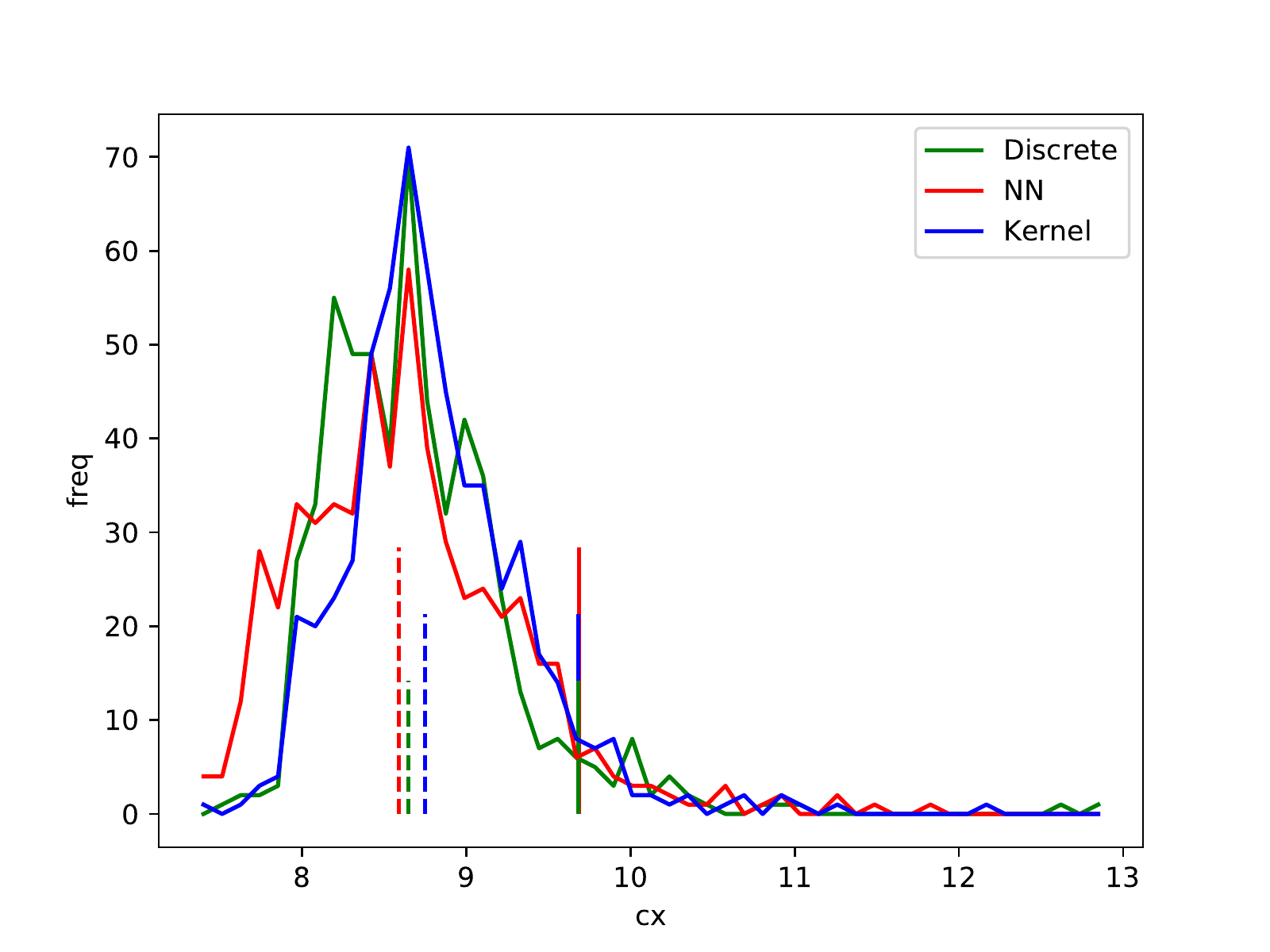}}%
\hspace{0.04\textwidth}%
\subfigure[NN path 07.]{\includegraphics[width=0.27\textwidth]{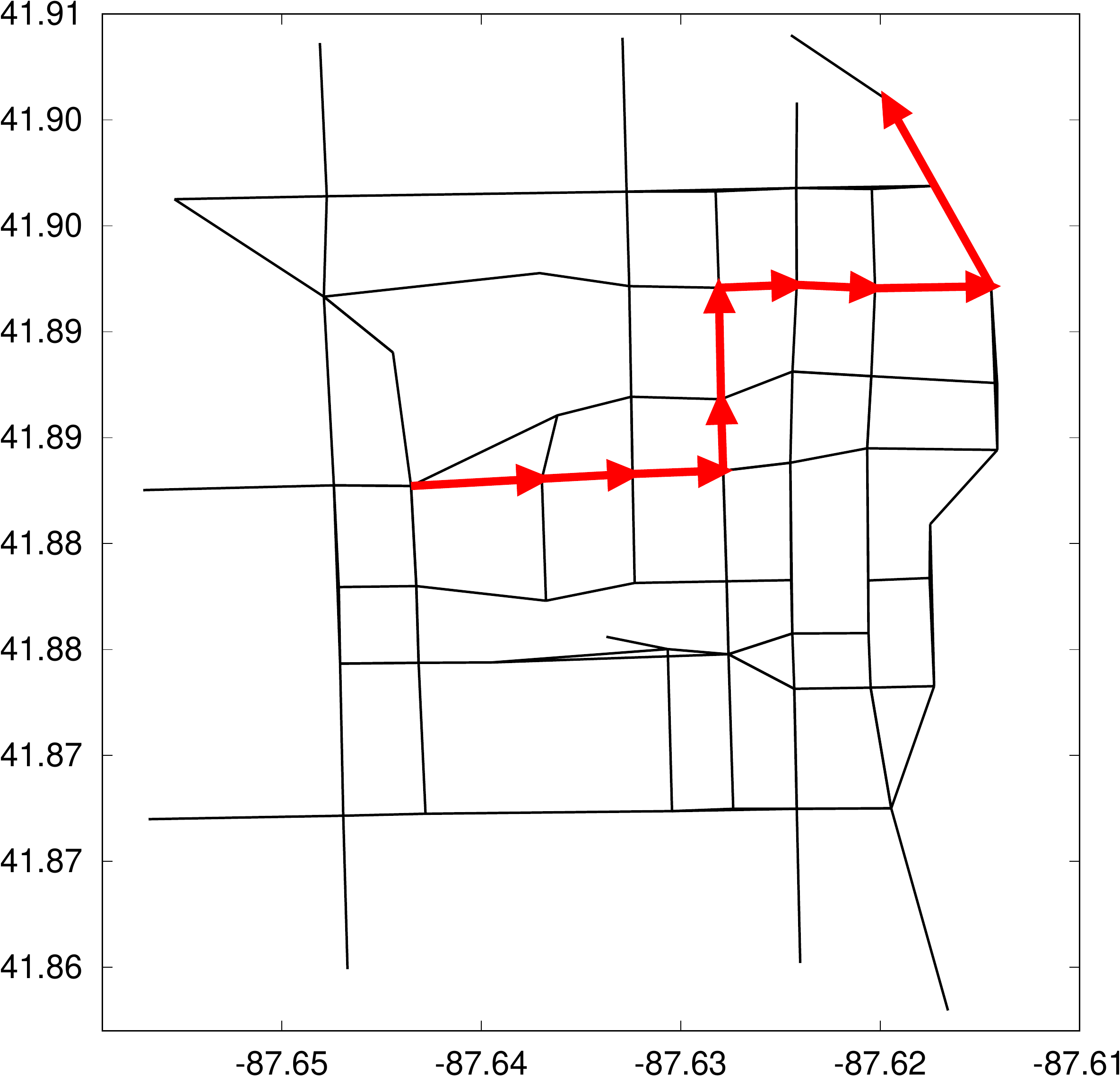}}
\caption{Comparison of paths 5-7 in experiment 3.}\label{exp3:figSP2}
\end{center}
\end{figure}

\begin{figure}[htbp]
\begin{center}
\subfigure[Kernel path 08.]{\includegraphics[width=0.27\textwidth]{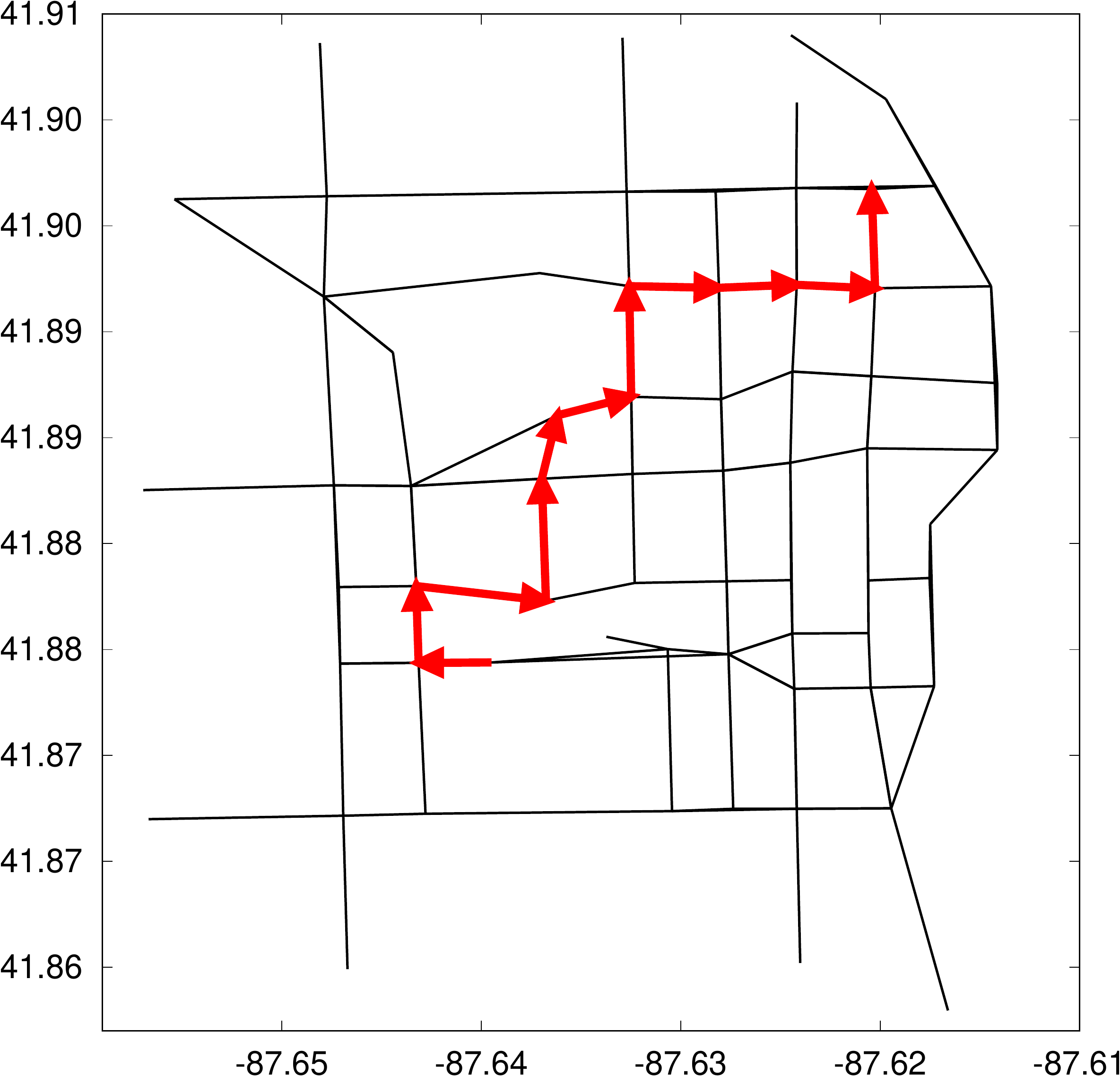}}%
\hspace{0.04\textwidth}%
\subfigure[Path lengths 08.]{\includegraphics[width=0.36\textwidth]{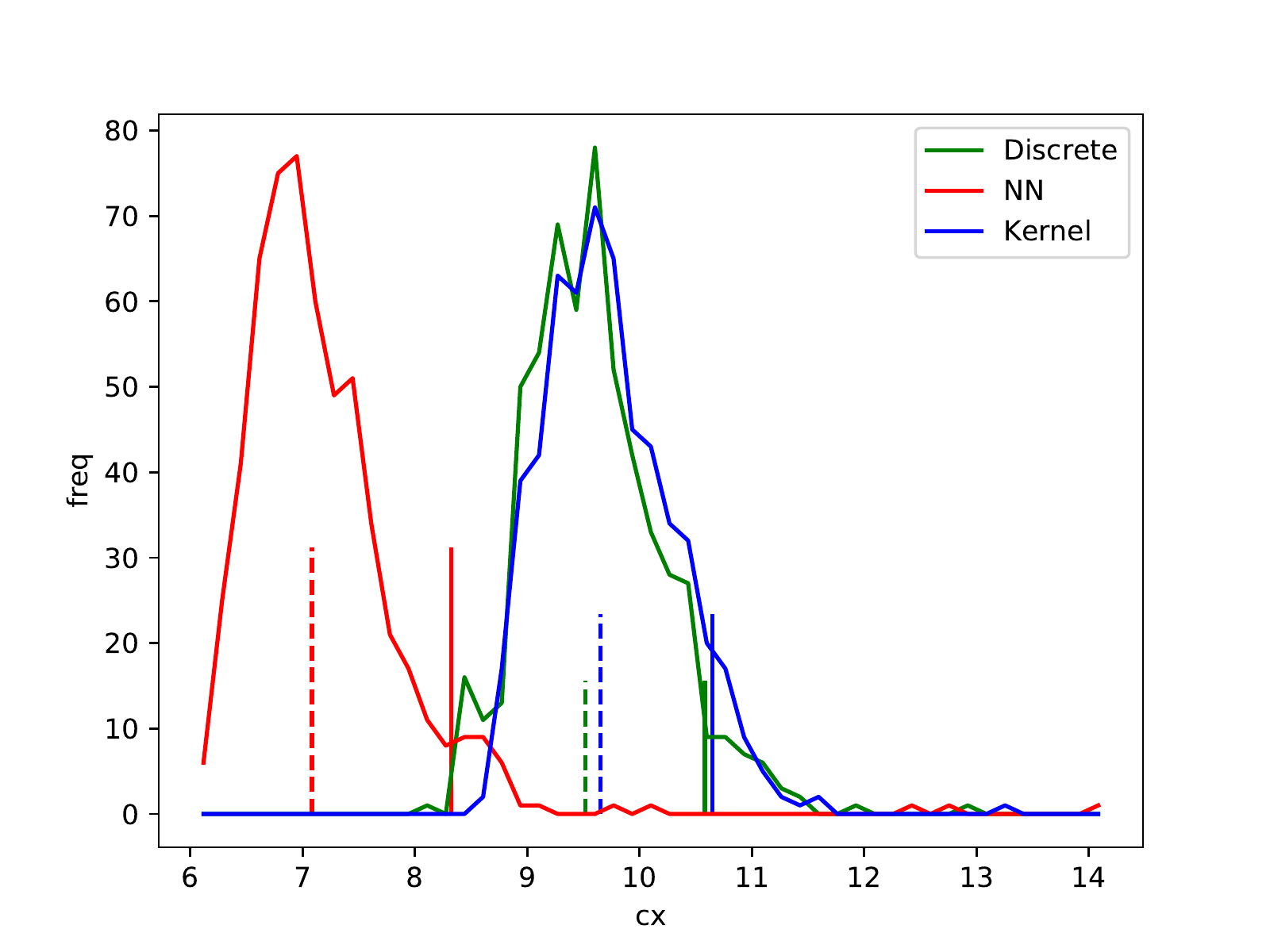}}%
\hspace{0.04\textwidth}%
\subfigure[NN path 08.]{\includegraphics[width=0.27\textwidth]{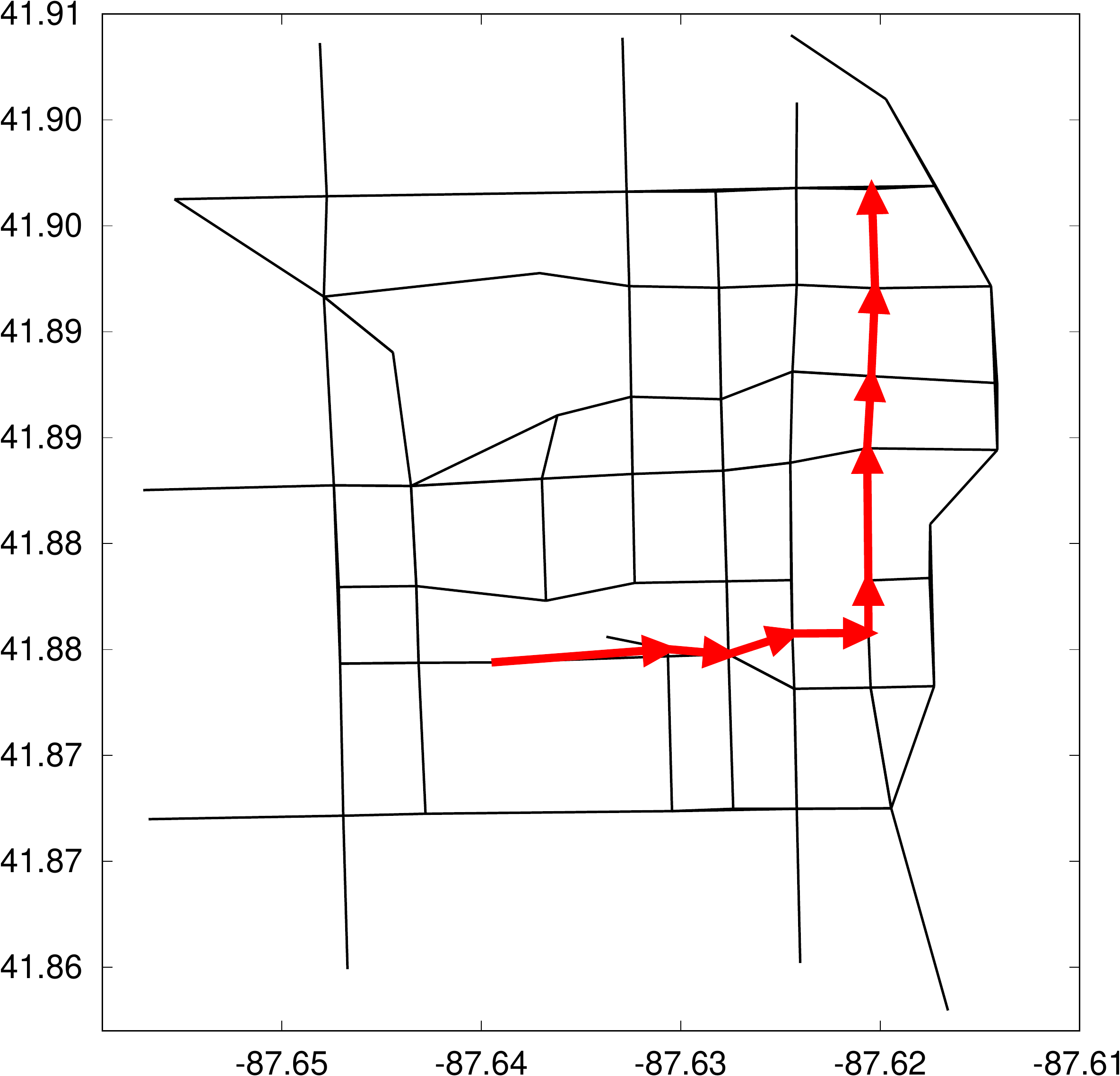}}
\subfigure[Kernel path 09.]{\includegraphics[width=0.27\textwidth]{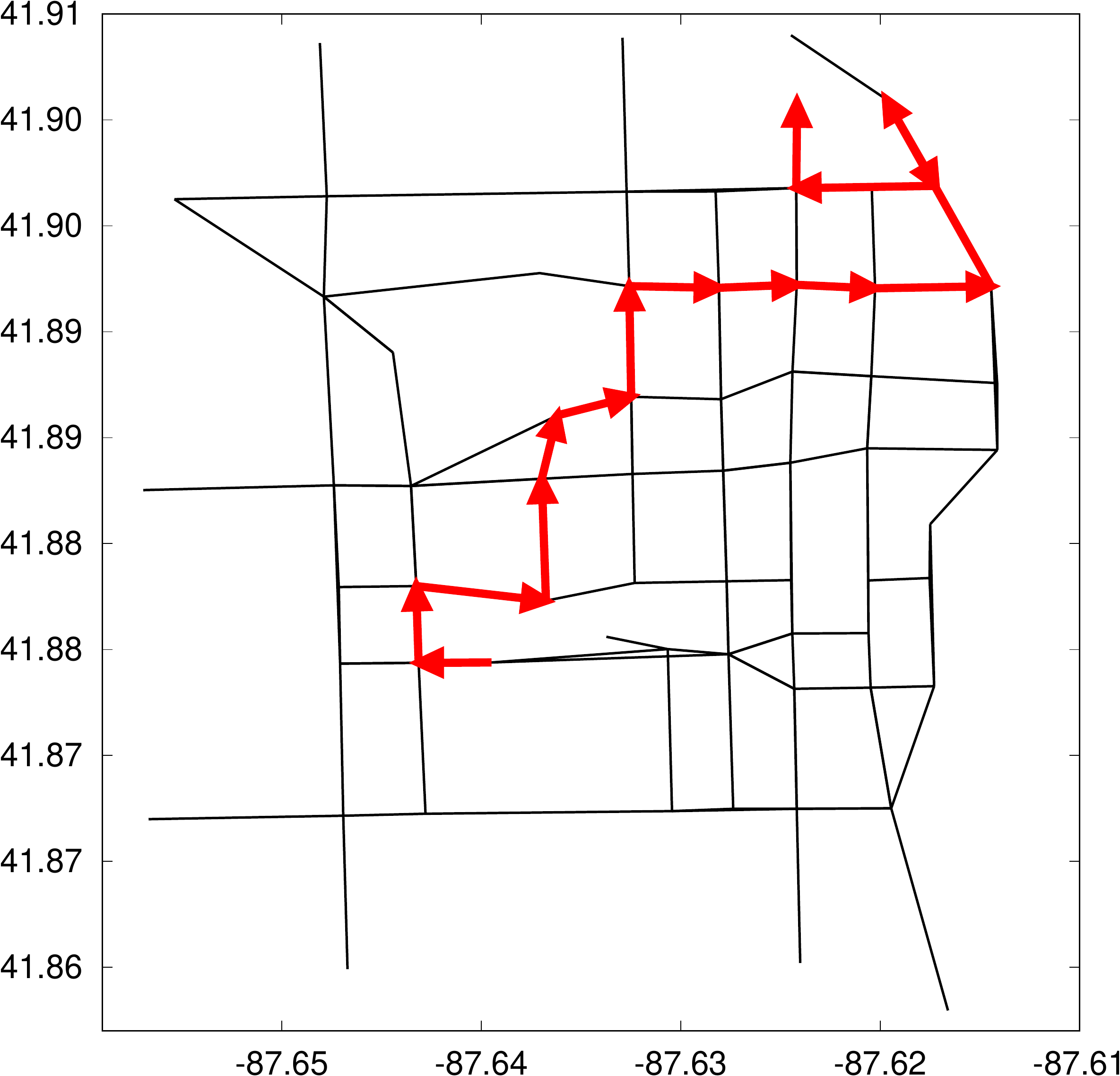}}%
\hspace{0.04\textwidth}
\subfigure[Path lengths 09.]{\includegraphics[width=0.36\textwidth]{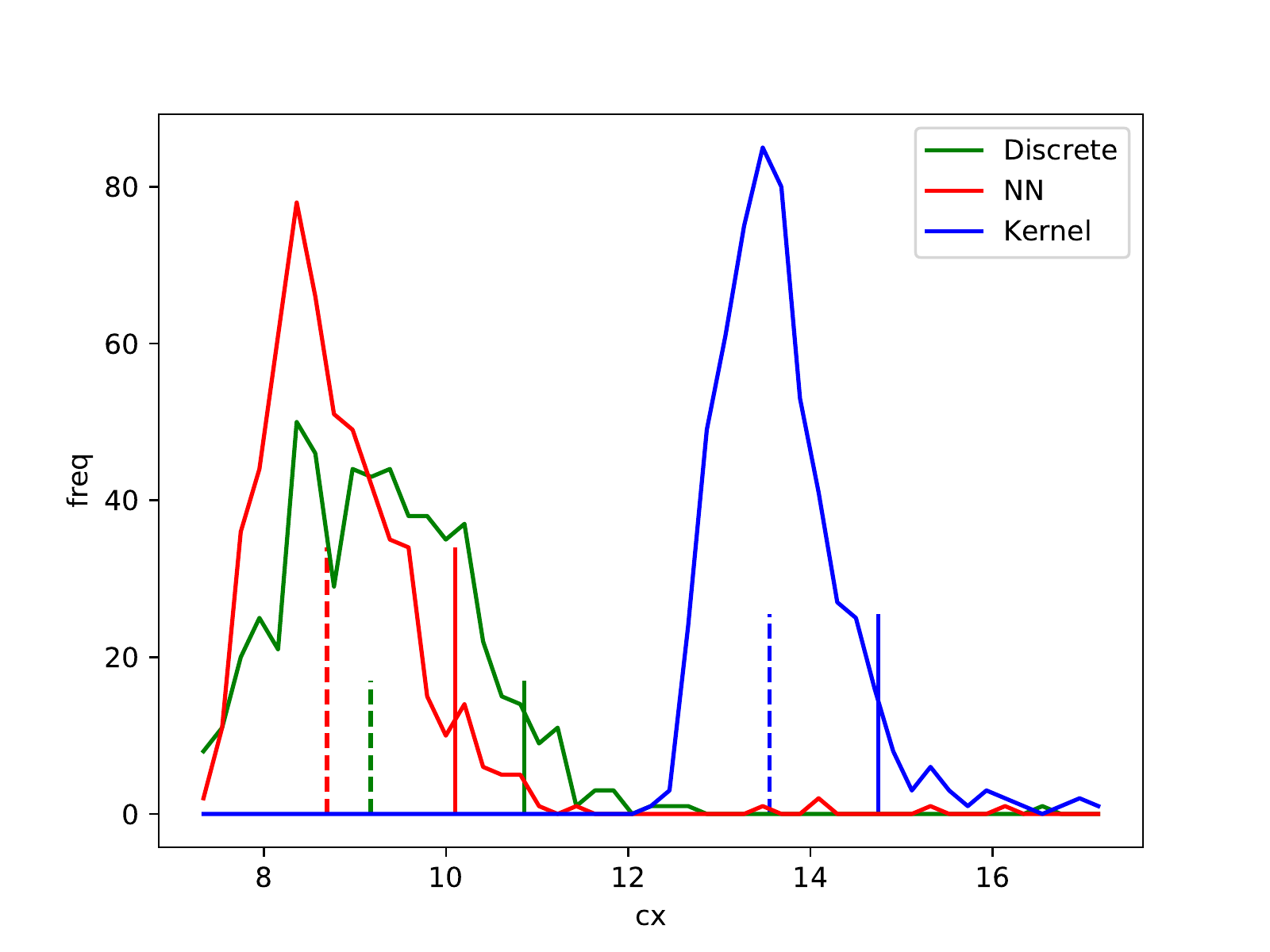}}%
\hspace{0.04\textwidth}%
\subfigure[NN path 09.]{\includegraphics[width=0.27\textwidth]{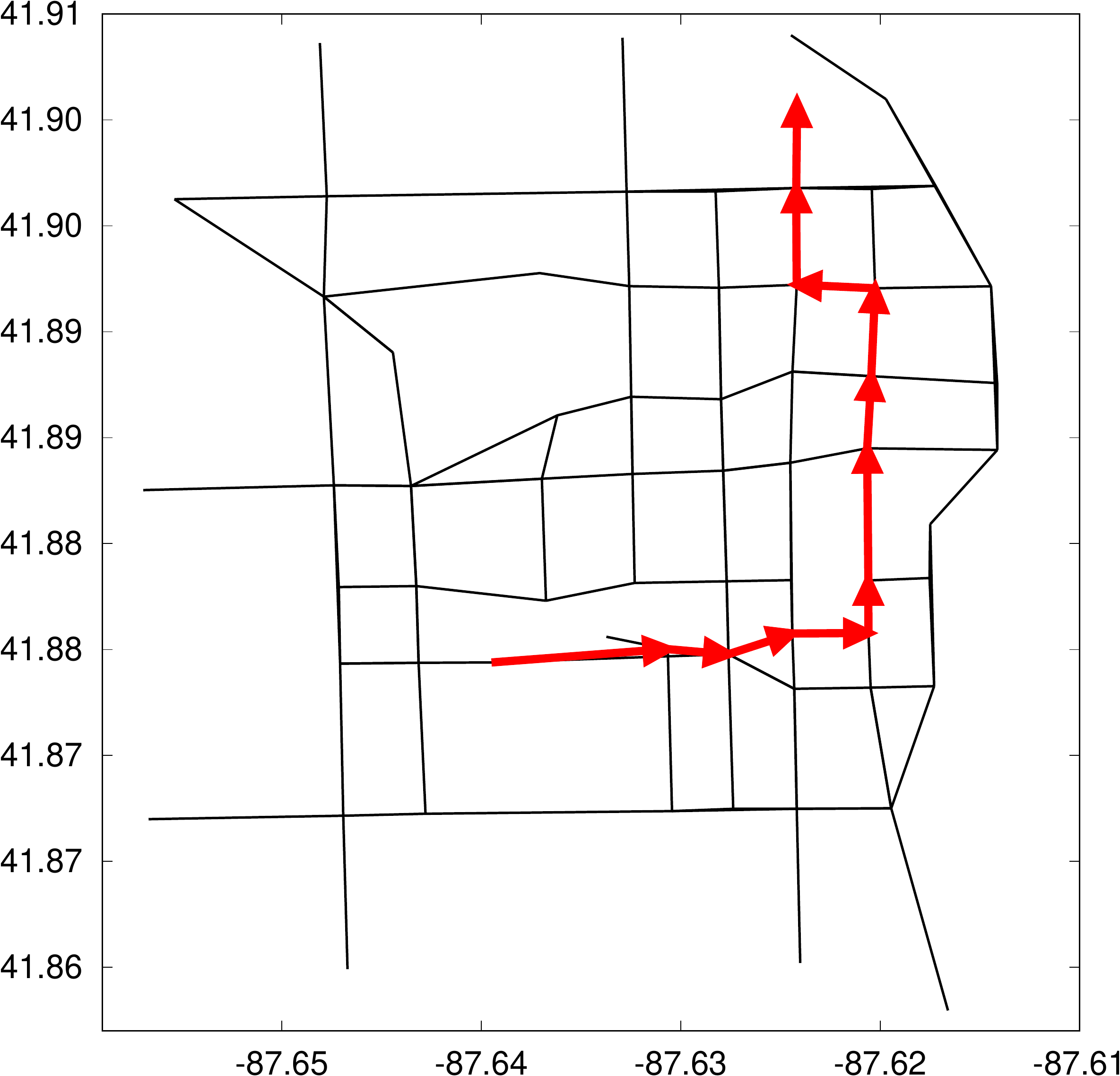}}
\subfigure[Kernel path 10.]{\includegraphics[width=0.27\textwidth]{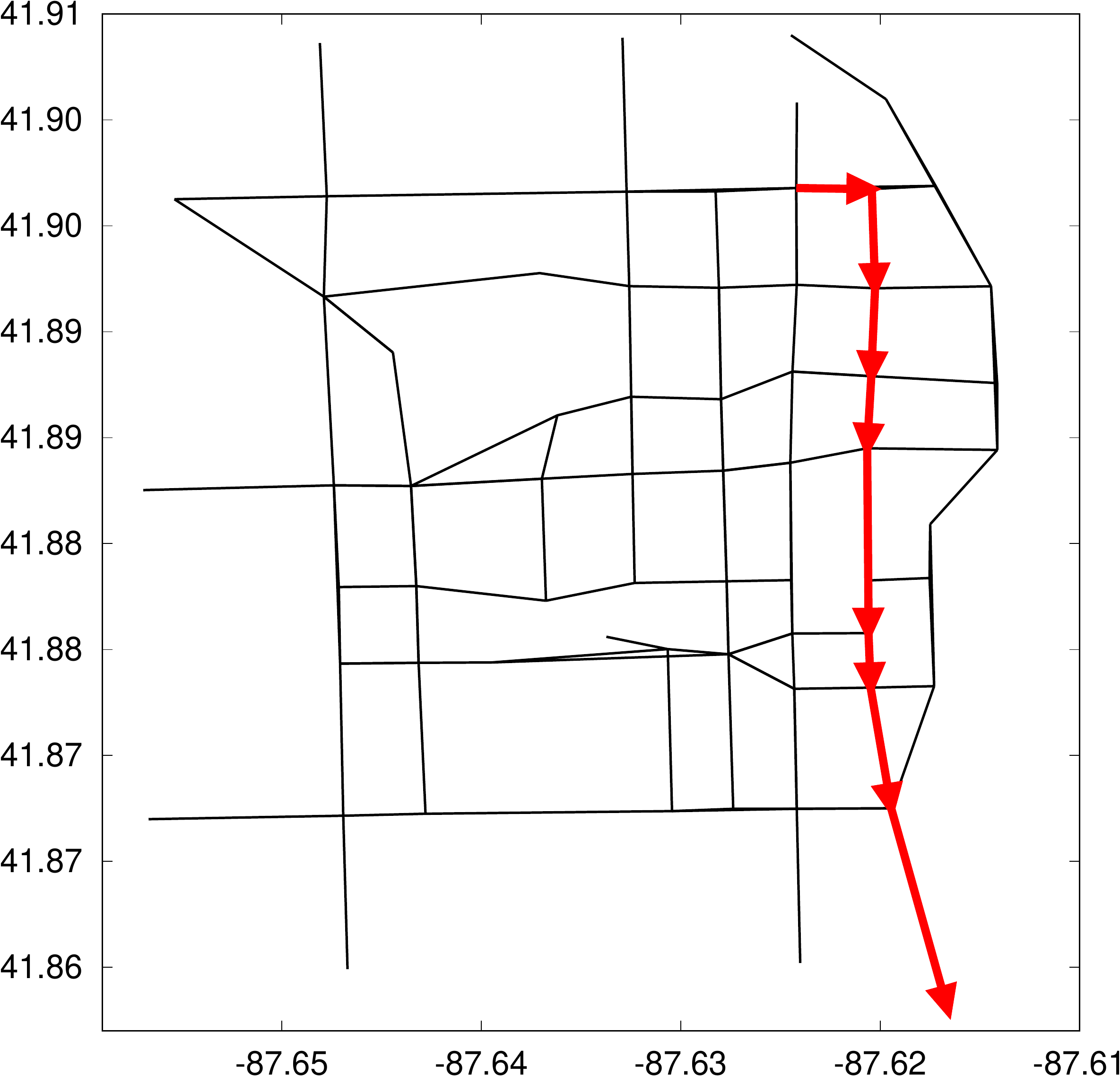}}%
\hspace{0.04\textwidth}%
\subfigure[Path lengths 10.\label{histo:10}]{\includegraphics[width=0.36\textwidth]{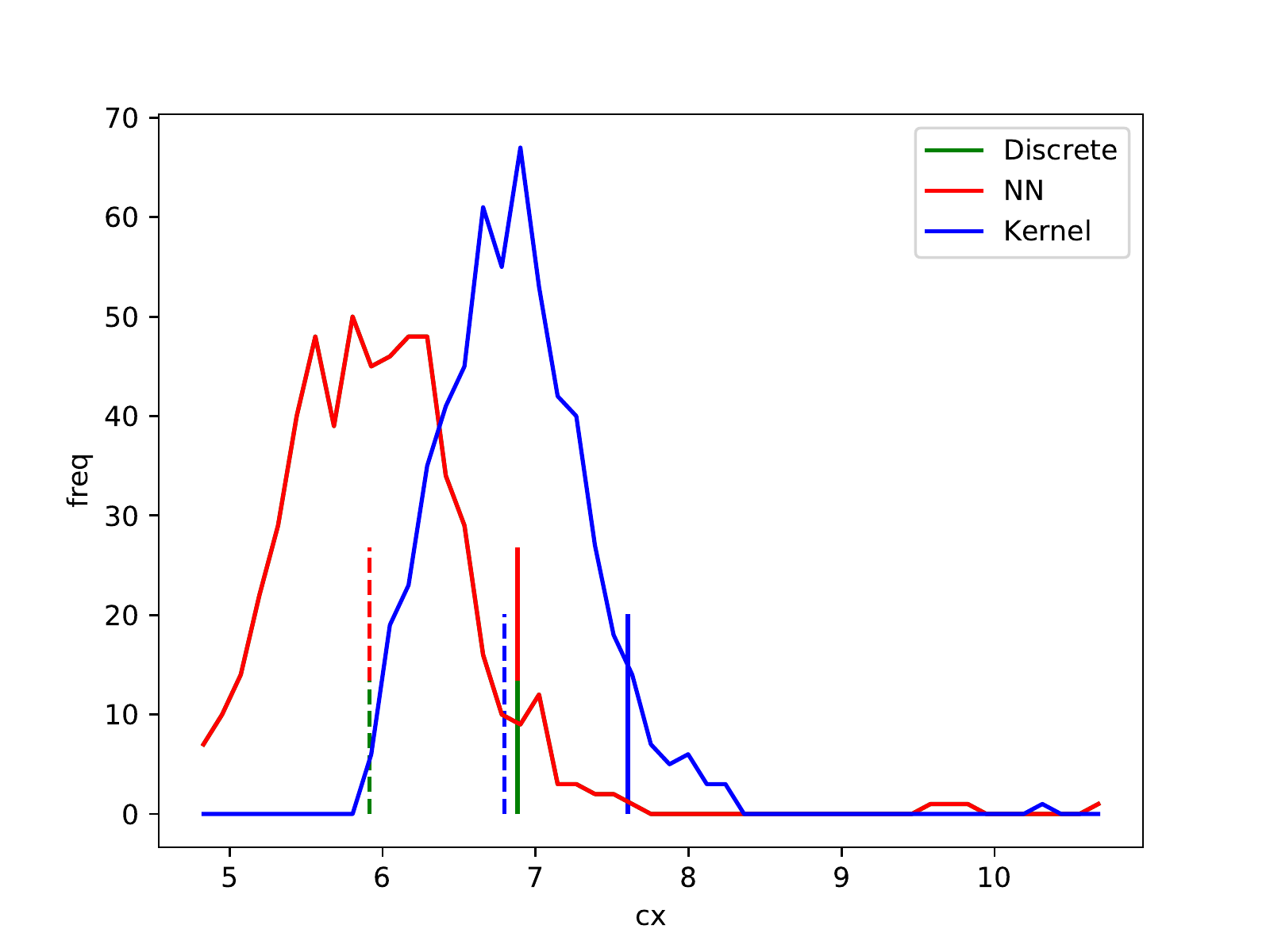}}%
\hspace{0.04\textwidth}%
\subfigure[NN path 10.]{\includegraphics[width=0.27\textwidth]{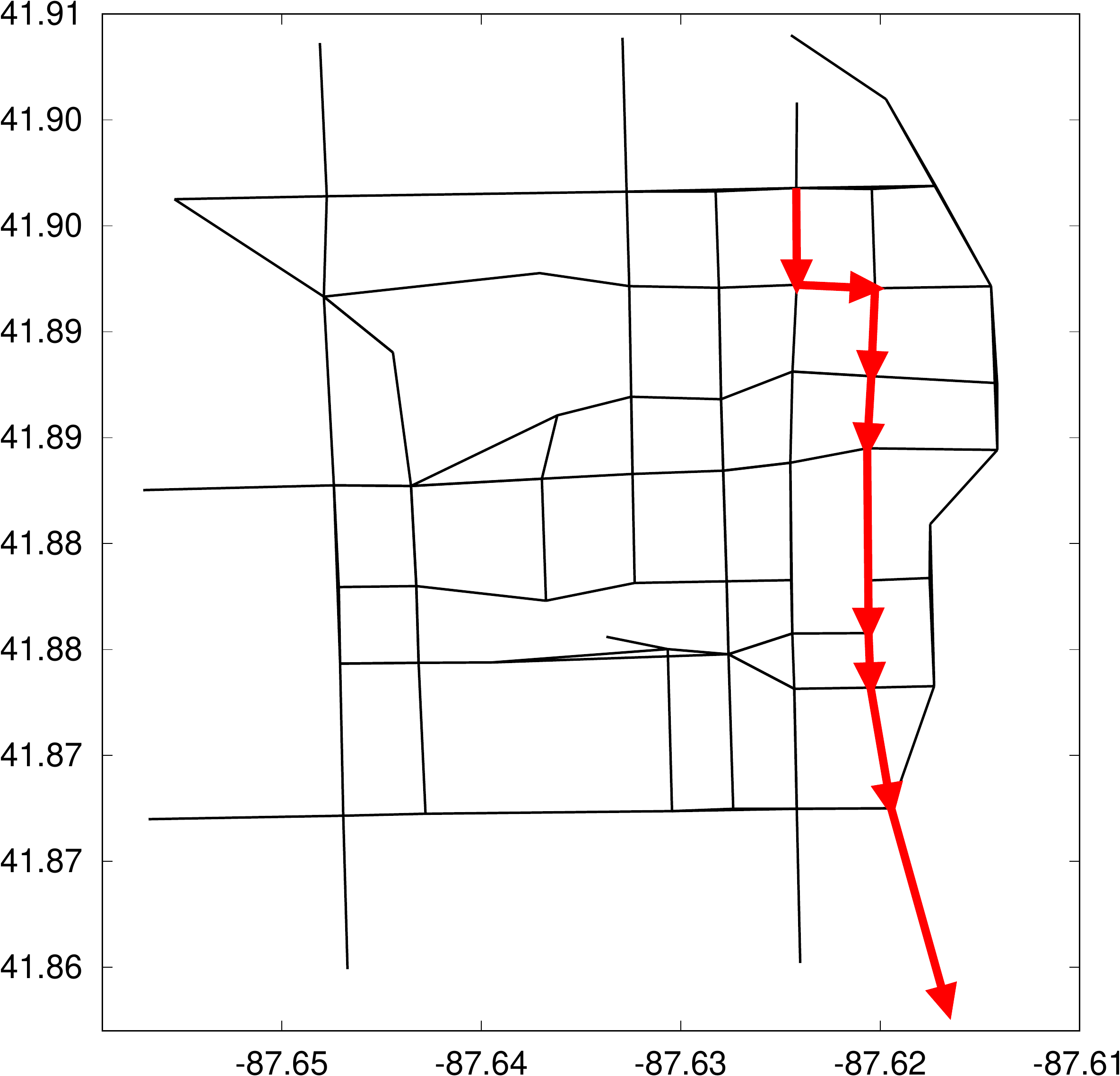}}
\caption{Comparison of paths 8-10 in experiment 3. In Figure~\ref{histo:10}, red line is on top of green line.}\label{exp3:figSP3}
\end{center}
\end{figure}

We note that classical flow constraints (inflow equals outflow for each node except source and sink) correctly model the path problem only if all costs are non-negative. Using the Kernel approach, we observed that solutions frequently were not connected and included cycles, which results from uncertainty sets that include negative costs. We modified the Kernel approach such that only non-negative arc lengths are contained in the uncertainty set. In two instances (paths 05 and 09) it can be observed that there still exists a single edge that is traversed twice in an unnecessary loop. This can only happen if the Kernel approach assumes that such costs are equal to zero. Apart from these two instances, all paths (both Kernel and NN) seem reasonable options in the street network. Note that in no case we observe the same path from Kernel and NN.

Comparing the out-of-sample performance of Kernel and NN, there are five instances with mixed results (paths 01, 04, 06, 07, and 10). For these solutions the histograms largely overlap and paths contain similar edges. In case of paths 01 and 10, Kernel shows better average and better $95\%$ quantile performance. For paths 04, 06, and 07, this comparison gives NN the advantage. For the other five paths (02, 03, 05, 08, and 09), NN has a distinct advantage, clearly outperforming Kernel.

There is one case (path 03) where Discrete clearly outperforms both Kernel and NN. Apart from this case, Discrete shows a good overall performance, but can sometimes be worse than NN, e.g., on paths 02 and 08, which highlights the benefit of generalization through the proposed machine learning methods when the data is less affected by noise.

To summarize the findings of the two experiments, we note that (i) when visualized on low-dimensional data, NN uncertainty sets reasonably capture the shape and size of uncertainty; (ii) NN uncertainty sets on random, 20-dimensional optimization problems result in significantly better solutions than what Kernel achieves, irrespective whether we consider objective or constraint uncertainty, which comes at the cost of higher computational effort; and (iii) using real-world data, NN becomes more computationally efficient, scaling better in the problem size than Kernel, but still produces solutions that clearly outperform those produced by Kernel. Finally, (iv) for relatively noise-free data, the discrete scenario approach can produce solutions of high quality that are fast to compute. 

\section{Conclusions}

To derive useful robust optimization models, it is central to have a suitable description of the uncertainty set available. The task of identifying whether a given scenario is similar to observed data or should be considered as an outlier is related to the task of describing the uncertainty set, and is a typical problem for which machine learning techniques have been proven to be highly efficient.

In this paper we combined one-class deep neural networks, to describe the uncertainty set, with robust optimization models. It turns out that the uncertainty sets created by the one-class deep neural networks are a finite union of convex sets, each set being a polyhedron intersected with a convex norm-constraint. Therefore, our constructed sets have a more complex structure than other data-driven sets. Indeed we can show that most of the classical uncertainty classes are special cases of these sets and hence the robust problems is at least as hard as the one for classical uncertainty sets. Furthermore we could show that optimizing over our sets is strongly NP-hard. Nevertheless, to solve the robust optimization problem, it is necessary that we can optimize over the set of scenarios that are classified as being representative for the historic data by the neural network. We show that this is possible by formulating the adversarial problem as a convex quadratic mixed-integer program. By further decomposing this problem using the historic data, it is even possible to remove the binary variables from the program and solve a sequence of continuous problems instead.

We tested our method in two experiments, where we compare against a kernel-based support vector clustering method that is most similar to our approach. 
Throughout our experiments, we found encouraging results, observing that the method proposed in this paper often finds better robust solutions (with respect to both objective value and feasibility) than when applying the previous method or even discrete uncertainty sets containing the historic data. At the same time, our method is easy to apply with suitable training and optimization code being readily available. A drawback is that solution times on randomly generated data are higher than for the comparison method; here, the development of heuristic solution methods may be beneficial.

Since in the classical linear robust optimization regime the uncertainty set can always be replaced by its convex hull, the non-convexity (and even non-connectedness) of our uncertainty set is implicitly reverted and does not have an impact on the performance. Considering two-stage robust optimization problems, replacing the uncertainty set by its convex hull is not possible anymore, and therefore extending our method to two-stage robust optimization problems in the future can lead to larger improvements compared to other uncertainty sets.

\newcommand{\etalchar}[1]{$^{#1}$}

\end{document}